\documentclass{article}
\usepackage{amsmath,amssymb,amsthm}
\usepackage{graphicx}

\numberwithin{equation}{section}
\numberwithin{figure}{section}
\numberwithin{table}{section}
\theoremstyle{plain}
\newtheorem{theorem}{Theorem}[section]
\newtheorem{proposition}[theorem]{Proposition}
\newtheorem{lemma}[theorem]{Lemma}

\theoremstyle{definition}

\theoremstyle{remark}
\newtheorem{remark}{Remark}[section]
\begin{document}

\title{%
Recursive calculation of connection formulas
for systems of differential equations of Okubo normal form}

\author{%
Toshiaki Yokoyama}

\date{%
{\small
Department of Mathematics,
Chiba Institute of Technology,
Narashino 275-0023, Japan%
}}



\maketitle

\begin{abstract}
We study the structure of analytic continuation of
solutions of an even rank system of
linear ordinary differential equations of
Okubo normal form (ONF).
We develop an adjustment of the method
by using the Euler integral
for evaluating the connection formulas of
the Gauss hypergeometric function
${}_2F_1(\alpha, \beta, \gamma; x)$
to the system of ONF.
We obtain recursive relations
between connection coefficients for
the system of ONF and ones for
the underlying system of half rank.
\end{abstract}



\section{Introduction}
\label{sec:intro}
The Gauss hypergeometric differential equation
\begin{equation} \label{eq:gauss}
x(1-x)\frac{d^2u}{dx^2}
+\{c-(1+a+b)x\}\frac{du}{dx}
-abu=0
\end{equation}
is transformed into a system of first-order differential
equations of the form
\begin{displaymath}
\left(x I_2-\begin{pmatrix}0&\\&1\end{pmatrix}\right)
\frac{d}{dx}
\begin{pmatrix}u_1\\u_2\end{pmatrix}
=\begin{pmatrix}1-c & 1 \\
-(1-c+a)(1-c+b) & c-1-a-b\end{pmatrix}
\begin{pmatrix}u_1\\u_2\end{pmatrix},
\end{displaymath}
where
\begin{equation} \label{trans:gausstookubo}
u_1=u, \quad u_2=x\frac{du}{dx}-(1-c)u.
\end{equation}
As a generalization of this system,
Okubo (\cite{O1}, \cite{O2}) studies a rank  $n$ system
of linear differential equations of the form
\begin{equation} \label{sys:okubo}
(xI_n-T)\frac{du}{dx}=Au,
\end{equation}
where $u$ is an $n$-vector of unknown functions,
$A$ is an $n\times n$ constant matrix,
$T$ is an $n\times n$ diagonal constant matrix,
and $I_n$ denotes the $n\times n$ identity matrix.
We call (\ref{sys:okubo}) a system of
linear differential equations of
{\em Okubo normal form} (ONF).

In this paper we shall study the structure of
analytic continuation of solutions of
a rank $2n$ system of ONF of the form
\begin{equation} \label{sys:E2}
\begin{gathered}
\left(x I_{2n}-\mathcal{T}\right)
\frac{dU}{dx}
=\mathcal{A}_{P}U,
\\
\mathcal{T}=\begin{pmatrix}T&\\&t I_n\end{pmatrix},
\quad
\mathcal{A}_{P}=\begin{pmatrix}A & P \\
-(A'-\rho_{1}I_n)(A'-\rho_{2}I_n)P^{-1} & (\rho_{1}+\rho_{2})I_n-A'
\end{pmatrix},
\end{gathered}
\end{equation}
where $U=U(x)$ is a $2n$-vector of unknown functions,
$T,A\in \mathrm{M}(n;\mathbb{C})$, $P\in \mathrm{GL}(n;\mathbb{C})$
are matrices such that $T$ and $A'=P^{-1}AP$ are diagonal,
$t$ is a complex number satisfying $\det(t I_n-T)\ne0$,
and $\rho_j$ ($j=1,2$) are complex constants.

The Gauss hypergeometric equation
is intimately related to the Euler-Darboux equation.
Our system (\ref{sys:E2}) is obtained through the process
of making the relation of the two equations clear.
Set
\begin{displaymath}
E(\alpha,\beta,\gamma)
=(\xi-\eta)^2\frac{\partial^2}{\partial\xi\partial\eta}
-\alpha(\xi-\eta)\frac{\partial}{\partial\xi}
-\beta(\eta-\xi)\frac{\partial}{\partial\eta}
-\gamma.
\end{displaymath}
In the case that $\gamma=0$,
which is done by using the relation
\begin{equation} \label{rel:underlyingE}
E(\alpha,\beta,\gamma)(\xi-\eta)^{\delta}
=(\xi-\eta)^{\delta}E(\alpha+\delta,\beta+\delta,
\gamma+\delta(\delta+\alpha+\beta-1))
\end{equation}
with a suitable choice of the parameter $\delta$,
the partial differential equation
\begin{equation} \label{eq:eulerdarboux}
E(\alpha,\beta,0)f=0
\end{equation}
is called the Euler-Darboux equation.
Darboux (\cite[\S347]{D}) shows that
by setting
$f=\xi^{\lambda}\varphi(t)$,
$t=\eta/\xi$,
with a fixed constant $\lambda$,
the equation (\ref{eq:eulerdarboux}) is reduced to
the ordinary differential equation
\begin{displaymath}
t(1-t)\frac{d^2\varphi}{dt^2}
+\{1-\lambda-\beta-(1-\lambda+\alpha)t\}\frac{d\varphi}{dt}
+\lambda\alpha\varphi=0,
\end{displaymath}
which is the hypergeometric equation with the parameters
$a$, $b$, $c$ replaced by $\alpha$, $-\lambda$,
$1-\lambda-\beta$, respectively.
Set
\begin{displaymath}
L_2(\lambda)
=\xi\frac{\partial}{\partial\xi}
+\eta\frac{\partial}{\partial\eta}
-\lambda.
\end{displaymath}
Miller (\cite{M}) notices that
the space of particular solutions
of (\ref{eq:eulerdarboux}) of the form
$f=\xi^{\lambda}\varphi(\eta/\xi)$
coincides with
the space of solutions of
the system of partial differential equations
\begin{equation} \label{sys:miller}
\left\{
\begin{aligned}
E(\alpha,\beta,0)f &= 0,
\\
L_2(\lambda)f &= 0,
\end{aligned}\right.
\end{equation}
since the second equation implies
$f=\xi^{\lambda}\varphi(\eta/\xi)$.
Let us consider a $1$-dimensional section of this system
obtained by taking $\eta$ for a constant, and show that
the section is also reduced to the hypergeometric equation.
We first make the substitution
$f=(\xi-\eta)^{-\beta}g$.
Then, using the relation (\ref{rel:underlyingE})
and the relation
\begin{displaymath}
L_2(\lambda)(\xi-\eta)^{\delta}
=(\xi-\eta)^{\delta}L_2(\lambda-\delta),
\end{displaymath}
we obtain the system of partial differential equations
\begin{equation} \label{sys:millerdash}
\left\{
\begin{aligned}
E(\alpha-\beta,0,(1-\alpha)\beta)g &= 0,
\\
L_2(\lambda+\beta)g &= 0.
\end{aligned}\right.
\end{equation}
Eliminating the term of
$\partial^2 g/\partial\xi\partial\eta$
from this system by
\begin{equation} \label{form:eliminat}
(\xi-\eta)^2\frac{\partial}{\partial\xi} L_2(\lambda+\beta)g
-\eta E(\alpha-\beta,0,(1-\alpha)\beta)g,
\end{equation}
we obtain
\begin{equation} \label{eq:section}
\xi(\xi-\eta)^2\frac{\partial^2 g}{\partial\xi^2}
+\{(1-\lambda-\beta)\xi-(1-\lambda-\alpha)\eta\}
(\xi-\eta)\frac{\partial g}{\partial\xi}
+(1-\alpha)\beta\eta g=0.
\end{equation}
As a differential equation in one variable $\xi$,
this equation is Fuchsian and has
the Riemann scheme (the list of exponents)
\begin{displaymath}
\left\{\begin{array}{c@{\qquad}c@{\qquad}c}
\xi=0          & \xi=\infty     & \xi=\eta \\[\jot]
0              & 0              & 1-\alpha \\[\jot]
\lambda+\alpha & -\lambda-\beta & \beta
\end{array}\right\}.
\end{displaymath}
So, if we write
\begin{equation} \label{trans:changepar}
\alpha=1-a, \
\beta=b, \
\lambda=a-c
\qquad \text{or} \qquad
\alpha=1-b, \
\beta=a, \
\lambda=b-c
\end{equation}
and change the variable $\xi$ to $x$ by
\begin{equation} \label{trans:changevar}
x=1+\frac{\eta}{\xi-\eta}
\end{equation}
which carries the points $\xi=0,\infty,\eta$
into the points $x=0,1,\infty$, respectively,
then we can transform the equation (\ref{eq:section})
into the equation (\ref{eq:gauss}).
Note that by virtue of the substitution (\ref{trans:changepar}),
we can write the system (\ref{sys:millerdash})
in the form
\begin{equation} \label{sys:symmetricpareudar}
\left\{
\begin{aligned}
E(1-a-b,0,ab)g &= 0,
\\
L_2(a+b-c)g &= 0,
\end{aligned}\right.
\end{equation}
which is not altered if $a$ and $b$ are interchanged.

Darboux (\cite[\S354]{D}) furthermore shows that
an integral of the form
\begin{equation} \label{ftn:inteudar}
f(\xi,\eta)
=\int_{\eta}^{\xi}
(\xi-\zeta)^{-\beta}(\eta-\zeta)^{-\alpha} h(\zeta)\,d\zeta
\end{equation}
becomes a solution of the equation (\ref{eq:eulerdarboux})
under the condition $\Re\alpha<0$, $\Re\beta<0$.
Set
\begin{displaymath}
L_1(\lambda)
=\zeta\frac{d}{d\zeta}-\lambda.
\end{displaymath}
If $h(\zeta)$ is a solution of
the ordinary differential equation
\begin{equation} \label{eq:onevardash}
L_1(\lambda+\alpha+\beta-1)h=0,
\end{equation}
then the integral (\ref{ftn:inteudar}) satisfies
$L_2(\lambda)f=0$ and hence becomes
a solution of the system (\ref{sys:miller}).
Indeed, since
\begin{displaymath}
\begin{aligned}
f_{\xi}(\xi,\eta)
&=
\int_{\eta}^{\xi} (-\beta)
(\xi-\zeta)^{-\beta-1}(\eta-\zeta)^{-\alpha} h(\zeta)\,d\zeta,
\\
f_{\eta}(\xi,\eta)
&=
\int_{\eta}^{\xi} (-\alpha)
(\xi-\zeta)^{-\beta}(\eta-\zeta)^{-\alpha-1} h(\zeta)\,d\zeta,
\end{aligned}
\end{displaymath}
we have
\begin{equation} \label{pf:inteudartwo}
\begin{aligned}
&
{\xi f_{\xi}(\xi,\eta)+\eta f_{\eta}(\xi,\eta)
 +(\alpha+\beta)f(\xi,\eta)}
 \\
&=
\int_{\eta}^{\xi} \{
\beta(-\zeta)
(\xi-\zeta)^{-\beta-1}(\eta-\zeta)^{-\alpha}
+\alpha(-\zeta)
(\xi-\zeta)^{-\beta}(\eta-\zeta)^{-\alpha-1}
\} h(\zeta)\,d\zeta
\\
&=
-\int_{\eta}^{\xi}
\frac{\partial}{\partial \zeta}\{(\xi-\zeta)^{-\beta}(\eta-\zeta)^{-\alpha}\}
\zeta h(\zeta)\,d\zeta
\\
&=
\int_{\eta}^{\xi}
(\xi-\zeta)^{-\beta}(\eta-\zeta)^{-\alpha}
\frac{d}{d\zeta}\{\zeta h(\zeta)\}\,d\zeta
\\
&=
(\lambda+\alpha+\beta)f(\xi,\eta),
\end{aligned}
\end{equation}
which leads to $L_2(\lambda)f(\xi,\eta)=0$.
Here we have used
$\displaystyle
 \frac{d}{d\zeta}\{\zeta h(\zeta)\}
 =h(\zeta)+\zeta \frac{d}{d\zeta}h(\zeta)
 =(\lambda+\alpha+\beta)h(\zeta)$
in the last equality.

It should be noted
that the equation (\ref{eq:onevardash})
is a rank $1$ system of ONF
and the integral (\ref{ftn:inteudar}) is considered
as an extension in two variables of the Euler
transformation
\begin{displaymath}
f(\zeta)
=\int_{}^{\zeta}
(\zeta-\tau)^{-\alpha-\beta} h(\tau)\,d\tau
\end{displaymath}
which carries a solution of the equation (\ref{eq:onevardash})
into a solution of the equation $L_1(\lambda)f=0$.
In view of affinity of the system of ONF
for the Euler transformation
we may extend $L_1(\lambda)$, $L_2(\lambda)$ to rank $n$
differential operators of ONF as
\begin{displaymath}
\begin{aligned}
L_1(T',A;\rho)
&=
(\zeta I_n-T')\frac{d}{d\zeta}-\rho I_n-A,
\\
L_2(T',A;\rho)
&=
(\xi I_n-T')\frac{\partial}{\partial\xi}
+(\eta I_n-T')\frac{\partial}{\partial\eta}
-\rho I_n-A,
\end{aligned}
\end{displaymath}
where $A$ is an $n\times n$ constant matrix and
$T'$ is an $n\times n$ diagonal constant matrix.
For a solution $w(\zeta)$ of the system of ONF
\begin{equation} \label{sys:oneparokubo}
L_1(T',A;\rho+\alpha+\beta-1)w=0
\end{equation}
we define
\begin{equation} \label{ftn:ndiminteudar}
u(\xi,\eta)
=\int_{\eta}^{\xi}
(\xi-\zeta)^{-\beta}(\eta-\zeta)^{-\alpha} w(\zeta)\,d\zeta
\end{equation}
provided that $\Re\alpha<0$ and $\Re\beta<0$.
Then, similarly to the integral (\ref{ftn:inteudar}),
$u(\xi,\eta)$ becomes a solution of
the system of partial differential equations
\begin{equation} \label{sys:ndimmiller}
\left\{
\begin{aligned}
E(\alpha,\beta,0)u &= 0,
\\
L_2(T',A;\rho)u &= 0.
\end{aligned}\right.
\end{equation}
The operator $L_2(T',A;\rho)$ also satisfies the relation
\begin{displaymath}
L_2(T',A;\rho)(\xi-\eta)^{\delta}
=(\xi-\eta)^{\delta}L_2(T',A;\rho-\delta).
\end{displaymath}
So, similarly to (\ref{sys:millerdash})
and (\ref{sys:symmetricpareudar}),
setting
\begin{displaymath}
\alpha=1+\rho_{j'}, \quad
\beta=-\rho_{j}, \quad
\rho=-1-\rho_{j'}, \quad
\text{and} \quad
u=(\xi-\eta)^{\rho_{j}}v,
\end{displaymath}
where $j\in\{1,2\}$ and $j'$ denotes the complement of $j$
in $\{1,2\}$, namely, $(j,j')=(1,2)$ or $(2,1)$,
we can transform the system (\ref{sys:ndimmiller}) into the system
\begin{equation} \label{sys:yokext}
\left\{
\begin{aligned}
E(1+\rho_{1}+\rho_{2},0,\rho_{1}\rho_{2})v &= 0,
\\
L_2(T',A;-1-\rho_{1}-\rho_{2})v &= 0,
\end{aligned}\right.
\end{equation}
which is not altered if $\rho_{1}$ and $\rho_{2}$ are interchanged.
The quantities $\rho_{1}$, $\rho_{2}$ and $A$
in (\ref{sys:yokext})
correspond to $-a$, $-b$ and $1-c$
in (\ref{sys:symmetricpareudar}),
respectively.

Our system (\ref{sys:E2}) is obtained as
a $1$-dimensional section of the system (\ref{sys:yokext}).
The author introduced the system (\ref{sys:E2}) firstly
in \cite{Y1}, and then used it
to establish an algorithm for constructing
all irreducible semisimple systems
of ONF having rigid monodromy in \cite{Y2}.
Haraoka (\cite{Ha}) constructs integral representations
of solutions of the system (\ref{sys:E2})
using the integral equivalent to (\ref{ftn:ndiminteudar}),
and shows by following the author's algorithm
that solutions of
all irreducible semisimple systems of ONF
can be represented by the integral of Euler type.

In this paper we study connection problems between
local solutions of the system (\ref{sys:E2})
by using the integral
(\ref{ftn:ndiminteudar}).
We are especially concerned with
the relation between connection coefficients
of the system (\ref{sys:E2}) and ones
of the system (\ref{sys:oneparokubo}).


\section{Preliminary}
\label{sec:preliminary}
In this section we give transformations for obtaining
the system (\ref{sys:E2}) from the system (\ref{sys:yokext}),
and recall Haraoka's results on integral representations
of solutions of the system (\ref{sys:E2}).

\subsection{Transformation of equations}
Set
\begin{displaymath}
\begin{aligned}
M(T',A)&=
-(\xi-\eta)(\xi I_n-T')\frac{\partial}{\partial\xi}
-A(\eta I_n-T'),
\\
N(T',A;\rho)&=
(\xi-\eta)L_2(T',A;\rho)+M(T',A)
\\
&=
(\xi-\eta)(\eta I_n-T')\frac{\partial}{\partial\eta}
-\rho(\xi-\eta)I_n-A(\xi I_n-T'),
\end{aligned}
\end{displaymath}
and write
\begin{displaymath}
\begin{aligned}
E&=E(1+\rho_{1}+\rho_{2},0,\rho_{1}\rho_{2}),
\\
L_2&=L_2(T',A;-1-\rho_{1}-\rho_{2}),
\\
M&=M(T',A),
\\
N&=N(T',A;-1-\rho_{1}-\rho_{2})
\end{aligned}
\end{displaymath}
for short.

\begin{remark}
Applying the change of variable (\ref{trans:changevar})
to the differential operator defining $u_2$
in (\ref{trans:gausstookubo}),
we obtain
\begin{displaymath}
x\frac{d}{dx}-(1-c)
=\frac{1}{\eta}
\left(-(\xi-\eta)\xi\frac{d}{d\xi}-(1-c)\eta\right).
\end{displaymath}
The differential operator $M(T',A)$
corresponds to this operator.
\end{remark}

Similarly to (\ref{form:eliminat})
we have the following proposition.

\begin{proposition} \label{prop:separate}
We have
\begin{displaymath}
\begin{aligned}
(\xi-\eta)^2\frac{\partial}{\partial\xi}L_2-(\eta I_n-T')E
&=
-\left((\xi-\eta)\frac{\partial}{\partial\xi}
+(\rho_{1}+\rho_{2})I_n-A\right)M
+(A-\rho_{1}I_n)(A-\rho_{2}I_n)(\eta I_n-T'),
\\
\rlap{$\displaystyle
(\xi-\eta)^2\frac{\partial}{\partial\eta}L_2
-(\xi I_n-T')E-(\rho_{1}+\rho_{2}+1)(\xi-\eta)L_2$}
\phantom{
(\xi-\eta)^2\frac{\partial}{\partial\xi}L_2-(\eta I_n-T')E
}
\\
&=
\left((\xi-\eta)\frac{\partial}{\partial\eta}
-(\rho_{1}+\rho_{2})I_n+A\right)N
+(A-\rho_{1}I_n)(A-\rho_{2}I_n)(\xi I_n-T').
\end{aligned}
\end{displaymath}
\end{proposition}

\begin{proof}
By direct calculation.
\end{proof}

\begin{proposition} \label{prop:secondtofirst}
Set $v_1=(\eta I_n-T')v$ and $v_2=Mv$.
Then the system of partial differential equations
$(\ref{sys:yokext})$
is transformed into a system of first-order partial
differential equations for a $2n$-vector
$\begin{pmatrix}v_1\\v_2\end{pmatrix}$ of the form
\begin{equation} \label{sys:firstorder}
\left\{\begin{aligned}
(\xi-\eta)\frac{\partial v_1}{\partial\xi}
&=
-(\xi I_n-T')^{-1}(\eta I_n-T')Av_1
-(\xi I_n-T')^{-1}(\eta I_n-T')v_2,
\\
(\xi-\eta)\frac{\partial v_2}{\partial\xi}
&=
(A-\rho_{1}I_n)(A-\rho_{2}I_n)v_1
+(A-(\rho_{1}+\rho_{2})I_n)v_2,
\\
(\xi-\eta)\frac{\partial v_1}{\partial\eta}
&=
\{A(\xi I_n-T')-(\rho_{1}+\rho_{2})(\xi-\eta)I_n\}
(\eta I_n-T')^{-1}v_1
+v_2,
\\
(\xi-\eta)\frac{\partial v_2}{\partial\eta}
&=
-(A-\rho_{1}I_n)(A-\rho_{2}I_n)(\xi I_n-T')(\eta I_n-T')^{-1}v_1
-(A-(\rho_{1}+\rho_{2})I_n)v_2.
\end{aligned}\right.
\end{equation}
\end{proposition}

\begin{proof}
The system of partial differential equations
\begin{displaymath}
\left\{\begin{aligned}
&\left((\xi-\eta)^2\frac{\partial}{\partial\xi}L_2
-(\eta I_n-T')E\right)v=0,
\\
&\left((\xi-\eta)^2\frac{\partial}{\partial\eta}L_2
-(\xi I_n-T')E-(\rho_{1}+\rho_{2}+1)(\xi-\eta)L_2
\right)v=0,
\\
&L_2 v=0
\end{aligned}\right.
\end{displaymath}
is equivalent to
the system (\ref{sys:yokext}).
By virtue of Proposition \ref{prop:separate} and
the definition of the operators $M$ and $N$
we can write this system in the form
\begin{displaymath}
\left\{\begin{aligned}
&
\left\{\left((\xi-\eta)\frac{\partial}{\partial\xi}
+(\rho_{1}+\rho_{2})I_n-A\right)M
-(A-\rho_{1}I_n)(A-\rho_{2}I_n)(\eta I_n-T')\right\}v=0,
\\
&
\left\{\left((\xi-\eta)\frac{\partial}{\partial\eta}
-(\rho_{1}+\rho_{2})I_n+A\right)N
+(A-\rho_{1}I_n)(A-\rho_{2}I_n)(\xi I_n-T')\right\}v=0,
\\
&(M-N) v=0.
\end{aligned}\right.
\end{displaymath}
The first equation of this system and $Mv-v_2=0$ are equivalent to
the first two equations of the system (\ref{sys:firstorder}).
Besides,
the second equation and $Nv-v_2=0$ are equivalent to
the last two equations of the system (\ref{sys:firstorder}).
\end{proof}

Using the relations
\begin{displaymath}
\begin{aligned}
(\xi I_n-T')^{-1}(\eta I_n-T')
&=
(\xi-\eta)\left\{
\frac{1}{\xi-\eta}I_n-(\xi I_n-T')^{-1}\right\}, \\
(\xi I_n-T')(\eta I_n-T')^{-1}
&=
(\xi-\eta)\left\{
\frac{1}{\xi-\eta}I_n+(\eta I_n-T')^{-1}\right\}, \\
\end{aligned}
\end{displaymath}
we can write the system (\ref{sys:firstorder})
in a Pfaffian system of the form
\begin{equation} \label{sys:totalV}
dV=\left\{
\begin{pmatrix} (\xi I_n-T')^{-1} & \\ & O \end{pmatrix}
\mathcal{A}\,d\xi
+(\mathcal{A}-(\rho_{1}+\rho_{2})I_{2n})
\begin{pmatrix} (\eta I_n-T')^{-1} & \\ & O \end{pmatrix}
d\eta
-\mathcal{A}\,\frac{d(\xi-\eta)}{\xi-\eta}
\right\}V,
\end{equation}
where
\begin{displaymath}
V =\begin{pmatrix} v_1\\ v_2 \end{pmatrix}, \quad
\mathcal{A}
=\begin{pmatrix}
A & I_n \\
-(A-\rho_{1}I_n)(A-\rho_{2}I_n) & (\rho_{1}+\rho_{2})I_n-A
\end{pmatrix}.
\end{displaymath}
For a fixed constant $\eta_0$
the $2n$-vector function $V(\xi)=V(\xi, \eta_0)$ satisfies
\begin{equation} \label{sys:secoftotalV}
\frac{dV}{d\xi}
=\left(
\begin{pmatrix} (\xi I_n-T')^{-1} & \\ & O \end{pmatrix}
-\frac{1}{\xi-\eta_0}I_{2n}
\right) \mathcal{A} V.
\end{equation}

\begin{proposition}
Assume that $\eta_0$
satisfies $\det(T'-\eta_0 I_n)\ne0$.
The change of variables
\begin{equation} \label{change:xitox}
\xi=\eta_0+\frac{1}{x-t}
\quad\text{and}\quad
V=\begin{pmatrix}I_n & \\ & P\end{pmatrix}U,
\end{equation}
where $t$ is a constant,
transforms the system $(\ref{sys:secoftotalV})$
into the system of ONF
\begin{equation} \label{sys:secoftotalVtoONF}
\left(x I_{2n}
-\begin{pmatrix} T & \\ & t I_n \end{pmatrix}
\right)
\frac{dU}{dx}
=\mathcal{A}_{P} U,
\end{equation}
where
\begin{displaymath}
T=t I_n+(T'-\eta_0 I_n)^{-1}.
\end{displaymath}
\end{proposition}

\begin{proof}
The change of variable (\ref{change:xitox})
leads to
\begin{displaymath}
\begin{pmatrix} (\xi I_n-T')^{-1} & \\ & O \end{pmatrix}
-\frac{1}{\xi-\eta_0}I_{2n}
=-(x-t)^2
\left(x I_{2n}
-\begin{pmatrix} T & \\ & t I_n \end{pmatrix}
\right)^{-1}.
\end{displaymath}
Substituting this formula into the right hand side of
\begin{displaymath}
\begin{aligned}
\frac{dV}{dx}
&=
\frac{d\xi}{dx} \frac{dV}{d\xi}
\\
&=
-\frac{1}{(x-t)^2}
\left(
\begin{pmatrix} (\xi I_n-T')^{-1} & \\ & O \end{pmatrix}
-\frac{1}{\xi-\eta_0}I_{2n}
\right) \mathcal{A} V,
\end{aligned}
\end{displaymath}
and then substituting $V=\begin{pmatrix}I_n & \\ & P\end{pmatrix}U$,
we obtain the system (\ref{sys:secoftotalVtoONF}).
\end{proof}

\subsection{Integral representation of solutions}
We denote by
$w(\rho;\zeta)$
a solution of the system of ONF
\begin{equation} \label{sys:associatedONF}
L_1(T',A;\rho)w=0.
\end{equation}
Recall the notation $j'$ for $j=1,2$:
\begin{displaymath}
j'=\left\{\begin{aligned}
&2 && \text{for\ } j=1, \\
&1 && \text{for\ } j=2.
\end{aligned}\right.
\end{displaymath}

\begin{proposition} \label{prop:integral}
For $j=1,2$, let
$w(-\rho_{j}-1;\zeta)$
be a solution of the system
$(\ref{sys:associatedONF})$ with $\rho=-\rho_{j}-1$.
Then the integral
\begin{displaymath}
v(\xi,\eta)
=(\xi-\eta)^{-\rho_{j}}
\int_{C}(\xi-\zeta)^{\rho_{j}}
(\eta-\zeta)^{-\rho_{j'}-1}
w(-\rho_{j}-1;\zeta)
\,d\zeta
\end{displaymath}
is a solution of the system $(\ref{sys:yokext})$,
provided that
\begin{equation} \label{condition:intC}
\Bigl[
(\xi-\zeta)^{\rho_{j}}
(\eta-\zeta)^{-\rho_{j'}-1}
(\zeta I_n-T')
w(-\rho_{j}-1;\zeta)
\Bigr]_{C}=0.
\end{equation}
\end{proposition}

\begin{proof}
The change of variables $v=(\xi-\eta)^{-\rho_{j}}u$
leads to a system
\begin{displaymath}
\left\{
\begin{aligned}
E(\rho_{j'}+1,-\rho_{j},0)u &= 0,
\\
L_2(T',A;-\rho_{j'}-1)u &= 0.
\end{aligned}\right.
\end{displaymath}
It is trivial that the integral
\begin{displaymath}
u(\xi,\eta)
=
\int_{C}(\xi-\zeta)^{\rho_{j}}
(\eta-\zeta)^{-\rho_{j'}-1}
w(-\rho_{j}-1;\zeta)
\,d\zeta
\end{displaymath}
satisfies the first equation of this system.
Similarly to (\ref{pf:inteudartwo}), we have
\begin{displaymath}
\begin{aligned}
&(\xi I_n-T')u_{\xi}(\xi,\eta)
+(\eta I_n-T')u_{\eta}(\xi,\eta)
+(-\rho_{j}+\rho_{j'}+1)u(\xi,\eta)
\\
&=
-\int_{C}
\frac{\partial}{\partial\zeta}\{
(\xi-\zeta)^{\rho_{j}} (\eta-\zeta)^{-\rho_{j'}-1}
\}
(\zeta I_n - T') w(-\rho_{j}-1;\zeta)
\,d\zeta
\\
&=
\int_{C}
(\xi-\zeta)^{\rho_{j}} (\eta-\zeta)^{-\rho_{j'}-1}
\frac{d}{d\zeta}\{
(\zeta I_n - T') w(-\rho_{j}-1;\zeta)
\}
\,d\zeta
\\
&
\phantom{{}={}}
{}-\Bigl[
(\xi-\zeta)^{\rho_{j}} (\eta-\zeta)^{-\rho_{j'}-1}
(\zeta I_n-T') w(-\rho_{j}-1;\zeta)
\Bigr]_{C}
\\
&=
(A-\rho_{j} I_n)u(\xi,\eta),
\end{aligned}
\end{displaymath}
which leads to $L_2(T',A;-\rho_{j'}-1)u(\xi,\eta)=0$.
Here in the last equality
we have used (\ref{condition:intC})
and
$\displaystyle
\frac{d}{d\zeta}\{
(\zeta I_n - T') w(-\rho_{j}-1;\zeta)
\}
=(A-\rho_{j} I_n)w(-\rho_{j}-1;\zeta)$.
\end{proof}

Combining Propositions
\ref{prop:integral}
and \ref{prop:secondtofirst},
we see that for $j=1,2$, the integral
\begin{displaymath}
V(\xi,\eta)
=
\begin{pmatrix}
\eta I_n-T' \\
M(T',A)
\end{pmatrix}
\int_{C}
\left(\frac{\xi-\zeta}{\xi-\eta}\right)^{\rho_{j}}
(\eta-\zeta)^{-\rho_{j'}-1}
w(-\rho_{j}-1;\zeta)
\,d\zeta
\end{displaymath}
becomes a solution of the Pfaffian system (\ref{sys:totalV})
under the condition (\ref{condition:intC}).
Substituting $\eta=\eta_0$
and then $\displaystyle \xi=\eta_0+\frac{1}{x-t}$,
we obtain the following proposition.

\begin{proposition} \label{prop:2n-integral}
For $j=1,2$, let
$w(-\rho_{j}-1;\zeta)$
be a solution of the system
$(\ref{sys:associatedONF})$ with
$\rho=-\rho_{j}-1$.
Then the integral
\begin{equation} \label{ftn:intCforVxi}
\begin{aligned}
V(\xi)&=V(\xi,\eta_0)
\\
&=
\begin{pmatrix}
\eta_0 I_n-T' \\
M_{\eta_0}(T',A)
\end{pmatrix}
\int_{C}
\left(\frac{\xi-\zeta}{\xi-\eta_0}\right)^{\rho_{j}}
(\eta_0-\zeta)^{-\rho_{j'}-1}
w(-\rho_{j}-1;\zeta)
\,d\zeta,
\end{aligned}
\end{equation}
where $M_{\eta_0}(T',A)$ denotes the operator $M(T',A)$
with $\eta$ replaced by $\eta_0$,
becomes a solution of the system
$(\ref{sys:secoftotalV})$.
Moreover, the integral
\begin{equation} \label{ftn:intCforVx}
\begin{aligned}
U(x)&=
\begin{pmatrix}I_n & \\ & P^{-1}\end{pmatrix}
V(\eta_0+\frac{1}{x-t},\eta_0)
\\
&=
\begin{pmatrix}
\eta_0 I_n-T' \\
P^{-1}M_{x,\eta_0}(T',A)
\end{pmatrix}
\int_{C}
\left\{
1+(x-t)(\eta_0-\zeta)
\right\}^{\rho_{j}}
(\eta_0-\zeta)^{-\rho_{j'}-1}
w(-\rho_{j}-1;\zeta)
\,d\zeta,
\end{aligned}
\end{equation}
where
\begin{displaymath}
M_{x,\eta_0}(T',A)
=
\left\{
I_n+(x-t)(\eta_0 I_n-T')
\right\}
\frac{\partial}{\partial x}
-A(\eta_0 I_n-T'),
\end{displaymath}
becomes a solution of the system of ONF
$(\ref{sys:secoftotalVtoONF})$.
\end{proposition}

Even if the integrals (\ref{ftn:intCforVxi})
and (\ref{ftn:intCforVx}) are divergent,
they make sense in the sense of the finite part
of a divergent integral (see \cite[2.3.3]{IKSY}).
For example, in the case that $C$ is a segment
from $\eta_0$ to $\xi$ in (\ref{ftn:intCforVxi}),
the integral (\ref{ftn:intCforVxi}) is divergent
if $\Re\rho_{j}\leq -1$ or $\Re\rho_{j'}\geq 0$;
however, the integral makes sense in the sense of
the finite part if $\rho_{j}\not\in\mathbb{Z}_{<0}$
and $\rho_{j'}\not\in\mathbb{Z}_{\geq0}$.
We do not assume the conditions
for convergence of the integrals stated below,
and treat them as the finite part of a divergent
integral when they are divergent.


\section{Solutions characterized by local behavior}
\label{sec:localsol}
We reverse our way of investigation.
Namely, we start with the system (\ref{sys:E2})
and then determine the system (\ref{sys:associatedONF}).

\subsection{Assumptions}
Consider the system (\ref{sys:E2}). We write $T$ in the form
\begin{displaymath}
T=\begin{pmatrix}
t_1I_{n_1} & & &\\
& t_2I_{n_2} & &\\
& &\ddots &\\
& & & t_pI_{n_p}
\end{pmatrix}
\end{displaymath}
and set
\begin{displaymath}
t=t_{p+1},
\end{displaymath}
where $t_i\neq t_j$ ($1\leq i\neq j\leq p+1$) and
$n_1+n_2+\cdots+n_p=n$.
We assume that
\begin{displaymath}
\text{No three $t_i$ ($1\leq i\leq p+1$) lie on a straight line.}
\end{displaymath}
Renumbering the $t_i$'s if necessary,
we fix an assignment of the arguments
$\theta_i=\arg(t_i-t_{p+1})$ ($1\leq i\leq p$)
such that
\begin{displaymath}
\theta_1 <
\theta_2 <
\cdots <
\theta_p <
\theta_1+2\pi.
\end{displaymath}
Besides, we fix a real number $\theta_{p+1}$
such that
\begin{displaymath}
\theta_p <
\theta_{p+1} <
\theta_1+2\pi.
\end{displaymath}

When we write
\begin{displaymath}
A=\begin{pmatrix}
A_{11} & A_{12} & \cdots & A_{1p} \\
A_{21} & A_{22} & \cdots & A_{2p} \\
\vdots & \vdots & \ddots & \vdots \\
A_{p1} & A_{p2} & \cdots & A_{pp}
\end{pmatrix}
\end{displaymath}
in the same partition as $T$,
namely, $A_{ij}\in\mathrm{M}(n_i,n_j;\mathbb{C})$,
we may assume that the diagonal blocks
$A_{ii}$ ($1\leq i\leq p$) are of Jordan canonical form,
since a transformation of the form
$U=\begin{pmatrix}
    Q & \\
      & I_n
   \end{pmatrix}\bar U$,
where
\begin{displaymath}
Q=\begin{pmatrix}
Q_1 & & & \\
& Q_2 & & \\
& & \ddots & \\
& & & Q_p
\end{pmatrix},
\quad
Q_i\in\mathrm{GL}(n_i;\mathbb{C}),
\end{displaymath}
changes the system (\ref{sys:E2}) into the same system
with $A_{ij}$ and $P$ replaced by
${Q_i}^{-1}A_{ij}Q_j$ and $Q^{-1}P$, respectively.
We assume that
\begin{displaymath}
\text{$A_{11}$, $A_{22}$, $\dots$, $A_{pp}$
are diagonal,}
\end{displaymath}
and set
for $1\leq i\leq p$
\begin{displaymath}
A_{ii}=\begin{pmatrix}
\lambda_{i,1}I_{\ell_{i,1}} & & &\\
&\lambda_{i,2}I_{\ell_{i,2}} & &\\
& &\ddots &\\
& & &\lambda_{i,r_i}I_{\ell_{i,r_i}}
\end{pmatrix},
\end{displaymath}
where $\lambda_{i,k}\neq\lambda_{i,h}$ ($k\neq h$) and
$\ell_{i,1}+\ell_{i,2}+\cdots+\ell_{i,r_i}=n_i$.
We write $A'=P^{-1}AP$ in the form
\begin{displaymath}
A'=\begin{pmatrix}
\mu_1I_{m_1} & & &\\
&\mu_2I_{m_2} & &\\
& &\ddots &\\
& & &\mu_qI_{m_q}
\end{pmatrix},
\end{displaymath}
where $\mu_{k}\neq\mu_h$ ($k\neq h$) and
$m_1+m_2+\cdots+m_q=n$.
Throughout this paper
we assume that
\begin{align}
\label{assumption:E2_1}
\lambda_{i,k}, \
\lambda_{i,k}-\lambda_{i,h}
\not\in \mathbb{Z}
&\text{\quad for\ }
1\leq i\leq p,\
1\leq k\neq h\leq r_i,
\\
\label{assumption:E2_2}
\mu_{k}-\mu_{h}
\not\in \mathbb{Z}
&\text{\quad for\ }
1\leq k\neq h\leq q,
\\
\label{assumption:E2_3}
\rho_{j}, \
\rho_{1}-\rho_{2}
\not\in \mathbb{Z}
&\text{\quad for\ }
1\leq j\leq 2,
\\
\label{assumption:E2_4}
\rho_{j}-\lambda_{i,k}
\not\in \mathbb{Z}
&\text{\quad for\ }
1\leq j\leq 2,\
1\leq i\leq p,\
1\leq k\leq r_i.
\end{align}

\subsection{Nonholomorphic solutions near singular points: generic case}
First, we treat the case that none of the $\rho_{j}$'s
is an eigenvalue of the matrix $A$.
In this case we assume that
\begin{equation}
\label{assumption:E2_0}
\rho_{j}-\mu_{k}
\not\in \mathbb{Z}
\quad\text{for }
1\leq j\leq 2,\
1\leq k\leq q
\end{equation}
in addition to (\ref{assumption:E2_1})--(\ref{assumption:E2_4}).
The system (\ref{sys:E2}) is a Fuchsian system
with singularities
$x=t_1,\ldots,t_{p+1}$ and $\infty$.
Note that
\begin{displaymath}
\mathcal{A}_{P}
\sim
\begin{pmatrix}
\rho_{1} I_n &            \\
& \rho_{2} I_n
\end{pmatrix}.
\end{displaymath}
The Riemann scheme of the system (\ref{sys:E2}) is
\begin{displaymath}
\left\{\ \begin{aligned}
&x=t_1 && \cdots &&
x=t_p &&
x=t_{p+1} &&
x=\infty \\
&0\,(\bar{n}_1) && \cdots &&
0\,(\bar{n}_p) &&
0\,(n) &&
-\rho_{1}\,(n) \\
&\lambda_{1,1}\,(\ell_{1,1}) && \cdots &&
\lambda_{p,1}\,(\ell_{p,1}) &&
\rho_{1}+\rho_{2}-\mu_{1}\,(m_{1}) &&
-\rho_{2}\,(n) \\
&\,\quad\vdots && &&
\,\quad\vdots &&
\,\quad\vdots &&
\\
&\lambda_{1,r_1}\,(\ell_{1,r_1}) && \cdots &&
\lambda_{p,r_p}\,(\ell_{p,r_p}) &&
\rho_{1}+\rho_{2}-\mu_{q}\,(m_{q}) &&
\end{aligned}\ \right\},
\end{displaymath}
where $\lambda\,(\ell)$ denotes an exponent $\lambda$
with its multiplicity $\ell$,
and $\bar{n}_i=2n-n_i$ $(1\leq i\leq p)$.

Applying the general theory of local solutions near a regular
singular point (e.g.\ \cite[Chapter 1]{IKSY}),
we have the following theorems.
We use the notation
\begin{displaymath}
\varepsilon_m(k)
=\text{the $k$-th unit $m$-vector}.
\end{displaymath}

\begin{theorem}
For
$1\leq i\leq p$,
$1\leq k\leq r_i$,
$1\leq h\leq \ell_{i,k}$
there exists a unique solution
of the system $(\ref{sys:E2})$
of the form
\begin{displaymath}
U_{t_i,k,h}(x)
=(x-t_i)^{\lambda_{i,k}}\sum_{m=0}^{\infty}
G_{t_i,k,h}(m)(x-t_i)^{m}
\end{displaymath}
with
\begin{displaymath}
G_{t_i,k,h}(0)
=\varepsilon_{2n}
(n_1+\cdots+n_{i-1}+\ell_{i,1}+\cdots+\ell_{i,k-1}+h)
\end{displaymath}
which is convergent for
$|x-t_i|<\min_{1\leq k\leq p+1, k\ne i}|t_k-t_i|$.
Besides, for
$1\leq k\leq q$,
$1\leq h\leq m_{k}$
there exists a unique solution
of the system $(\ref{sys:E2})$
of the form
\begin{displaymath}
U_{t_{p+1},k,h}(x)
=(x-t_{p+1})^{\rho_{1}+\rho_{2}-\mu_{k}}\sum_{m=0}^{\infty}
G_{t_{p+1},k,h}(m)(x-t_{p+1})^{m}
\end{displaymath}
with
\begin{displaymath}
G_{t_{p+1},k,h}(0)
=\varepsilon_{2n}
(n+m_1+\cdots+m_{k-1}+h)
\end{displaymath}
which is convergent for
$|x-t_{p+1}|<\min_{1\leq k\leq p}|t_k-t_{p+1}|$.
\end{theorem}

\begin{theorem}
For
$1\leq j\leq 2$,
$1\leq h\leq n$
there exists a unique solution
of the system $(\ref{sys:E2})$
of the form
\begin{displaymath}
U_{\infty,j,h}(x)
=\left(\frac{1}{x-t_{p+1}}\right)^{-\rho_{j}}\sum_{m=0}^{\infty}
G_{\infty,j,h}(m)
\left(\frac{1}{x-t_{p+1}}\right)^m
\end{displaymath}
with
\begin{displaymath}
\begin{aligned}
G_{\infty,j,h}(0)
&=(\mathcal{A}_{P}-\rho_{j'}I_{2n})
\varepsilon_{2n}(n+h)
\\
&=\begin{pmatrix}P\\\rho_{j}I_n-A'\end{pmatrix}
\varepsilon_n(h)
\end{aligned}
\end{displaymath}
which is convergent for
$|x-t_{p+1}|>\max_{1\leq k\leq p}|t_k-t_{p+1}|$.
\end{theorem}

\subsection{Reducible case (i)}
Next, we treat the case that one of the $\rho_{j}$'s
is an eigenvalue of the matrix $A$.
Put
\begin{displaymath}
\rho_{2}=\mu_{q}
\end{displaymath}
and assume that
\begin{equation}
\label{assumption:E1_0}
\rho_{1}-\mu_{k}
\not\in \mathbb{Z}
\quad\text{for }
1\leq k\leq q
\end{equation}
in addition to (\ref{assumption:E2_1})--(\ref{assumption:E2_4}).
In this case the coefficient $\mathcal{A}_{P}$ has the form
\begin{displaymath}
\mathcal{A}_{P}
=
\left(\begin{array}{@{}c@{\ }|@{\ }c@{}}
\mathcal{A}'_{P} & \ast \\[\jot]
\hline\\[-3\jot]
O & \rho_{1} I_{m_{q}}
\end{array}\right),
\quad
\mathcal{A}'_{P}\in \mathrm{M}(2n-m_{q};\mathbb{C}),
\end{displaymath}
and the system (\ref{sys:E2}) has a solution of the form
$ 
U
=
\begin{pmatrix}
U' \\
0_{m_{q}}
\end{pmatrix}
$. 
Here $U'$ satisfies the system of ONF of rank $2n-m_{q}$
\begin{equation} \label{sys:E1}
\left(x I_{2n-m_{q}}-\mathcal{T}'\right)
\frac{dU'}{dx}
=\mathcal{A}'_{P}U',
\quad
\mathcal{T}'=\begin{pmatrix}T&\\&t I_{n-m_{q}}\end{pmatrix}.
\end{equation}
Note that
\begin{displaymath}
\mathcal{A}'_{P}
\sim
\begin{pmatrix}
\rho_{1} I_{n-m_{q}} &            \\
& \mu_{q} I_n
\end{pmatrix}.
\end{displaymath}
The Riemann scheme of the system (\ref{sys:E1}) is
\begin{displaymath}
\left\{\ \begin{aligned}
&x=t_1 && \cdots &&
x=t_p &&
x=t_{p+1} &&
x=\infty \\
&0\,(\bar{n}'_1) && \cdots &&
0\,(\bar{n}'_p) &&
0\,(n) &&
-\rho_{1}\,(n-m_{q}) \\
&\lambda_{1,1}\,(\ell_{1,1}) && \cdots &&
\lambda_{p,1}\,(\ell_{p,1}) &&
\rho_{1}+\mu_{q}-\mu_{1}\,(m_{1}) &&
-\mu_{q}\,(n) \\
&\,\quad\vdots && &&
\,\quad\vdots &&
\,\quad\vdots &&
\\
&\lambda_{1,r_1}\,(\ell_{1,r_1}) && \cdots &&
\lambda_{p,r_p}\,(\ell_{p,r_p}) &&
\rho_{1}+\mu_{q}-\mu_{q-1}\,(m_{q-1}) &&
\end{aligned}\ \right\},
\end{displaymath}
where $\bar{n}'_i=2n-m_{q}-n_i$ $(1\leq i\leq p)$.

\begin{theorem}
For
$1\leq i\leq p$,
$1\leq k\leq r_i$,
$1\leq h\leq \ell_{i,k}$
there exists a unique solution
of the system $(\ref{sys:E1})$
of the form
\begin{displaymath}
U'_{t_i,k,h}(x)
=(x-t_i)^{\lambda_{i,k}}\sum_{m=0}^{\infty}
G'_{t_i,k,h}(m)(x-t_i)^{m}
\end{displaymath}
with
\begin{displaymath}
G'_{t_i,k,h}(0)
=\varepsilon_{2n-m_{q}}
(n_1+\cdots+n_{i-1}+\ell_{i,1}+\cdots+\ell_{i,k-1}+h)
\end{displaymath}
which is convergent for
$|x-t_i|<\min_{1\leq k\leq p+1, k\ne i}|t_k-t_i|$.
Besides, for
$1\leq k\leq q-1$,
$1\leq h\leq m_{k}$
there exists a unique solution
of the system $(\ref{sys:E1})$
of the form
\begin{displaymath}
U'_{t_{p+1},k,h}(x)
=(x-t_{p+1})^{\rho_{1}+\mu_{q}-\mu_{k}}\sum_{m=0}^{\infty}
G'_{t_{p+1},k,h}(m)(x-t_{p+1})^{m}
\end{displaymath}
with
\begin{displaymath}
G'_{t_{p+1},k,h}(0)
=\varepsilon_{2n-m_{q}}
(n+m_1+\cdots+m_{k-1}+h)
\end{displaymath}
which is convergent for
$|x-t_{p+1}|<\min_{1\leq k\leq p}|t_k-t_{p+1}|$.
\end{theorem}

\begin{theorem}
For
$1\leq h\leq n-m_{q}$
there exists a unique solution
of the system $(\ref{sys:E1})$
of the form
\begin{displaymath}
U'_{\infty,1,h}(x)
=\left(\frac{1}{x-t_{p+1}}\right)^{-\rho_{1}}\sum_{m=0}^{\infty}
G'_{\infty,1,h}(m)
\left(\frac{1}{x-t_{p+1}}\right)^m
\end{displaymath}
with
\begin{displaymath}
G'_{\infty,1,h}(0)
=(\mathcal{A}'_{P}-\mu_{q}I_{2n-m_{q}})
\varepsilon_{2n-m_{q}}(n+h).
\end{displaymath}
Besides, for
$1\leq h\leq n$
there exists a unique solution
of the system $(\ref{sys:E1})$
of the form
\begin{displaymath}
U'_{\infty,\mu_{q},h}(x)
=\left(\frac{1}{x-t_{p+1}}\right)^{-\mu_{q}}\sum_{m=0}^{\infty}
G'_{\infty,\mu_{q},h}(m)
\left(\frac{1}{x-t_{p+1}}\right)^m
\end{displaymath}
with
\begin{displaymath}
G'_{\infty,\mu_{q},h}(0)
=
\left\{\begin{aligned}
&(\mathcal{A}'_{P}-\rho_{1}I_{2n-m_{q}}) \varepsilon_{2n-m_{q}}(n+h)
&& \text{for } 1\leq h\leq n-m_{q},
\\[\jot]
&\begin{pmatrix}P\varepsilon_n(h)\\0_{n-m_{q}}\end{pmatrix}
&& \text{for } n-m_{q}+1\leq h\leq n.
\end{aligned}\right.
\end{displaymath}
The series are convergent for
$|x-t_{p+1}|>\max_{1\leq k\leq p}|t_k-t_{p+1}|$.
\end{theorem}

\subsection{Reducible case (ii)}
Lastly, we treat the case that both of the $\rho_{j}$'s
are an eigenvalue of the matrix $A$.
Put
\begin{displaymath}
\rho_{1}=\mu_{q-1}
\quad\text{and}\quad
\rho_{2}=\mu_{q}
\end{displaymath}
with the conditions
(\ref{assumption:E2_1})--(\ref{assumption:E2_4}).
In this case the coefficient $\mathcal{A}_{P}$ has the form
\begin{displaymath}
\mathcal{A}_{P}
=
\left(\begin{array}{@{}c@{\ }|@{\ }c@{}}
\mathcal{A}''_{P} & \ast \\[\jot]
\hline\\[-3\jot]
O & \begin{matrix}\mu_{q} I_{m_{q-1}}&\\&\mu_{q-1} I_{m_{q}}\end{matrix}
\end{array}\right),
\quad
\mathcal{A}''_{P}\in \mathrm{M}(2n-m_{q-1}-m_{q};\mathbb{C}),
\end{displaymath}
and the system (\ref{sys:E2}) has a solution of the form
$ 
U
=
\begin{pmatrix}
U'' \\
0_{m_{q-1}+m_{q}}
\end{pmatrix}
$. 
Here $U''$ satisfies the system of ONF of rank $2n-m_{q-1}-m_{q}$
\begin{equation} \label{sys:E0}
\left(x I_{2n-m_{q-1}-m_{q}}-\mathcal{T}''\right)
\frac{dU''}{dx}
=\mathcal{A}''_{P}U'',
\quad
\mathcal{T}''=\begin{pmatrix}T&\\&t I_{n-m_{q-1}-m_{q}}\end{pmatrix}.
\end{equation}
Note that
\begin{displaymath}
\mathcal{A}''_{P}
\sim
\begin{pmatrix}
\mu_{q-1} I_{n-m_{q}} &            \\
& \mu_{q} I_{n-m_{q-1}}
\end{pmatrix}.
\end{displaymath}
The Riemann scheme of the system (\ref{sys:E0}) is
\begin{displaymath}
\left\{\ \begin{aligned}
&x=t_1 && \cdots &&
x=t_p &&
x=t_{p+1} &&
x=\infty \\
&0\,(\bar{n}''_1) && \cdots &&
0\,(\bar{n}''_p) &&
0\,(n) &&
-\mu_{q-1}\,(n-m_{q}) \\
&\lambda_{1,1}\,(\ell_{1,1}) && \cdots &&
\lambda_{p,1}\,(\ell_{p,1}) &&
\mu_{q-1}+\mu_{q}-\mu_{1}\,(m_{1}) &&
-\mu_{q}\,(n-m_{q-1}) \\
&\,\quad\vdots && &&
\,\quad\vdots &&
\,\quad\vdots &&
\\
&\lambda_{1,r_1}\,(\ell_{1,r_1}) && \cdots &&
\lambda_{p,r_p}\,(\ell_{p,r_p}) &&
\mu_{q-1}+\mu_{q}-\mu_{q-2}\,(m_{q-2}) &&
\end{aligned}\ \right\},
\end{displaymath}
where $\bar{n}''_i=2n-m_{q-1}-m_{q}-n_i$ $(1\leq i\leq p)$.

\begin{theorem}
For
$1\leq i\leq p$,
$1\leq k\leq r_i$,
$1\leq h\leq \ell_{i,k}$
there exists a unique solution
of the system $(\ref{sys:E0})$
of the form
\begin{displaymath}
U''_{t_i,k,h}(x)
=(x-t_i)^{\lambda_{i,k}}\sum_{m=0}^{\infty}
G''_{t_i,k,h}(m)(x-t_i)^{m}
\end{displaymath}
with
\begin{displaymath}
G''_{t_i,k,h}(0)
=\varepsilon_{2n-m_{q-1}-m_{q}}
(n_1+\cdots+n_{i-1}+\ell_{i,1}+\cdots+\ell_{i,k-1}+h)
\end{displaymath}
which is convergent for
$|x-t_i|<\min_{1\leq k\leq p+1, k\ne i}|t_k-t_i|$.
Besides, for
$1\leq k\leq q-2$,
$1\leq h\leq m_{k}$
there exists a unique solution
of the system $(\ref{sys:E0})$
of the form
\begin{displaymath}
U''_{t_{p+1},k,h}(x)
=(x-t_{p+1})^{\mu_{q-1}+\mu_{q}-\mu_{k}}\sum_{m=0}^{\infty}
G''_{t_{p+1},k,h}(m)(x-t_{p+1})^{m}
\end{displaymath}
with
\begin{displaymath}
G''_{t_{p+1},k,h}(0)
=\varepsilon_{2n-m_{q-1}-m_{q}}
(n+m_1+\cdots+m_{k-1}+h)
\end{displaymath}
which is convergent for
$|x-t_{p+1}|<\min_{1\leq k\leq p}|t_k-t_{p+1}|$.
\end{theorem}

\begin{theorem}
For
$1\leq h\leq n-m_{q}$
there exists a unique solution
of the system $(\ref{sys:E0})$
of the form
\begin{displaymath}
U''_{\infty,\mu_{q-1},h}(x)
=\left(\frac{1}{x-t_{p+1}}\right)^{-\mu_{q-1}}\sum_{m=0}^{\infty}
G''_{\infty,\mu_{q-1},h}(m)
\left(\frac{1}{x-t_{p+1}}\right)^m
\end{displaymath}
with
\begin{displaymath}
G''_{\infty,\mu_{q-1},h}(0)
=
\left\{\begin{aligned}
&(\mathcal{A}''_{P}-\mu_{q}I_{2n-m_{q-1}-m_{q}})
\varepsilon_{2n-m_{q-1}-m_{q}}(n+h)
&& \text{for } 1\leq h\leq n-m_{q-1}-m_{q},
\\[\jot]
&\begin{pmatrix}P\varepsilon_n(h)\\0_{n-m_{q-1}-m_{q}}\end{pmatrix}
&& \text{for } n-m_{q-1}-m_{q}+1\leq h\leq n-m_{q}.
\end{aligned}\right.
\end{displaymath}
Besides, for
$1\leq h\leq n-m_{q-1}-m_{q}$
and
$n-m_{q}+1\leq h\leq n$
there exists a unique solution
of the system $(\ref{sys:E0})$
of the form
\begin{displaymath}
U''_{\infty,\mu_{q},h}(x)
=\left(\frac{1}{x-t_{p+1}}\right)^{-\mu_{q}}\sum_{m=0}^{\infty}
G''_{\infty,\mu_{q},h}(m)
\left(\frac{1}{x-t_{p+1}}\right)^m
\end{displaymath}
with
\begin{displaymath}
G''_{\infty,\mu_{q},h}(0)
=
\left\{\begin{aligned}
&(\mathcal{A}''_{P}-\mu_{q-1}I_{2n-m_{q-1}-m_{q}})
\varepsilon_{2n-m_{q-1}-m_{q}}(n+h)
&& \text{for } 1\leq h\leq n-m_{q-1}-m_{q},
\\[\jot]
&\begin{pmatrix}P\varepsilon_n(h)\\0_{n-m_{q-1}-m_{q}}\end{pmatrix}
&& \text{for } n-m_{q}+1\leq h\leq n.
\end{aligned}\right.
\end{displaymath}
The series are convergent for
$|x-t_{p+1}|>\max_{1\leq k\leq p}|t_k-t_{p+1}|$.
\end{theorem}


\section{Solutions of the underlying system}
\label{sec:underlyingeq}
Let us fix a complex number $\eta_0$.
We define $t'_i$ ($1\leq i\leq p$) by
\begin{displaymath}
t'_i=\eta_0+\frac{1}{t_i-t_{p+1}},
\end{displaymath}
and define
\begin{equation} \label{mat:Tdash}
T'=\begin{pmatrix}
t'_1I_{n_1} & & &\\
& t'_2I_{n_2} & &\\
& &\ddots &\\
& & & t'_pI_{n_p}
\end{pmatrix}.
\end{equation}
Using this $T'$ and the matrix $A$ in the system (\ref{sys:E2}),
we set up the system of ONF
\begin{equation} \label{sys:underlying}
(\zeta I_n-T')\frac{dw}{d\zeta}=(\rho I_n+A)w,
\end{equation}
which we call
the {\em underlying system} associated with the system
(\ref{sys:E2}).

We define the assignment of the arguments
$\theta'_i=\arg(t'_i-\eta_0)$ ($1\leq i\leq p$) by
\begin{displaymath}
\begin{aligned}
\theta'_i&=- \arg(t_i-t_{p+1}) \\
&=- \theta_i.
\end{aligned}
\end{displaymath}
Moreover, we set
\begin{displaymath}
\theta'_{\infty}=- \theta_{p+1}.
\end{displaymath}
Note that the $\theta'_i$'s satisfy
\begin{displaymath}
\theta'_{\infty} <
\theta'_p <
\theta'_{p-1} <
\cdots <
\theta'_2 <
\theta'_1 <
\theta'_{\infty}+2\pi.
\end{displaymath}

In addition to the assumptions
(\ref{assumption:E2_1})--(\ref{assumption:E2_4})
we assume that
\begin{equation} \label{assumption:underlying}
\mu_{j}, \
\lambda_{i,k}-\mu_{j}
\not\in \mathbb{Z}
\quad \text{for }
1\leq j\leq q, \
1\leq i\leq p, \
1\leq k\leq r_i.
\end{equation}
Moreover, we assume
that the parameter $\rho$ satisfies
\begin{equation} \label{assumption:rho}
\rho+\lambda_{i,k}
\not\in \mathbb{Z}
\quad \text{for }
1\leq i\leq p, \
1\leq k\leq r_i.
\end{equation}

\subsection{Nonholomorphic solutions near singular points}
The Riemann scheme of the system (\ref{sys:underlying}) is
\begin{displaymath}
\left\{\ \begin{aligned}
&\zeta=t'_1 && \cdots &&
\zeta=t'_p &&
\zeta=\infty \\
&0\,(n-n_1) && \cdots &&
0\,(n-n_p) &&
-\rho-\mu_1\,(m_1) \\
&\rho+\lambda_{1,1}\,(\ell_{1,1}) && \cdots &&
\rho+\lambda_{p,1}\,(\ell_{p,1}) &&
-\rho-\mu_2\,(m_2) \\
&\,\quad\vdots && &&
\,\quad\vdots &&
\,\quad\vdots \\
&\rho+\lambda_{1,r_1}\,(\ell_{1,r_1}) && \cdots &&
\rho+\lambda_{p,r_p}\,(\ell_{p,r_p}) &&
-\rho-\mu_q\,(m_q)
\end{aligned}\ \right\}.
\end{displaymath}

\begin{theorem}
For
$1\leq i\leq p$,
$1\leq k\leq r_i$,
$1\leq h\leq \ell_{i,k}$,
there exists a unique solution
of the system $(\ref{sys:underlying})$
of the form
\begin{equation} \label{def:w_ti}
w_{t'_i,k,h}(\rho;\zeta)
=(\zeta-t'_i)^{\rho+\lambda_{i,k}}
\sum_{m=0}^{\infty}
\frac{\Gamma(\rho+\lambda_{i,k}+1)}{\Gamma(\rho+\lambda_{i,k}+1+m)}
g_{t'_i,k,h}(m)(\zeta-t'_i)^{m}
\end{equation}
with
\begin{displaymath}
g_{t'_i,k,h}(0)
=\varepsilon_{n}
(n_1+\cdots+n_{i-1}+\ell_{i,1}+\cdots+\ell_{i,k-1}+h),
\end{displaymath}
which is convergent for
$|\zeta-t'_i|<R'_i$,
$R'_i$ denoting $\min_{1\leq k\leq p, k\ne i}|t'_k-t'_i|$.
\end{theorem}

\begin{theorem}
For
$1\leq k\leq q$,
$1\leq h\leq m_k$,
there exists a unique solution
of the system $(\ref{sys:underlying})$
of the form
\begin{equation} \label{def:w_inf}
w_{\infty,k,h}(\rho;\zeta)
=\left(\frac{1}{\zeta-\eta_0}\right)^{-\rho-\mu_{k}}
\sum_{m=0}^{\infty}
\frac{\Gamma(-\rho-\mu_{k}+m)}{\Gamma(-\rho-\mu_{k})}
g_{\infty,k,h}(m)
\left(\frac{1}{\zeta-\eta_0}\right)^{m}
\end{equation}
with
\begin{displaymath}
g_{\infty,k,h}(0)
=P \varepsilon_n(m_1+\cdots+m_{k-1}+h),
\end{displaymath}
which is convergent for
$|\zeta-\eta_0|>R'_{\infty}$,
$R'_{\infty}$ denoting $\max_{1\leq k\leq p}|t'_k-\eta_0|$.
\end{theorem}

\begin{remark}
The coefficients $g_{t'_i,k,h}(m)$ and $g_{\infty,k,h}(m)$
for $m=0,1,2,\ldots$
do not depend on $\rho$.
\end{remark}

Set
\begin{displaymath}
\mathcal{P}'
=
\mathbb{C}\setminus
\bigcup_{i=1}^{p}\{\eta_0+(t'_i-\eta_0)s\mid 1\leq s< \infty\}.
\end{displaymath}
In $\mathcal{P}'$ we specify the branch of $w_{t'_i,k,h}(\rho;\zeta)$
by the assignment of argument 
\begin{equation} \label{branch:wt'ihk}
\arg(\zeta-t'_i)\in(\theta'_i-2\pi,\,\theta'_i)
\end{equation}
for $1\leq i\leq p$.

\begin{theorem} \label{thm:eulertr_z}
Assume that $\rho'-\rho\not\in\mathbb{Z}_{\geq0}$
in addition to $(\ref{assumption:rho})$ for $\rho$ and $\rho'$.
For $\zeta\in\mathcal{P}'$ we have
\begin{equation} \label{eutra:ti}
w_{t'_i,k,h}(\rho;\zeta)
=\frac{\Gamma(\rho+\lambda_{i,k}+1)}
{\Gamma(\rho-\rho')
\Gamma(\rho'+\lambda_{i,k}+1)}
\int_{t'_i}^{\zeta}
(\zeta-\tau)^{\rho-\rho'-1}
w_{t'_i,k,h}(\rho';\tau)
\,d\tau,
\end{equation}
where the path of integration is a segment or a curve
in $\mathcal{P}'$ with initial point $t'_i$ and terminal point $\zeta$,
and the branch of the integrand is determined by the following
assignment of the arguments
\begin{displaymath}
  \arg(\zeta-\tau)
 =\arg(\tau-t'_i)
 =\arg(\zeta-t'_i)
\end{displaymath}
for $\zeta$ sufficiently close to $t'_i$.
\end{theorem}

\begin{proof}
It suffices to prove (\ref{eutra:ti})
for $\zeta$ near $t'_i$.
Set
\begin{equation} \label{tmp:int_z}
u(\zeta)
=\int_{t'_i}^{\zeta}
(\zeta-\tau)^{\rho-\rho'-1}
w_{t'_i,k,h}(\rho';\tau)
\,d\tau.
\end{equation}
By the way similar to (\ref{pf:inteudartwo})
we can easily see that $u(\zeta)$ becomes a solution of
$(\ref{sys:underlying})$.
Substituting the expansion (\ref{def:w_ti}) with $\rho$
replaced by $\rho'$,
we have
\begin{displaymath}
u(\zeta)
=\sum_{m=0}^{\infty}
\frac{\Gamma(\rho'+\lambda_{i,k}+1)}{\Gamma(\rho'+\lambda_{i,k}+1+m)}
g_{t'_i,k,h}(m)
\int_{t'_i}^{\zeta}
(\zeta-\tau)^{\rho-\rho'-1}
(\tau-t'_i)^{\rho'+\lambda_{i,k}+m}
\,d{\tau}.
\end{displaymath}
Changing the variable of integration $\tau$
to $s$ by $\tau=t'_i+(\zeta-t'_i)s$,
we obtain
\begin{displaymath}
\begin{aligned}
\int_{t'_i}^{\zeta}
(\zeta-\tau)^{\rho-\rho'-1}
(\tau-t'_i)^{\rho'+\lambda_{i,k}+m}
\,d{\tau}
&=
(\zeta-t'_i)^{\rho+\lambda_{i,k}+m}
\int_{0}^{1}
(1-s)^{\rho-\rho'-1}
s^{\rho'+\lambda_{i,k}+m}
\,d{s}
\\
&=
(\zeta-t'_i)^{\rho+\lambda_{i,k}+m}
\frac{\Gamma(\rho-\rho')
\Gamma(\rho'+\lambda_{i,k}+m+1)}
{\Gamma(\rho+\lambda_{i,k}+m+1)}
\end{aligned}
\end{displaymath}
and hence
\begin{displaymath}
u(\zeta)
=
\frac{\Gamma(\rho-\rho')
\Gamma(\rho'+\lambda_{i,k}+1)}
{\Gamma(\rho+\lambda_{i,k}+1)}
\sum_{m=0}^{\infty}
\frac{\Gamma(\rho+\lambda_{i,k}+1)}{\Gamma(\rho+\lambda_{i,k}+1+m)}
g_{t'_i,k,h}(m)
(\zeta-t'_i)^{\rho+\lambda_{i,k}+m}.
\end{displaymath}
This implies (\ref{eutra:ti}).
\end{proof}

\subsection{Holomorphic solutions in the plane cut from one singular point}
Set
\begin{displaymath}
\mathcal{P}'_i
=\mathbb{C}\setminus
\{\eta_0+(t'_i-\eta_0)s\mid 1\leq s< \infty\}
\end{displaymath}
for $1\leq i\leq p$.

\begin{theorem} \label{thm:underlying_1cut}
For
$1\leq i\leq p$,
$1\leq k\leq r_i$,
$1\leq h\leq \ell_{i,k}$
there exists a unique solution
$\tilde{w}_{t'_i,k,h}(\rho;\zeta)$
of the system $(\ref{sys:underlying})$
such that
\begin{displaymath}
\begin{aligned}
&\tilde{w}_{t'_i,k,h}(\rho;\zeta)
\ \text{is holomorphic in}\ \mathcal{P}'_i,
\\
&\tilde{w}_{t'_i,k,h}(\rho;\zeta)
=w_{t'_i,k,h}(\rho;\zeta)
+\mathrm{hol}(\zeta-t'_i)
\quad \text{near}\ \zeta=t'_i.
\end{aligned}
\end{displaymath}
\end{theorem}

\begin{figure}
 \centering
 \includegraphics{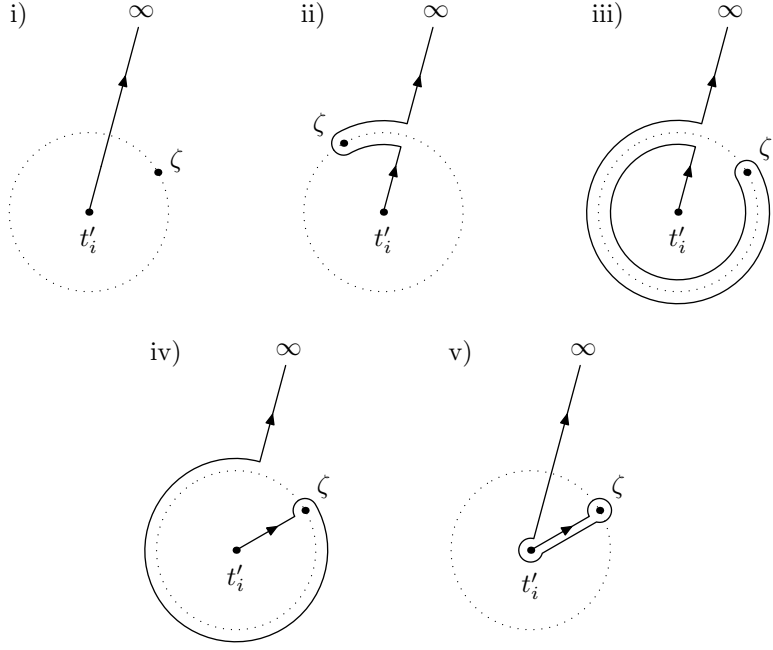}
 \caption{Deformation of the path $\overline{t'_i \infty}$}
 \label{fig:deformation_zeta}
\end{figure}

\begin{theorem}
Assume that $\rho+\mu_{k}\not\in\mathbb{Z}_{\geq0}$
for $1\leq k\leq q$.
Under the specification $(\ref{branch:wt'ihk})$ we have
\begin{equation} \label{eutra:t'i_inf}
\tilde{w}_{t'_i,k,h}(\rho;\zeta)
=-\frac{e^{-\pi \sqrt{-1}(\rho'+\lambda_{i,k})}
\Gamma(1-\rho+\rho')}
{\Gamma(-\rho-\lambda_{i,k})
\Gamma(\rho'+\lambda_{i,k}+1)}
\int_{t'_i}^{\infty}
(\zeta-\tau)^{\rho-\rho'-1}
w_{t'_i,k,h}(\rho';\tau)
\,d\tau
\end{equation}
for $\zeta \in \mathcal{P}'_i$,
where the path of integration is the ray from $t'_i$
to $\infty$ along the right-hand side of the cut
$\arg(\zeta-t'_i)=\theta'_i$,
and the branch of the integrand is determined
by the following assignment of the arguments
\begin{displaymath}
\left\{\begin{aligned}
&\arg(\tau-t'_i)=\theta'_i,
\\
&\arg(\zeta-\tau)\in[\arg(\zeta-t'_i),\,\theta'_i-\pi]^*,
\end{aligned}\right.
\end{displaymath}
$[\alpha,\,\beta]^*$ denoting the interval
$[\min(\alpha,\beta),\,\max(\alpha,\beta)]$.
\end{theorem}

\begin{proof}
Set
\begin{displaymath}
\tilde{u}(\zeta)
=\int_{t'_i}^{\infty}
(\zeta-\tau)^{\rho-\rho'-1}
w_{t'_i,k,h}(\rho';\tau)
\,d\tau.
\end{displaymath}
Similarly to the integral (\ref{tmp:int_z})
in the proof of Theorem \ref{thm:eulertr_z}
we see that $\tilde{u}(\zeta)$ becomes a solution of
$(\ref{sys:underlying})$.
It is trivial that $\tilde{u}(\zeta)$ is holomorphic in
$\mathcal{P}'_i$.
In order to find the behavior of $\tilde{u}(\zeta)$
near $\zeta=t'_i$ we shall analytically continue it
along a circle of center $t'_i$ with sufficiently small radius
in the positive direction.
The analytic continuation is found by deforming the path of
integration as $\zeta$ travels along the circle.
Figure \ref{fig:deformation_zeta} illustrates the deformation
of the path of integration.
From the last picture of Figure \ref{fig:deformation_zeta}
we see that the analytic continuation of $\tilde{u}(\zeta)$
leads to
\begin{displaymath}
\tilde{u}(\zeta)
+\left(e^{2\pi\sqrt{-1}(\rho-\rho')}-1\right)
e^{2\pi\sqrt{-1}(\rho'+\lambda_{i,k})}
u(\zeta),
\end{displaymath}
where $u(\zeta)$ denotes the integral
(\ref{tmp:int_z}) in the proof of Theorem \ref{thm:eulertr_z}.

On the other hand, when we analytically continue
$\tilde{w}_{t'_i,k,h}(\rho;\zeta)$
along the circle, we obtain
\begin{displaymath}
\tilde{w}_{t'_i,k,h}(\rho;\zeta)
+\left(e^{2\pi\sqrt{-1}(\rho+\lambda_{i,k})}-1\right)
w_{t'_i,k,h}(\rho;\zeta).
\end{displaymath}
So the solution defined by
\begin{displaymath}
\tilde{w}_{t'_i,k,h}(\rho;\zeta)
-
\frac{e^{2\pi\sqrt{-1}(\rho+\lambda_{i,k})}-1}
{(e^{2\pi\sqrt{-1}(\rho-\rho')}-1)
e^{2\pi\sqrt{-1}(\rho'+\lambda_{i,k})}}
\frac{\Gamma(\rho+\lambda_{i,k}+1)}
{\Gamma(\rho-\rho')\Gamma(\rho'+\lambda_{i,k}+1)}
\tilde{u}(\zeta)
\end{displaymath}
is not only holomorphic in $\mathcal{P}'_i$
but also single-valued near $\zeta=t'_i$,
and then must be $0$ under the assumption (\ref{assumption:rho}).
Lastly, using the formula of $\Gamma$-function
$\displaystyle
\Gamma(z)\Gamma(1-z)=\frac{\pi}{\sin \pi z}$,
we obtain (\ref{eutra:t'i_inf}).
\end{proof}

\subsection{Connection coefficients for the underlying system}
In this subsection we explain the dependence on $\rho$
of connection coefficients
for the system $(\ref{sys:underlying})$.
The following results are due to R.~Sch\"afke.
See his papers \cite{S1} and \cite{S2} in detail.

First, we consider the connection coefficients between
solutions near finite singular points.
In $\mathcal{P}'$,
under the specification $(\ref{branch:wt'ihk})$,
let us express $w_{t'_i,k,h}(\rho;\zeta)$
by a linear combination of solutions near $\zeta=t'_{\nu}$
($\nu\ne i$) as
\begin{equation} \label{conn_formula:t'i_t'nu}
w_{t'_i,k,h}(\rho;\zeta)
=\sum_{\tilde{k}=1}^{r_{\nu}} \sum_{\tilde{h}=1}^{\ell_{\nu,\tilde{k}}}
c_{t'_{\nu},\tilde{k},\tilde{h};t'_i,k,h}(\rho)
w_{t'_{\nu},\tilde{k},\tilde{h}}(\rho;\zeta)
+\mathrm{hol}(\zeta-t'_{\nu}).
\end{equation}

\begin{theorem}[{\cite[(3.6) Satz]{S2}}]
We have
\begin{equation} \label{conn_coef:t'i_t'nu}
c_{t'_{\nu},\tilde{k},\tilde{h};t'_i,k,h}(\rho)
=
\frac{e^{\pi\sqrt{-1}\rho}
      \Gamma(\rho+\lambda_{i,k}+1)
      \Gamma(-\rho-\lambda_{\nu,\tilde{k}})}
     {\Gamma(\lambda_{i,k}+1)
      \Gamma(-\lambda_{\nu,\tilde{k}})}
c_{t'_{\nu},\tilde{k},\tilde{h};t'_i,k,h}
\end{equation}
for $
1\leq i\ne \nu\leq q,\
1\leq k\leq r_i,\
1\leq h\leq \ell_{i,k},\
1\leq \tilde{k}\leq r_{\nu},\
1\leq \tilde{h}\leq \ell_{\nu,\tilde{k}}
$,
where
$c_{t'_{\nu},\tilde{k},\tilde{h};t'_i,k,h}
 =c_{t'_{\nu},\tilde{k},\tilde{h};t'_i,k,h}(0)$.
\end{theorem}

\begin{figure}
 \centering
 \includegraphics{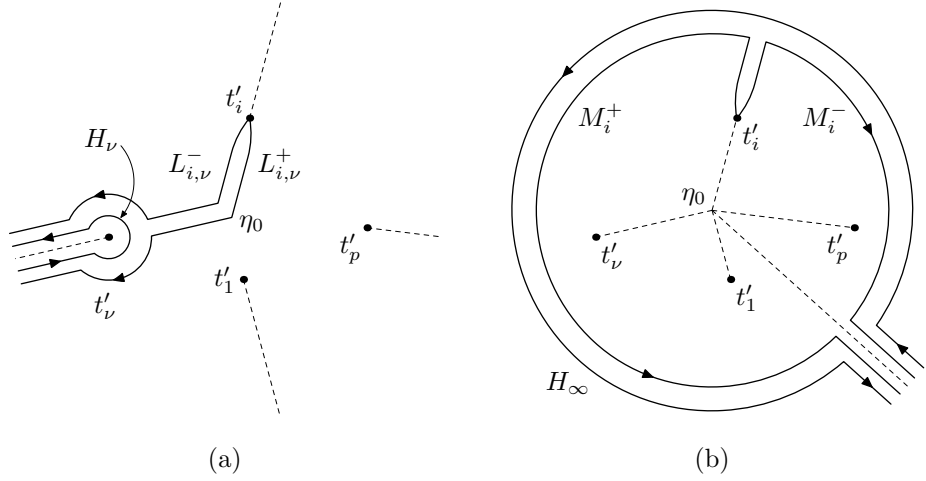}
 \caption{Paths of the Laplace transforms}
 \label{fig:Laplace_path}
\end{figure}

\begin{proof}
Let $L_{i,\nu}^{+}$
(resp.\ $L_{i,\nu}^{-}$)
be a curve in $\mathcal{P}'$ starting from $t'_i$
and going to $\infty$ along the left-hand
(resp.\ right-hand)
side of the cut $\arg(\zeta-t'_{\nu})=\theta'_{\nu}$,
and $H_{\nu}$
the Hankel loop in $\mathcal{P}'$
surrounding the cut $\arg(\zeta-t'_{\nu})=\theta'_{\nu}$
with sufficiently small radius
(see Figure \ref{fig:Laplace_path} (a)).
For $ |\arg z+\theta'_{\nu}|<\pi/2$
we define
\begin{displaymath}
\begin{aligned}
y_{t'_i,k,h}(L_{i,\nu}^{\pm};z)
&=
\frac{z^{\rho+1}}{\Gamma(\rho+\lambda_{i,k}+1)}
\int_{L_{i,\nu}^{\pm}}
e^{-z\zeta}
w_{t'_i,k,h}(\rho;\zeta)
\,d\zeta,
\\
y_{t'_{\nu},k,h}(H_{\nu};z)
&=
\frac{z^{\rho+1}}
     {\bigl(1-e^{-2\pi\sqrt{-1}(\rho+\lambda_{\nu,k})}\bigr)
      \Gamma(\rho+\lambda_{\nu,k}+1)}
\int_{H_{\nu}}
e^{-z\zeta}
w_{t'_{\nu},k,h}(\rho;\zeta)
\,d\zeta.
\end{aligned}
\end{displaymath}
It is well known (\cite{B}, \cite{I}, \cite{BJL}) that
these satisfy the system of Poincar\'e rank one
\begin{equation} \label{sys:poincarerankone}
z\frac{dy}{dz}=-\left(A+zT'\right)y,
\end{equation}
which does not depend on $\rho$.
Moreover, $y_{t'_i,k,h}(L_{i,\nu}^{\pm};z)$ and
$y_{t'_{\nu},k,h}(H_{\nu};z)$ themselves do not depend on $\rho$.
Indeed, substituting $(\ref{eutra:ti})$
and then reversing the order of integration,
we have
\begin{displaymath}
y_{t'_i,k,h}(L_{i,\nu}^{\pm};z)
=
\frac{z^{\rho+1}}{\Gamma(\rho-\rho')\Gamma(\rho'+\lambda_{i,k}+1)}
\int_{L_{i,\nu}^{\pm}}
\left\{
\int_{\smash{L_{i,\nu}^{\pm}(\tau)}}
e^{-z\zeta}
(\zeta-\tau)^{\rho-\rho'-1}
\,d\zeta
\right\}
w_{t'_i,k,h}(\rho';\tau)
\,d\tau,
\end{displaymath}
where $L_{i,\nu}^{\pm}(\tau)$ denotes the subpath of
$L_{i,\nu}^{\pm}$ from $\tau$ to $\infty$.
Since
\begin{displaymath}
\int_{L_{i,\nu}^{\pm}(\tau)}
e^{-z\zeta}
(\zeta-\tau)^{\rho-\rho'-1}
\,d\zeta
=
\Gamma(\rho-\rho')
z^{\rho'-\rho}
e^{-z \tau},
\end{displaymath}
we have
\begin{displaymath}
y_{t'_i,k,h}(L_{i,\nu}^{\pm};z)
=
\frac{z^{\rho'+1}}{\Gamma(\rho'+\lambda_{i,k}+1)}
\int_{L_{i,\nu}^{\pm}}
e^{-z\tau}
w_{t'_i,k,h}(\rho';\tau)
\,d\tau,
\end{displaymath}
which means that $y_{t'_i,k,h}(L_{i,\nu}^{\pm};z)$
does not depend on $\rho$.
As for $y_{t'_{\nu},k,h}(H_{\nu};z)$,
using the representation
\begin{displaymath}
y_{t'_{\nu},k,h}(H_{\nu};z)
=
\frac{z^{\rho+1}}
     {\Gamma(\rho+\lambda_{\nu,k}+1)}
\int_{t'_{\nu}}^{\infty e^{\sqrt{-1}\theta'_{\nu}}}
e^{-z\zeta}
w_{t'_{\nu},k,h}(\rho;\zeta)
\,d\zeta,
\end{displaymath}
we can prove its independence from $\rho$
similarly to $y_{t'_i,k,h}(L_{i,\nu}^{\pm};z)$.

Now let us consider the difference of
$y_{t'_i,k,h}(L_{i,\nu}^{+};z)$ and $y_{t'_i,k,h}(L_{i,\nu}^{-};z)$.
Substituting $(\ref{conn_formula:t'i_t'nu})$ into
\begin{displaymath}
y_{t'_i,k,h}(L_{i,\nu}^{-};z)
-
y_{t'_i,k,h}(L_{i,\nu}^{+};z)
=
\frac{z^{\rho+1}}{\Gamma(\rho+\lambda_{i,k}+1)}
\int_{H_{\nu}}
e^{-z\zeta}
w_{t'_i,k,h}(\rho;\zeta)
\,d\zeta,
\end{displaymath}
we have
\begin{displaymath}
y_{t'_i,k,h}(L_{i,\nu}^{-};z)
-
y_{t'_i,k,h}(L_{i,\nu}^{+};z)
=
\sum_{\tilde{k}=1}^{r_{\nu}} \sum_{\tilde{h}=1}^{\ell_{\nu,\tilde{k}}}
\frac{\bigl(1-e^{-2\pi\sqrt{-1}(\rho+\lambda_{\nu,\tilde{k}})}\bigr)
      \Gamma(\rho+\lambda_{\nu,\tilde{k}}+1)}
     {\Gamma(\rho+\lambda_{i,k}+1)}
c_{t'_{\nu},\tilde{k},\tilde{h};t'_i,k,h}(\rho)
y_{t'_{\nu},\tilde{k},\tilde{h}}(H_{\nu};z).
\end{displaymath}
Since the $y_{t'_{\nu},k,h}(H_{\nu};z)$'s
are linearly independent,
the coefficients must be uniquely determined and hence 
not depend on $\rho$.
The equality of the coefficients and those for $\rho=0$
leads to $(\ref{conn_coef:t'i_t'nu})$.
\end{proof}

Next, we consider the connection coefficients between
solutions near a finite singular point
and ones near infinity.
Set
\begin{displaymath}
\check{\mathcal{P}}'
=\mathbb{C}\setminus\left(\,
\bigcup_{i=1}^{p}\{\eta_0+(t'_i-\eta_0)s\mid 0\leq s\leq 1\}
\bigcup\{\eta_0+se^{\sqrt{-1}\theta'_{\infty}}\mid 0\leq s< \infty\}
\right).
\end{displaymath}
For $\zeta\in\check{\mathcal{P}}'$ we determine
the assignment of argument of $\zeta-\eta_0$ and
$\zeta-t'_i$ ($1\leq i\leq p$) as
\begin{displaymath}
\left\{
\begin{aligned}
\arg(\zeta-\eta_0)&\in (\theta'_{\infty},\,\theta'_{\infty}+2\pi),
\\
\arg(\rlap{$\zeta-t'_i)$}\phantom{\zeta-\eta_0)}
                  &\in (\theta'_i-\pi,\,\theta'_i+\pi)
\ \text{near}\ \zeta=t'_i.
\end{aligned}
\right.
\end{displaymath}
In $\check{\mathcal{P}}'$ let us express $w_{t'_i,k,h}(\rho;\zeta)$
by a linear combination of solutions near $\zeta=\infty$ as
\begin{equation} \label{conn_formula:t'i_infty}
w_{t'_i,k,h}(\rho;\zeta)
=\sum_{\tilde{k}=1}^{q} \sum_{\tilde{h}=1}^{m_{\tilde{k}}}
c_{\infty,\tilde{k},\tilde{h};t'_i,k,h}(\rho)
w_{\infty,\tilde{k},\tilde{h}}(\rho;\zeta).
\end{equation}

\begin{theorem} \label{thm:zero_conn.coef}
We have
\begin{equation} \label{conn_coef:t'i_infty-redcase}
c_{\infty,l,\tilde{h};t'_i,k,h}(-\mu_{l}-1)
=0
\end{equation}
for $
1\leq l\leq q,\
1\leq \tilde{h}\leq m_{l},\
1\leq i\leq p,\
1\leq k\leq r_i,\
1\leq h\leq \ell_{i,k}
$.
\end{theorem}

To prove this theorem we prepare a lemma.
We denote by $[\quad]_{m}$ the $m$-th row of a matrix
or the $m$-th component of a column vector.  

\begin{lemma}[{cf.~\cite[Proposition 3.1]{Ha}}] \label{lem:reducible}
Suppose that all of the residue matrices $A_i$ $(1\leq i\leq p)$
of a rank $n$ Fuchsian system
\begin{displaymath}
\frac{du}{d\zeta}
=\left(\sum_{i=1}^{p}\frac{1}{\zeta-t_i} A_i \right)u
\end{displaymath}
satisfy
\begin{displaymath}
[A_i]_{m}=0
\quad \text{for }n'+1\leq m\leq n'+n'',
\end{displaymath}
where $n'+n''\leq n$.
Then for $n'+1\leq m\leq n'+n''$ the $m$-th component of any solution
of this system is a constant.
Moreover, if the analytic continuation of a solution $u(\zeta)$
along a closed curve encircling a singular point
coincides with itself multiplied by a constant different from $1$,
then the solution $u(\zeta)$ satisfies
\begin{displaymath}
[u(\zeta)]_{m}=0
\quad \text{for }n'+1\leq m\leq n'+n''.
\end{displaymath}
\end{lemma}

\begin{proof}
Trivial.
\end{proof}

\begin{proof}[Proof of Theorem $\ref{thm:zero_conn.coef}$]
By direct calculation we see that
$u(\zeta)=P^{-1}(\zeta I_n-T')w_{t'_i,k,h}(-\mu_{l}-1;\zeta)$
satisfies the system
\begin{displaymath}
\frac{du}{d\zeta}
=
\left(\sum_{i=1}^{p}\frac{1}{\zeta-t'_i}
\left(A'-\mu_{l}I_n\right)B_i
\right)
u,
\end{displaymath}
where the matrices $B_i$ are defined by
$ 
P^{-1}(\zeta I_n-T')^{-1}P
=
\sum_{i=1}^{p}\frac{1}{\zeta-t'_i} B_i.
$ 
Hence by Lemma \ref{lem:reducible}
we have
\begin{displaymath}
\left[P^{-1}(\zeta I_n-T')w_{t'_i,k,h}(-\mu_{l}-1;\zeta)\right]_{m}=0
\end{displaymath}
for $m_1+\cdots+m_{l-1}+1\leq m\leq m_1+\cdots+m_{l}$.
Similarly, for $\tilde{k}\ne l$ we have
\begin{displaymath}
\left[P^{-1}(\zeta I_n-T')
w_{\infty,\tilde{k},\tilde{h}}(-\mu_{l}-1;\zeta)\right]_{m}=0
\end{displaymath}
for $m_1+\cdots+m_{l-1}+1\leq m\leq m_1+\cdots+m_{l}$.
On the other hand, we have
\begin{displaymath}
\left[P^{-1}(\zeta I_n-T')
w_{\infty,l,\tilde{h}}(-\mu_{l}-1;\zeta)\right]_{m}=1
\end{displaymath}
for $m=m_1+\cdots+m_{l-1}+\tilde{h}$.
These facts and the $(m_1+\cdots+m_{l-1}+\tilde{h})$-th
component of the relation
(\ref{conn_formula:t'i_infty}) with $\rho=-\mu_{l}-1$
multiplied by $P^{-1}(\zeta I_n-T')$
lead to (\ref{conn_coef:t'i_infty-redcase}).
\end{proof}

\begin{theorem}
Assume that $\rho+\mu_{k}\not\in\mathbb{Z}_{\geq0}$ for $1\leq k\leq q$.
We have
\begin{equation} \label{conn_coef:t'i_infty}
c_{\infty,\tilde{k},\tilde{h};t'_i,k,h}(\rho)
=
\frac{\Gamma(\rho+\lambda_{i,k}+1)
      \Gamma(\mu_{\tilde{k}}+1)}
     {\Gamma(\rho+\mu_{\tilde{k}}+1)
      \Gamma(\lambda_{i,k}+1)}
c_{\infty,\tilde{k},\tilde{h};t'_i,k,h}
\end{equation}
for $
1\leq \tilde{k}\leq q,\
1\leq \tilde{h}\leq m_{\tilde{k}},\
1\leq i\leq p,\
1\leq k\leq r_i,\
1\leq h\leq \ell_{i,k}
$,
where
$c_{\infty,\tilde{k},\tilde{h};t'_i,k,h}
 =c_{\infty,\tilde{k},\tilde{h};t'_i,k,h}(0)$.
\end{theorem}

\begin{proof}
Let $M_{i}^{-}$
(resp.\ $M_{i}^{+}$)
be a curve in $\check{\mathcal{P}}'$ starting from $t'_i$
and going to $\infty$ along the left-hand
(resp.\ right-hand)
side of the cut $\arg(\zeta-\eta_0)=\theta'_{\infty}$,
and $H_{\infty}$
the Hankel loop in $\check{\mathcal{P}}'$
surrounding the cut $\arg(\zeta-\eta_0)=\theta'_{\infty}$
with sufficiently large radius
(see Figure \ref{fig:Laplace_path} (b)).
For $ |\arg z+\theta'_{\infty}|<\pi/2$
we define
\begin{displaymath}
\begin{aligned}
y_{t'_i,k,h}(M_{i}^{-};z)
&=
\frac{z^{\rho+1}}{\Gamma(\rho+\lambda_{i,k}+1)}
\int_{M_{i}^{-}}
e^{-z\zeta}
w_{t'_i,k,h}(\rho;\zeta)
\,d\zeta,
\\
y_{t'_i,k,h}(M_{i}^{+};z)
&=
\frac{e^{-2\pi\sqrt{-1}\rho}z^{\rho+1}}
     {\Gamma(\rho+\lambda_{i,k}+1)}
\int_{M_{i}^{+}}
e^{-z\zeta}
w_{t'_i,k,h}(\rho;\zeta)
\,d\zeta,
\\
y_{\infty,k,h}(H_{\infty};z)
&=
\frac{\Gamma(-\rho-\mu_k)z^{\rho+1}}
     {2\pi\sqrt{-1} e^{\pi\sqrt{-1}(\rho+\mu_k+1)}}
\int_{H_{\infty}}
e^{-z\zeta}
w_{\infty,k,h}(\rho;\zeta)
\,d\zeta.
\end{aligned}
\end{displaymath}
These are not only solutions of the system
$(\ref{sys:poincarerankone})$
but also independent from $\rho$.
The independence from $\rho$ of $y_{t'_i,k,h}(M_{i}^{\mp};z)$
can be proved by the same way as
that of $y_{t'_i,k,h}(L_{i,\nu}^{\pm};z)$
thanks to
\begin{displaymath}
\begin{aligned}
\int_{M_{i}^{-}(\tau)}
e^{-z\zeta}
(\zeta-\tau)^{\rho-\rho'-1}
\,d\zeta
&=
\Gamma(\rho-\rho')
z^{\rho'-\rho}
e^{-z \tau},
\\
\int_{M_{i}^{+}(\tau)}
e^{-z\zeta}
(\zeta-\tau)^{\rho-\rho'-1}
\,d\zeta
&=
e^{2\pi\sqrt{-1}(\rho-\rho')}
\Gamma(\rho-\rho')
z^{\rho'-\rho}
e^{-z \tau},
\end{aligned}
\end{displaymath}
where $M_{i}^{\mp}(\tau)$ denotes the subpath of
$M_{i}^{\mp}$ from $\tau$ to $\infty$.
As for $y_{\infty,k,h}(H_{\infty};z)$,
substituting $(\ref{def:w_ti})$ and then integrating term by term,
we have
\begin{displaymath}
y_{\infty,k,h}(H_{\infty};z)
=
z^{\rho+1}
\sum_{m=0}^{\infty}
\frac{\Gamma(-\rho-\mu_k+m)}
     {2\pi\sqrt{-1} e^{\pi\sqrt{-1}(\rho+\mu_k+1)}}
 g_{\infty,k,h}(m)
\int_{H_{\infty}}
e^{-z\zeta}
 (\zeta-\eta_0)^{\rho+\mu_k-m}
\,d\zeta.
\end{displaymath}
Since
\begin{displaymath}
\begin{aligned}
\int_{H_{\infty}}
e^{-z\zeta}
 (\zeta-\eta_0)^{\rho+\mu_k-m}
\,d\zeta
&=
\left(e^{2\pi\sqrt{-1}(\rho+\mu_k)}-1\right)
\Gamma(\rho+\mu_k-m+1)
e^{-\eta_0 z}
z^{-\rho-\mu_k+m-1}
\\
&=
\frac{2\pi\sqrt{-1}e^{\pi\sqrt{-1}(\rho+\mu_k-m+1)}}
     {\Gamma(-\rho-\mu_k+m)}
e^{-\eta_0 z}
z^{-\rho-\mu_k+m-1},
\end{aligned}
\end{displaymath}
we obtain
\begin{displaymath}
y_{\infty,k,h}(H_{\infty};z)
=
e^{-\eta_0 z}
z^{-\mu_k}
\sum_{m=0}^{\infty}
(-1)^{m}
g_{\infty,k,h}(m)
z^{m},
\end{displaymath}
which does not depend on $\rho$.

Now let us consider the difference of
$e^{2\pi\sqrt{-1}\rho}y_{t'_i,k,h}(M_{i}^{+};z)$ and
$y_{t'_i,k,h}(M_{i}^{-};z)$.
Substituting $(\ref{conn_formula:t'i_infty})$ into
\begin{displaymath}
e^{2\pi\sqrt{-1}\rho}y_{t'_i,k,h}(M_{i}^{+};z)
-
y_{t'_i,k,h}(M_{i}^{-};z)
=
\frac{z^{\rho+1}}{\Gamma(\rho+\lambda_{i,k}+1)}
\int_{H_{\infty}}
e^{-z\zeta}
w_{t'_i,k,h}(\rho;\zeta)
\,d\zeta,
\end{displaymath}
we have
\begin{displaymath}
e^{2\pi\sqrt{-1}\rho}y_{t'_i,k,h}(M_{i}^{+};z)
-
y_{t'_i,k,h}(M_{i}^{-};z)
=
\sum_{\tilde{k}=1}^{q} \sum_{\tilde{h}=1}^{m_{\tilde{k}}}
\frac{2\pi\sqrt{-1} e^{\pi\sqrt{-1}(\rho+\mu_{\tilde{k}}+1)}}
     {\Gamma(-\rho-\mu_{\tilde{k}})\Gamma(\rho+\lambda_{i,k}+1)}
c_{\infty,\tilde{k},\tilde{h};t'_i,k,h}(\rho)
y_{\infty,\tilde{k},\tilde{h}}(H_{\infty};z).
\end{displaymath}
For any $l$, $1\leq l\leq q$,
replacing $\rho$ by $-\mu_{l}-1$ in this formula
and subtracting the resulting formula from this formula, we find
\begin{displaymath}
\begin{aligned}
\left(e^{2\pi\sqrt{-1}\rho}-e^{-2\pi\sqrt{-1}\mu_{l}}\right)
y_{t'_i,k,h}(M_{i}^{+};z)
&=
\sum_{\tilde{h}=1}^{m_{l}}
\frac{2\pi\sqrt{-1} e^{\pi\sqrt{-1}(\rho+\mu_{l}+1)}}
     {\Gamma(-\rho-\mu_{l})\Gamma(\rho+\lambda_{i,k}+1)}
c_{\infty,l,\tilde{h};t'_i,k,h}(\rho)
y_{\infty,l,\tilde{h}}(H_{\infty};z)
\\
&\phantom{{}={}}
+ \text{a linear combination of
$y_{\infty,\tilde{k},\tilde{h}}(H_{\infty};z)$ for $\tilde{k}\ne l$}
\end{aligned}
\end{displaymath}
by virtue of (\ref{conn_coef:t'i_infty-redcase})
in Theorem \ref{thm:zero_conn.coef}.
Since the $y_{\infty,k,h}(H_{\infty};z)$'s
are linearly independent, the quantity
\begin{displaymath}
\frac{2\pi\sqrt{-1} e^{\pi\sqrt{-1}(\rho+\mu_{l}+1)}}
     {\bigl(e^{2\pi\sqrt{-1}\rho}-e^{-2\pi\sqrt{-1}\mu_{l}}\bigr)
      \Gamma(-\rho-\mu_{l})\Gamma(\rho+\lambda_{i,k}+1)}
c_{\infty,l,\tilde{h};t'_i,k,h}(\rho)
=
\frac{e^{2\pi\sqrt{-1}\mu_{l}}\Gamma(\rho+\mu_{l}+1)}
     {\Gamma(\rho+\lambda_{i,k}+1)}
c_{\infty,l,\tilde{h};t'_i,k,h}(\rho)
\end{displaymath}
must be uniquely determined
and hence not depend on $\rho$,
which leads to $(\ref{conn_coef:t'i_infty})$.
\end{proof}


\section{Integrals associated with the solutions of the underlying system}
\label{sec:integrals}
For $1\leq i\leq p$ we set
\begin{displaymath}
\begin{aligned}
\phi^{-}_i
&=\max(\theta'_{i+1}+\delta,\theta'_i-\pi),
\\
\phi^{+}_i
&=\min(\theta'_{i-1}-\delta,\theta'_i+\pi),
\end{aligned}
\end{displaymath}
where $\theta'_0=\theta'_{\infty}+2\pi$,
$\theta'_{p+1}=\theta'_{\infty}$,
and $\delta$ is a sufficiently small positive number,
and set
\begin{displaymath}
\begin{aligned}
\mathcal{S}^{\prime-}_i
&=\{\zeta\mid\arg(\zeta-\eta_0)\in(\phi^{-}_i,\,\theta'_i)\},
\\
\mathcal{S}^{\prime+}_i
&=\{\zeta\mid\arg(\zeta-\eta_0)\in(\theta'_i,\,\phi^{+}_i)\}.
\end{aligned}
\end{displaymath}

\begin{figure}
 \centering
 \includegraphics{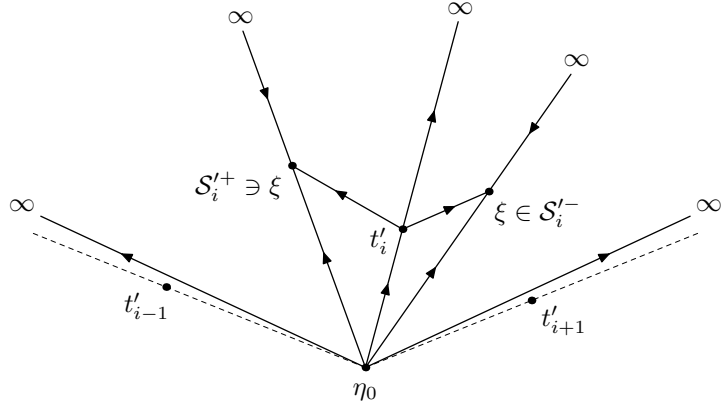}
 \caption{Paths of integration}
 \label{fig:path}
\end{figure}

For $\xi\in\mathcal{S}^{\prime\pm}_i$
we consider the integrals
\begin{equation} \label{def:W_ab_general}
W^{\langle\mathcal{S}^{\prime\pm}_i;ab\rangle}
_{\mathit{sing},k,h}(\nu_{1},\nu_{2};\xi)
=
\begin{pmatrix}
\eta_0 I_n-T' \\
P^{-1}M_{\eta_0}(T',A)
\end{pmatrix}
\int_{a}^{b}
\left(\frac{\xi-\zeta}{\xi-\eta_0}\right)^{\nu_{1}}
(\zeta-\eta_0)^{-\nu_{2}-1}
w_{\mathit{sing},k,h}(-\nu_{1}-1;\zeta)
\,d\zeta,
\end{equation}
where
$\nu_{1}, \nu_{2}$ are complex parameters,
$a,b \in \{\xi, \eta_0, t'_i, \infty\}$, $a\ne b$,
$\mathit{sing}=t'_i$ or $\infty$.
The path of integration for
$W^{\langle\mathcal{S}^{\prime\pm}_i;ab\rangle}_{\mathit{sing},k,h}$
is the segment 
or the ray $\overline{ab}$ from $a$ to $b$
indicated in Figure \ref{fig:path}.
Among the twelve possible choice of combinations
of $\mathit{sing}$ and $\overline{ab}$
we investigate the following seven cases:
\begin{displaymath}
\begin{aligned}
&W^{\langle\mathcal{S}^{\prime\pm}_i;\eta_0\xi\rangle}_{t'_i,k,h},&&
W^{\langle\mathcal{S}^{\prime\pm}_i;t'_i\xi\rangle}_{t'_i,k,h},&&
W^{\langle\mathcal{S}^{\prime\pm}_i;\eta_0t'_i\rangle}_{t'_i,k,h},&&
W^{\langle\mathcal{S}^{\prime\pm}_i;t'_i\infty\rangle}_{t'_i,k,h},
\\
&W^{\langle\mathcal{S}^{\prime\pm}_i;\eta_0\xi\rangle}_{\infty,k,h},&&
W^{\langle\mathcal{S}^{\prime\pm}_i;\infty\xi\rangle}_{\infty,k,h},&&
W^{\langle\mathcal{S}^{\prime\pm}_i;\eta_0\infty\rangle}_{\infty,k,h}. &&
\end{aligned}
\end{displaymath}
For $\xi\in\mathcal{S}^{\prime\pm}_i$
we determine the assignment of argument of
$\xi-t'_i$ as
\begin{equation} \label{arg:xi-t'i}
\begin{aligned}
  &\arg(\xi-t'_i)\in(\theta'_i-\pi,\,\theta'_i)
  && \text{for\ }\xi\in\mathcal{S}^{\prime-}_i,
  \\
  &\arg(\xi-t'_i)\in(\theta'_i-2\pi,\,\theta'_i-\pi)
  && \text{for\ }\xi\in\mathcal{S}^{\prime+}_i.
  \end{aligned}
\end{equation}
The branches of integrand of
$W^{\langle\mathcal{S}^{\prime\pm}_i;ab\rangle}_{\mathit{sing},k,h}$
along $\overline{ab}$ are determined by the assignment of
$\arg(\xi-\zeta)$,
$\arg(\zeta-\eta_0)$ and
$\arg(\zeta-t'_i)$
tabulated in Table \ref{tab:arg_t'i_minus} or \ref{tab:arg_t'i_plus}.
Similarly to the integral (\ref{ftn:intCforVxi})
in Proposition \ref{prop:2n-integral},
the integrals (\ref{def:W_ab_general}) satisfy
the system
\begin{equation} \label{sys:forW}
\frac{dW}{d\xi}
=
\left(
\begin{pmatrix} (\xi I_n-T')^{-1} & \\ & O \end{pmatrix}
-\frac{1}{\xi-\eta_0}I_{2n}
\right)
\begin{pmatrix}
A & P \\
-(A'-\nu_{1}I_n)(A'-\nu_{2}I_n)P^{-1} & (\nu_{1}+\nu_{2})I_n-A'
\end{pmatrix}
W.
\end{equation}

\begin{table}
\caption{Assignment of branches in case $\xi\in\mathcal{S}^{\prime-}_i$}
\begin{center}
$\begin{array}{c@{\qquad}c@{\qquad}c@{\qquad}c}
\hline
 & & & \\[-3\jot]
 & \arg(\xi-\zeta) & \arg(\zeta-\eta_0)
                   & \arg(\zeta-t'_i) \\[\jot]
\hline\\[-2\jot]
\overline{\eta_0 \xi}
 & \psi & \psi & [\theta'_i-\pi,\chi_i] \\[2\jot]
\overline{t'_i \xi}
 & \chi_i & [\psi,\theta'_i] & \chi_i \\[2\jot]
\overline{\eta_0 t'_i}
 & [\chi_i,\psi] & \theta'_i & \theta'_i-\pi \\[2\jot]
\overline{t'_i \infty}
 & (\theta'_i-\pi,\chi_i] & \theta'_i & \theta'_i \\[2\jot]
\overline{\infty \xi}
 & \psi+\pi & \psi &  \\[2\jot]
\overline{\eta_0 \infty}
 & [\psi,\phi_i^{-}+\pi) & \phi_i^{-} &  \\[2\jot]
\hline
\end{array}$ \\[.75\baselineskip]
$\begin{array}{r@{\;}c@{\;}l@{\;}l}
\psi &=& \arg(\xi-\eta_0) & \in (\phi^{-}_i,\,\theta'_i), \\[2\jot]
\chi_i &=& \arg(\xi-t'_i) & \in (\theta'_i-\pi,\,\theta'_i).
\end{array}$
\end{center}
\label{tab:arg_t'i_minus}
\end{table}

\begin{table}
\caption{Assignment of branches in case $\xi\in\mathcal{S}^{\prime+}_i$}
\begin{center}
$\begin{array}{c@{\qquad}c@{\qquad}c@{\qquad}c}
\hline
 & & & \\[-3\jot]
 & \arg(\xi-\zeta) & \arg(\zeta-\eta_0)
                   & \arg(\zeta-t'_i) \\[\jot]
\hline\\[-2\jot]
\overline{\eta_0 \xi}
 & \psi & \psi & [\chi_i,\theta'_i-\pi]  \\[2\jot]
\overline{t'_i \xi}
 & \chi_i & [\theta'_i,\psi] & \chi_i \\[2\jot]
\overline{\eta_0 t'_i}
 & [\psi-2\pi,\chi_i] & \theta'_i & \theta'_i-\pi \\[2\jot]
\overline{t'_i \infty}
 & [\chi_i,\theta'_i-\pi) & \theta'_i & \theta'_i \\[2\jot]
\overline{\infty \xi}
 & \psi-\pi & \psi &  \\[2\jot]
\overline{\eta_0 \infty}
 & (\phi_i^{+}-\pi,\psi] & \phi_i^{+} &  \\[2\jot]
\hline
\end{array}$ \\[.75\baselineskip]
$\begin{array}{r@{\;}c@{\;}l@{\;}l}
\psi &=& \arg(\xi-\eta_0) & \in (\theta'_i,\,\phi^{+}_i), \\[2\jot]
\chi_i &=& \arg(\xi-t'_i) & \in (\theta'_i-2\pi,\,\theta'_i-\pi).
\end{array}$
\end{center}
\label{tab:arg_t'i_plus}
\end{table}

\subsection{Integrals associated with the solutions
near finite singular points}
We use the notation
\begin{displaymath}
\Sigma^{J}_i=\{\eta_0+(t'_i-\eta_0)s\mid s\in J\}
\end{displaymath}
for $J=[0,1]$, $[1,\infty)$, and so on.


\begin{theorem} \label{thm:asymp_W_eta0_xi}
Assume that $\nu_{1}\notin \mathbb{Z}_{<0}$
and $\nu_{2}\notin \mathbb{Z}_{\geq0}$.
The analytic continuation of the integral
$W^{\langle\mathcal{S}^{\prime-}_i;\eta_0\xi\rangle}
_{t'_i,k,h}(\nu_{1},\nu_{2};\xi)$
across the open segment
$\Sigma^{(0,1)}_i$
into $\mathcal{S}^{\prime+}_i$
coincides with the integral
$W^{\langle\mathcal{S}^{\prime+}_i;\eta_0\xi\rangle}
_{t'_i,k,h}(\nu_{1},\nu_{2};\xi)$.
Moreover, as $\xi\longrightarrow\eta_0$,
$\xi\in
\mathcal{S}^{\prime-}_i\cup\Sigma^{(0,1)}_i
                       \cup\mathcal{S}^{\prime+}_i$,
we have
\begin{equation} \label{asymp:W_eta0xi_t'i}
\begin{aligned}
&
W^{\langle\mathcal{S}^{\prime\pm}_i;\eta_0\xi\rangle}
_{t'_i,k,h}(\nu_{1},\nu_{2};\xi)
\\
&=
(\xi-\eta_0)^{-\nu_{2}}
\left\{
\frac{\Gamma(\nu_{1}+1)\Gamma(-\nu_{2})}
{\Gamma(\nu_{1}-\nu_{2}+1)}
\begin{pmatrix}
P \\
\nu_{2}I_n-A'
\end{pmatrix}
P^{-1}(\eta_0 I_n-T')
w_{t'_i,k,h}(-\nu_{1}-1;0)
+O\left(\xi-\eta_0\right)
\right\}.
\end{aligned}
\end{equation}
\end{theorem}

\begin{proof}
The equality of the analytic continuation of
$W^{\langle\mathcal{S}^{\prime-}_i;\eta_0\xi\rangle}_{t'_i,k,h}$
and the integral
$W^{\langle\mathcal{S}^{\prime+}_i;\eta_0\xi\rangle}_{t'_i,k,h}$
follows from the continuity of the arguments
$\arg(\xi-\zeta)$,
$\arg(\zeta-\eta_0)$
and
$\arg(\zeta-t'_i)$
in $\mathcal{S}^{\prime-}_i\cup\Sigma^{(0,1)}_i
                       \cup\mathcal{S}^{\prime+}_i$,
where we define
$\arg(\xi-\zeta)=\theta'_i$,
$\arg(\zeta-\eta_0)=\theta'_i$
and
$\arg(\zeta-t'_i)=\theta'_i-\pi$
for $\xi\in\Sigma^{(0,1)}_i$.

Note that
$w_{t'_i,k,h}(-\nu_{1}-1;\zeta)$
is holomorphic for
$\zeta\in\mathbb{C}\setminus
\bigcup_{\ell=1}^{p}\Sigma^{[1,\infty)}_{\ell}$,
where we specify $\theta'_i-2\pi<\arg(\zeta-t'_i)<\theta'_i$.
Provided that
$|\zeta-\eta_0|<\min_{1\leq k\leq p}|t'_k-\eta_0|$,
we expand
$w_{t'_i,k,h}(-\nu_{1}-1;\zeta)$
in powers of $\zeta-\eta_0$:
\begin{displaymath}
w_{t'_i,k,h}(-\nu_{1}-1;\zeta)
=\sum_{m=0}^{\infty}
g^{\eta_0}_{t'_i,k,h}(-\nu_{1}-1;m)
(\zeta-\eta_0)^{m}.
\end{displaymath}
Substituting this expansion
into $W^{\langle\mathcal{S}^{\prime\pm}_i;\eta_0 \xi\rangle}
_{t'_i,k,h}(\nu_{1},\nu_{2};\xi)$,
we obtain
\begin{displaymath}
W^{\langle\mathcal{S}^{\prime\pm}_i;\eta_0\xi\rangle}
_{t'_i,k,h}(\nu_{1},\nu_{2};\xi)
=
\begin{pmatrix}
\eta_0 I_n-T' \\
P^{-1}M_{\eta_0}(T',A)
\end{pmatrix}
\sum_{m=0}^{\infty}
g^{\eta_0}_{t'_i,k,h}(-\nu_{1}-1;m)
\int_{\eta_0}^{\xi}
\left(\frac{\xi-\zeta}{\xi-\eta_0}\right)^{\nu_{1}}
(\zeta-\eta_0)^{m-\nu_{2}-1}
\,d\zeta.
\end{displaymath}
Changing the variable of integration $\zeta$
to $s$ by $\zeta=\eta_0+(\xi-\eta_0)s$,
we obtain
\begin{displaymath}
\begin{aligned}
\int_{\eta_0}^{\xi}
\left(\frac{\xi-\zeta}{\xi-\eta_0}\right)^{\nu_{1}}
(\zeta-\eta_0)^{m-\nu_{2}-1}
\,d\zeta
&=
(\xi-\eta_0)^{m-\nu_{2}}
\int_{0}^{1}
(1-s)^{\nu_{1}} s^{m-\nu_{2}-1}
\,ds
\\
&=
\frac{\Gamma(\nu_{1}+1)\Gamma(m-\nu_{2})}
{\Gamma(m+\nu_{1}-\nu_{2}+1)}
(\xi-\eta_0)^{m-\nu_{2}}.
\end{aligned}
\end{displaymath}
Then we have
\begin{displaymath}
\begin{aligned}
&
W^{\langle\mathcal{S}^{\prime\pm}_i;\eta_0\xi\rangle}
_{t'_i,k,h}(\nu_{1},\nu_{2};\xi)
\\
&=
\begin{pmatrix}
\eta_0 I_n-T' \\
P^{-1}M_{\eta_0}(T',A)
\end{pmatrix}
\sum_{m=0}^{\infty}
\frac{\Gamma(\nu_{1}+1)\Gamma(m-\nu_{2})}
{\Gamma(m+\nu_{1}-\nu_{2}+1)}
g^{\eta_0}_{t'_i,k,h}(-\nu_{1}-1;m)
(\xi-\eta_0)^{m-\nu_{2}}
\\
&=
(\xi-\eta_0)^{-\nu_{2}}
\left\{
\frac{\Gamma(\nu_{1}+1)\Gamma(-\nu_{2})}
{\Gamma(\nu_{1}-\nu_{2}+1)}
\begin{pmatrix}
\eta_0 I_n-T' \\
P^{-1}(\nu_{2}I_n-A)(\eta_0 I_n-T')
\end{pmatrix}
g^{\eta_0}_{t'_i,k,h}(-\nu_{1}-1;0)
+O\left(\xi-\eta_0\right)
\right\},
\end{aligned}
\end{displaymath}
which leads to (\ref{asymp:W_eta0xi_t'i}).
\end{proof}


\begin{theorem} \label{thm:asymp_W_t'i_xi}
Assume that $\nu_{1}\notin \mathbb{Z}_{<0}$
and $\nu_{1}-\lambda_{i,k}\notin \mathbb{Z}_{\geq0}$.
The analytic continuation of the integral
$W^{\langle\mathcal{S}^{\prime-}_i;t'_i \xi\rangle}
_{t'_i,k,h}(\nu_{1},\nu_{2};\xi)$
across the open segment
$\Sigma^{(0,1)}_i$
into $\mathcal{S}^{\prime+}_i$
coincides with the integral
$W^{\langle\mathcal{S}^{\prime+}_i;t'_i \xi\rangle}
_{t'_i,k,h}(\nu_{1},\nu_{2};\xi)$.
Moreover, as $\xi\longrightarrow t'_i$,
$\xi\in
\mathcal{S}^{\prime-}_i\cup\Sigma^{(0,1)}_i
                       \cup\mathcal{S}^{\prime+}_i$,
we have
\begin{equation} \label{asymp:W_xi_t'i}
\begin{aligned}
W^{\langle\mathcal{S}^{\prime\pm}_i;t'_i \xi\rangle}
_{t'_i,k,h}(\nu_{1},\nu_{2};\xi)
&=
-(t'_i-\eta_0)^{-\nu_{1}-\nu_{2}}
\frac{\Gamma(\nu_{1}+1)
\Gamma(\lambda_{i,k}-\nu_{1})}
{\Gamma(\lambda_{i,k}+1)}
\\
&
\phantom{{}={}}
{}\times
(\xi-t'_i)^{\lambda_{i,k}}
\left\{\varepsilon_{2n}
(n_1+\cdots+n_{i-1}+\ell_{i,1}+\cdots+\ell_{i,k-1}+h)
+O(\xi-t'_i)\right\}.
\end{aligned}
\end{equation}
\end{theorem}

To prove this theorem we prepare a lemma.

\begin{lemma}[{\cite[Proposition 2.5]{Ha}}] \label{lem:Mint}
The integral
\begin{displaymath}
v(\xi)
=\int_{a}^{b}
\left(\frac{\xi-\zeta}{\xi-\eta_0}\right)^{\nu_{1}}
(\zeta-\eta_0)^{-\nu_{2}-1}
w(-\nu_{1}-1;\zeta)\,d\zeta,
\end{displaymath}
where $w(-\nu_{1}-1;\zeta)$ is a solution of $(\ref{sys:underlying})$,
satisfies
\begin{equation} \label{ftn:MintC}
M_{\eta_0}(T',A)v(\xi)
=
(\nu_{2}I_n-A)
\int_{a}^{b}
\left(\frac{\xi-\zeta}{\xi-\eta_0}\right)^{\nu_{1}}
(\zeta-\eta_0)^{-\nu_{2}-1}
(\zeta I_n-T')
w(-\nu_{1}-1;\zeta)
\,d\zeta,
\end{equation}
provided that 
\begin{equation} \label{condition:MintC}
\Bigl[
(\xi-\zeta)^{\nu_{1}}
(\zeta-\eta_0)^{-\nu_{2}}
(\zeta I_n-T')
w(-\nu_{1}-1;\zeta)
\Bigr]_{a}^{b}=0.
\end{equation}
\end{lemma}

\begin{proof}
Since
\begin{displaymath}
\frac{d}{d\xi}v(\xi)
=\int_{a}^{b}
\frac{\nu_{1}}{(\xi-\eta_0)^2}
\left(\frac{\xi-\zeta}{\xi-\eta_0}\right)^{\nu_{1}-1}
(\zeta-\eta_0)^{-\nu_{2}}
w(-\nu_{1}-1;\zeta)\,d\zeta,
\end{displaymath}
we have
\begin{displaymath}
\begin{aligned}
&-(\xi-\eta_0)(\xi I_n-T')
\frac{d}{d\xi}v(\xi)
\\
&=
-
\int_{a}^{b}
\frac{\nu_{1}}{\xi-\eta_0}
\{(\xi-\zeta)I_n+(\zeta I_n-T')\}
\left(\frac{\xi-\zeta}{\xi-\eta_0}\right)^{\nu_{1}-1}
(\zeta-\eta_0)^{-\nu_{2}}
w(-\nu_{1}-1;\zeta)\,d\zeta
\\
&=
-
\int_{a}^{b}\left[\nu_{1}
\left(\frac{\xi-\zeta}{\xi-\eta_0}\right)^{\nu_{1}}
-(\zeta I_n-T')
\frac{\partial}{\partial \zeta}\smash{\biggl\{
\left(\frac{\xi-\zeta}{\xi-\eta_0}\right)^{\nu_{1}}
\biggr\}}
\,\right]
(\zeta-\eta_0)^{-\nu_{2}}
w(-\nu_{1}-1;\zeta)\,d\zeta
\\
&=
-
\int_{a}^{b}
\left(\frac{\xi-\zeta}{\xi-\eta_0}\right)^{\nu_{1}}
\left[
\nu_{1} (\zeta-\eta_0)^{-\nu_{2}} w(-\nu_{1}-1;\zeta)
+\frac{d}{d\zeta}\left\{
(\zeta-\eta_0)^{-\nu_{2}}(\zeta I_n-T')w(-\nu_{1}-1;\zeta)
\right\}
\right]\,d\zeta
\\
& \phantom{{}={}}
+(\xi-\eta_0)^{-\nu_{1}}
\Bigl[
(\xi-\zeta)^{\nu_{1}} (\zeta-\eta_0)^{-\nu_{2}}
(\zeta I_n-T') w(-\nu_{1}-1;\zeta)
\Bigr]_{a}^{b}
\\
&=
A(\eta_0 I_n-T')v(\xi)
+(\nu_{2}I_n-A)
\int_{a}^{b}
\left(\frac{\xi-\zeta}{\xi-\eta_0}\right)^{\nu_{1}}
(\zeta-\eta_0)^{-\nu_{2}-1}
(\zeta I_n-T')w(-\nu_{1}-1;\zeta)\,d\zeta,
\end{aligned}
\end{displaymath}
which leads to (\ref{ftn:MintC}).
Here we have used (\ref{condition:MintC})
and
\begin{displaymath}
\begin{aligned}
&\frac{d}{d\zeta}\left\{
(\zeta-\eta_0)^{-\nu_{2}}(\zeta I_n-T')w(-\nu_{1}-1;\zeta)
\right\}
\\
&=
-\nu_{2}(\zeta-\eta_0)^{-\nu_{2}-1}(\zeta I_n-T')w(-\nu_{1}-1;\zeta)
+(\zeta-\eta_0)^{-\nu_{2}}(A-\nu_{1}I_n)w(-\nu_{1}-1;\zeta)
\\
&=
-(\zeta-\eta_0)^{-\nu_{2}-1}
\left\{
A(\eta_0 I_n-T')+(\nu_{2}I_n-A)(\zeta I_n-T')
\right\}
w(-\nu_{1}-1;\zeta)
-\nu_{1} (\zeta-\eta_0)^{-\nu_{2}} w(-\nu_{1}-1;\zeta)
\end{aligned}
\end{displaymath}
in the last equality.
\end{proof}

\begin{proof}[Proof of Theorem $\ref{thm:asymp_W_t'i_xi}$]
The equality of the analytic continuation of
$W^{\langle\mathcal{S}^{\prime-}_i;t'_i \xi\rangle}_{t'_i,k,h}$
and the integral
$W^{\langle\mathcal{S}^{\prime+}_i;t'_i \xi\rangle}_{t'_i,k,h}$
follows from the continuity of the arguments
$\arg(\xi-\zeta)$,
$\arg(\zeta-\eta_0)$
and
$\arg(\zeta-t'_i)$
in $\mathcal{S}^{\prime-}_i\cup\Sigma^{(0,1)}_i
                           \cup\mathcal{S}^{\prime+}_i$,
where we define
$\arg(\xi-\zeta)=\theta'_i-\pi$,
$\arg(\zeta-\eta_0)=\theta'_i$
and
$\arg(\zeta-t'_i)=\theta'_i-\pi$
for $\xi\in\Sigma^{(0,1)}_i$.

Applying Lemma \ref{lem:Mint} we have
\begin{displaymath}
W^{\langle\mathcal{S}^{\prime\pm}_i;t'_i \xi\rangle}
_{t'_i,k,h}(\nu_{1},\nu_{2};\xi)
=
\int_{t'_i}^{\xi}
\left(\frac{\xi-\zeta}{\xi-\eta_0}\right)^{\nu_{1}}
(\zeta-\eta_0)^{-\nu_{2}-1}
\begin{pmatrix}
\eta_0 I_n-T' \\
P^{-1}(\nu_{2} I_n-A)(\zeta I_n-T')
\end{pmatrix}
w_{t'_i,k,h}(-\nu_{1}-1;\zeta)
\,d\zeta.
\end{displaymath}
Provided that $|\xi-t'_i|<\min_{k\ne i}|t'_k-t'_i|$,
substituting (\ref{def:w_ti})
with $\rho=-\nu_{1}-1$ and
\begin{displaymath}
\begin{pmatrix}
\eta_0 I_n-T' \\
P^{-1}(\nu_{2} I_n-A)(\zeta I_n-T')
\end{pmatrix}
=
\begin{pmatrix}
\eta_0 I_n-T' \\
P^{-1}(\nu_{2} I_n-A)(t'_i I_n-T')
\end{pmatrix}
+(\zeta-t'_i)
\begin{pmatrix}
O \\
P^{-1}(\nu_{2} I_n-A)
\end{pmatrix},
\end{displaymath}
we obtain
\begin{displaymath}
W^{\langle\mathcal{S}^{\prime\pm}_i;t'_i \xi\rangle}
_{t'_i,k,h}(\nu_{1},\nu_{2};\xi)
=
\sum_{m=0}^{\infty}
F_{t'_i,k,h}(m)
\int_{t'_i}^{\xi}
\left(\frac{\xi-\zeta}{\xi-\eta_0}\right)^{\nu_{1}}
(\zeta-\eta_0)^{-\nu_{2}-1}
(\zeta-t'_i)^{m+\lambda_{i,k}-\nu_{1}-1}
\,d\zeta,
\end{displaymath}
where the initial coefficient is given by
\begin{displaymath}
\begin{aligned}
F_{t'_i,k,h}(0)
&=
\begin{pmatrix}
\eta_0 I_n-T' \\
P^{-1}(\nu_{2} I_n-A)(t'_i I_n-T')
\end{pmatrix}
g_{t'_i,k,h}(0)
\\
&=-(t'_i-\eta_0)
\varepsilon_{2n}
(n_1+\cdots+n_{i-1}+\ell_{i,1}+\cdots+\ell_{i,k-1}+h).
\end{aligned}
\end{displaymath}
Changing the variable of integration $\zeta$
to $s$ by $\zeta=t'_i+(\xi-t'_i)s$,
we obtain
\begin{displaymath}
\begin{aligned}
&
\int_{t'_i}^{\xi}
\left(\frac{\xi-\zeta}{\xi-\eta_0}\right)^{\nu_{1}}
(\zeta-\eta_0)^{-\nu_{2}-1}
(\zeta-t'_i)^{m+\lambda_{i,k}-\nu_{1}-1}
\,d\zeta
\\
&= (\xi-\eta_0)^{-\nu_{1}}
(\xi-t'_i)^{m+\lambda_{i,k}}
(t'_i-\eta_0)^{-\nu_{2}-1}
\int_{0}^{1}
(1-s)^{\nu_{1}}
\left(1-\frac{\xi-t'_i}{\eta_0-t'_i} s \right)^{-\nu_{2}-1}
s^{m+\lambda_{i,k}-\nu_{1}-1}
\,ds,
\end{aligned}
\end{displaymath}
where $\arg(t'_i-\eta_0)=\theta'_i$
and hence
$\displaystyle
\arg\left(1-\frac{\xi-t'_i}{\eta_0-t'_i} s \right)
\in[0,\psi-\theta'_i]\subset(-\pi,\pi)$.
So the integral on the right-hand side is
expressed by the Gauss hypergeometric function as
\begin{displaymath}
\frac{\Gamma(\nu_{1}+1)
\Gamma(m+\lambda_{i,k}-\nu_{1})}
{\Gamma(m+\lambda_{i,k}+1)}
{}_2F_1\biggl(
\genfrac{}{}{0pt}{0}
{m+\lambda_{i,k}-\nu_{1}, \ \  \nu_{2}+1}
{m+\lambda_{i,k}+1}
\,;\,
\frac{\xi-t'_i}{\eta_0-t'_i}\biggr).
\end{displaymath}
Then we have
\begin{displaymath}
\begin{aligned}
W^{\langle\mathcal{S}^{\prime\pm}_i;t'_i \xi\rangle}
_{t'_i,k,h}(\nu_{1},\nu_{2};\xi)
&=
\Gamma(\nu_{1}+1)
(\xi-\eta_0)^{-\nu_{1}}
(t'_i-\eta_0)^{-\nu_{2}-1}
(\xi-t'_i)^{\lambda_{i,k}}
\\
&
\phantom{{}={}}
{}\times
\sum_{m=0}^{\infty}
\frac{\Gamma(m+\lambda_{i,k}-\nu_{1})}
{\Gamma(m+\lambda_{i,k}+1)}
F_{t'_i,k,h}(m)
{}_2F_1\biggl(
\genfrac{}{}{0pt}{0}
{m+\lambda_{i,k}-\nu_{1}, \ \  \nu_{2}+1}
{m+\lambda_{i,k}+1}
\,;\,
\frac{\xi-t'_i}{\eta_0-t'_i}\biggr)
(\xi-t'_i)^{m}.
\end{aligned}
\end{displaymath}
Combining this with
$\displaystyle
(\xi-\eta_0)^{-\nu_{1}}
=(t'_i-\eta_0)^{-\nu_{1}}
 \left(1-\frac{\xi-t'_i}{\eta_0-t'_i}\right)^{-\nu_{1}}$,
we have (\ref{asymp:W_xi_t'i}).
\end{proof}


We next consider the integrals
$W^{\langle\mathcal{S}^{\prime\pm}_i;\eta_0 t'_i\rangle}
_{t'_i,k,h}(\nu_{1},\nu_{2};\xi)$,
which make sense if $\xi$ is in
$\mathbb{C}\setminus\Sigma^{[0,1]}_i$.

\begin{theorem} \label{thm:asymp_W_t'i_eta0}
Assume that $\nu_{2}\notin \mathbb{Z}_{\geq0}$
and $\nu_{1}-\lambda_{i,k}\notin \mathbb{Z}_{\geq0}$.
The integrals
$W^{\langle\mathcal{S}^{\prime\pm}_i;\eta_0 t'_i\rangle}
_{t'_i,k,h}(\nu_{1},\nu_{2};\xi)$
are analytically continued to a $2n$-vector function
that is holomorphic and single-valued in
$\mathbb{C}\setminus\Sigma^{[0,1]}_i$.
Moreover, the integral
$W^{\langle\mathcal{S}^{\prime-}_i;\eta_0 t'_i\rangle}
_{t'_i,k,h}(\nu_{1},\nu_{2};\xi)$
coincides with the integral
$W^{\langle\mathcal{S}^{\prime+}_i;\eta_0 t'_i\rangle}
_{t'_i,k,h}(\nu_{1},\nu_{2};\xi)$
multiplied by $e^{2\pi\sqrt{-1}\nu_{1}}$.
Besides, near $\xi=t'_i$,
$\xi\in\mathcal{S}^{\prime\pm}_i$,
we have
\begin{equation} \label{asymp:W_eta0_t'i}
W^{\langle\mathcal{S}^{\prime\pm}_i;\eta_0 t'_i\rangle}
_{t'_i,k,h}(\nu_{1},\nu_{2};\xi)
=
-\frac{e^{\pm\pi\sqrt{-1}(\lambda_{i,k}-\nu_{1})}
      \sin\pi\nu_{1}}
     {\sin\pi\lambda_{i,k}}
W^{\langle\mathcal{S}^{\prime\pm}_i;t'_i \xi\rangle}
_{t'_i,k,h}(\nu_{1},\nu_{2};\xi)
+\mathrm{hol}(\xi-t'_i),
\end{equation}
where the double-signs correspond.
\end{theorem}

\begin{figure}
 \centering
 \includegraphics{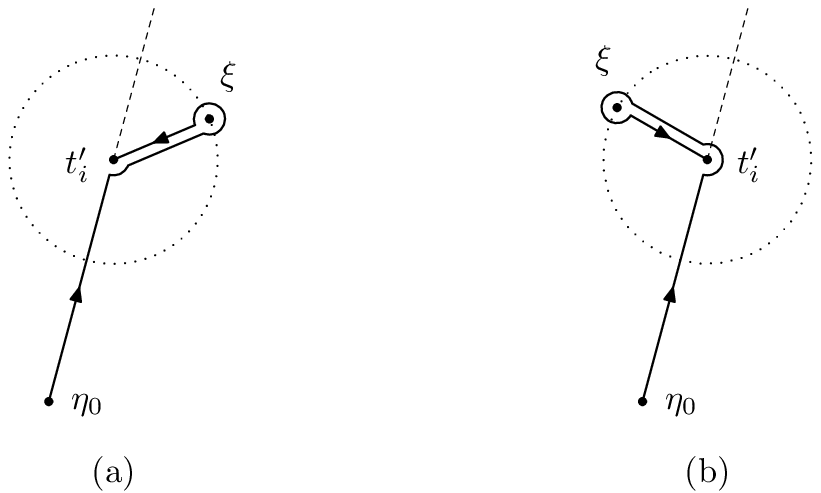}
 \caption{Deformation of the path $\overline{\eta_0 t'_i}$}
 \label{fig:deformation_small}
\end{figure}

\begin{proof}
It is trivial that $W^{\langle\mathcal{S}^{\prime\pm}_i;\eta_0 t'_i\rangle}
_{t'_i,k,h}(\nu_{1},\nu_{2};\xi)$
is holomorphic in $\xi\in\mathbb{C}\setminus\Sigma^{[0,1]}_i$.
Since the factor
\begin{displaymath}
\left(\frac{\xi-\zeta}{\xi-\eta_0}\right)^{\nu_{1}}
\end{displaymath}
is single-valued in $\xi\in\mathbb{C}\setminus\Sigma^{[0,1]}_i$
for $\zeta\in\Sigma^{[0,1]}_i$,
$W^{\langle\mathcal{S}^{\prime\pm}_i;\eta_0 t'_i\rangle}
_{t'_i,k,h}(\nu_{1},\nu_{2};\xi)$
is also single-valued in $\xi\in\mathbb{C}\setminus\Sigma^{[0,1]}_i$.

If we extend the assignment of the arguments
$\arg(\xi-t'_i)$ and $\arg(\xi-\zeta)$ for
$\xi\in\mathcal{S}^{\prime-}_i$
into $\mathcal{S}^{\prime-}_i\cup\Sigma^{(1,\infty)}_i
\cup\mathcal{S}^{\prime+}_i$ continuously,
then for $\xi\in\mathcal{S}^{\prime+}_i$ we have
\begin{displaymath}
\left\{\begin{aligned}
\arg(\xi-t'_i)&\in (\theta'_i,\theta'_i+\pi), \\
\arg(\rlap{$\xi-\zeta)$}\phantom{\xi-t'_i)}
  &\in [\arg(\xi-\eta_0),\arg(\xi-t'_i)]
\subset(\theta'_i,\theta'_i+\pi).
\end{aligned}\right.
\end{displaymath}
On the other hand,
the assignment of $\arg(\xi-\zeta)$
for $W^{\langle\mathcal{S}^{\prime+}_i;\eta_0 t'_i\rangle}_{t'_i,k,h}$
prescribed in Table \ref{tab:arg_t'i_plus} satisfies
\begin{displaymath}
\arg(\xi-\zeta)\in[\psi-2\pi,\chi_i]
\subset(\theta'_i-2\pi,\theta'_i-\pi).
\end{displaymath}
This means that
the analytic continuation of
$W^{\langle\mathcal{S}^{\prime-}_i;\eta_0 t'_i\rangle}_{t'_i,k,h}$
across the ray $\Sigma^{(1,\infty)}_i$
into $\mathcal{S}^{\prime+}_i$
differs from
$W^{\langle\mathcal{S}^{\prime+}_i;\eta_0 t'_i\rangle}_{t'_i,k,h}$
in the multiplier factor $e^{2\pi\sqrt{-1}\nu_{1}}$.

To find the behavior of
$W^{\langle\mathcal{S}^{\prime\pm}_i;\eta_0 t'_i\rangle}
_{t'_i,k,h}(\nu_{1},\nu_{2};\xi)$
near $\zeta=t'_i$
we analytically continue it
along a circle of center $t'_i$ with sufficiently small radius
in the positive direction (see Figure \ref{fig:deformation_small}
(a) for $W^{\langle\mathcal{S}^{\prime-}_i;\eta_0 t'_i\rangle}_{t'_i,k,h}$ or
(b) for $W^{\langle\mathcal{S}^{\prime+}_i;\eta_0 t'_i\rangle}_{t'_i,k,h}$).
Carefully tracing how the arguments change along the path,
we obtain
\begin{displaymath}
W^{\langle\mathcal{S}^{\prime\pm}_i;\eta_0 t'_i\rangle}
_{t'_i,k,h}(\nu_{1},\nu_{2};\xi)
+\epsilon^{\pm}
\left(1-e^{2\pi\sqrt{-1}\nu_{1}}\right)
W^{\langle\mathcal{S}^{\prime\pm}_i;t'_i \xi\rangle}
_{t'_i,k,h}(\nu_{1},\nu_{2};\xi),
\end{displaymath}
where $\epsilon^{-}=1$ for
$W^{\langle\mathcal{S}^{\prime-}_i;\eta_0 t'_i\rangle}_{t'_i,k,h}$
while
$\epsilon^{+}=e^{2\pi\sqrt{-1}(\lambda_{i,k}-\nu_{1})}$
for $W^{\langle\mathcal{S}^{\prime+}_i;\eta_0 t'_i\rangle}_{t'_i,k,h}$.
This means that
\begin{displaymath}
W^{\langle\mathcal{S}^{\prime\pm}_i;\eta_0 t'_i\rangle}
_{t'_i,k,h}(\nu_{1},\nu_{2};\xi)
=-\frac{\epsilon^{\pm}
\bigl(e^{2\pi\sqrt{-1}\nu_{1}}-1\bigr)}
{e^{2\pi\sqrt{-1}\lambda_{i,k}}-1}
W^{\langle\mathcal{S}^{\prime\pm}_i;t'_i \xi\rangle}
_{t'_i,k,h}(\nu_{1},\nu_{2};\xi)
+\mathrm{hol}(\xi-t'_i)
\end{displaymath}
near $\xi=t'_i$,
which leads to (\ref{asymp:W_eta0_t'i}).
\end{proof}

Now we shall investigate the integral
$W^{\langle\mathcal{S}^{\prime\pm}_i;t'_i \infty\rangle}
_{t'_i,k,h}(\nu_{1},\nu_{2};\xi)$.
Define the integral
\begin{displaymath}
\tilde{W}^{\langle\mathcal{S}^{\prime\pm}_i;\eta_0\xi\rangle}
_{t'_i,k,h}(\nu_{1},\nu_{2};\xi)
=
\begin{pmatrix}
\eta_0 I_n-T' \\
P^{-1}M_{\eta_0}(T',A)
\end{pmatrix}
\int_{\eta_0}^{\xi}
\left(\frac{\xi-\zeta}{\xi-\eta_0}\right)^{\nu_{1}}
(\zeta-\eta_0)^{-\nu_{2}-1}
\tilde{w}_{t'_i,k,h}(-\nu_{1}-1;\zeta)
\,d\zeta,
\end{displaymath}
where $\tilde{w}_{t'_i,k,h}(-\nu_{1}-1;\zeta)$
is the solution of the underlying system $(\ref{sys:underlying})$
with $\rho=-\nu_{1}-1$
stated in Theorem \ref{thm:underlying_1cut}.

\begin{theorem} \label{thm:asymp_W_t'i_inf}
Assume that
$\nu_{1}\notin \mathbb{Z}_{\geq0}$,
$\nu_{1}-\lambda_{i,k}\notin \mathbb{Z}_{\geq0}$,
$\nu_{2}\notin \mathbb{Z}_{<0}$,
and
$\nu_{2}-\mu_{l}\notin \mathbb{Z}_{<0}$ for $1\leq l\leq q$.
The analytic continuation of the integral
$W^{\langle\mathcal{S}^{\prime-}_i;t'_i \infty\rangle}
_{t'_i,k,h}(\nu_{1},\nu_{2};\xi)$
across the open segment $\Sigma^{(0,1)}_i$
into $\mathcal{S}^{\prime+}_i$ coincides with the integral
$W^{\langle\mathcal{S}^{\prime+}_i;t'_i \infty\rangle}
_{t'_i,k,h}(\nu_{1},\nu_{2};\xi)$.
Moreover,
the integral $W^{\langle\mathcal{S}^{\prime\pm}_i;t'_i \infty\rangle}
_{t'_i,k,h}(\nu_{1},\nu_{2};\xi)$
multiplied by $(\xi-\eta_0)^{\nu_{2}}$
is holomorphic
for $\xi\in\mathcal{P}'_i$, where
$\mathcal{P}'_i=\mathbb{C}\setminus\Sigma^{[1,\infty)}_i$.
Besides, we have
\begin{equation} \label{int:W_t'iinf_t'i}
W^{\langle\mathcal{S}^{\prime\pm}_i;t'_i\infty\rangle}
_{t'_i,k,h}(\nu_{1},\nu_{2};\xi)
=
-\frac{e^{\pi\sqrt{-1}(\lambda_{i,k}-\nu_{1}-\nu_{2})}
       \Gamma(\nu_{2}-\lambda_{i,k}+1)\Gamma(\lambda_{i,k}-\nu_{1})}
      {\Gamma(\nu_{2}+1)\Gamma(-\nu_{1})}
\tilde{W}^{\langle\mathcal{S}^{\prime\pm}_i;\eta_0\xi\rangle}
_{t'_i,k,h}(\nu_{2},\nu_{1};\xi).
\end{equation}
\end{theorem}

\begin{proof}
The equality of the analytic continuation of
$W^{\langle\mathcal{S}^{\prime-}_i;t'_i \infty\rangle}_{t'_i,k,h}$
and the integral
$W^{\langle\mathcal{S}^{\prime+}_i;t'_i \infty\rangle}_{t'_i,k,h}$
follows from the continuity of the arguments
$\arg(\xi-\zeta)$,
$\arg(\zeta-\eta_0)$
and
$\arg(\zeta-t'_i)$
in $\mathcal{S}^{\prime-}_i\cup\Sigma^{(0,1)}_i
                       \cup\mathcal{S}^{\prime+}_i$,
where we define
$\arg(\xi-\zeta)=\theta'_i-\pi$,
$\arg(\zeta-\eta_0)=\theta'_i$
and
$\arg(\zeta-t'_i)=\theta'_i$
for $\xi\in\Sigma^{(0,1)}_i$.

It is trivial that
$W^{\langle\mathcal{S}^{\prime\pm}_i;t'_i \infty\rangle}
_{t'_i,k,h}(\nu_{1},\nu_{2};\xi)$
multiplied by $(\xi-\eta_0)^{\nu_{2}}$
is holomorphic in $\mathbb{C}\setminus\Sigma^{[1,\infty)}_i$.

Set
\begin{displaymath}
\begin{aligned}
v(\xi)
&=
\int_{t'_i}^{\infty}
\left(\frac{\xi-\zeta}{\xi-\eta_0}\right)^{\nu_{1}}
(\zeta-\eta_0)^{-\nu_{2}-1}
w_{t'_i,k,h}(-\nu_{1}-1;\zeta)
\,d\zeta,
\\
\tilde{v}(\xi)
&=
\int_{\eta_0}^{\xi}
\left(\frac{\xi-\zeta}{\xi-\eta_0}\right)^{\nu_{2}}
(\zeta-\eta_0)^{-\nu_{1}-1}
\tilde{w}_{t'_i,k,h}(-\nu_{2}-1;\zeta)
\,d\zeta.
\end{aligned}
\end{displaymath}
Substituting the formula (\ref{eutra:t'i_inf}) with
$\rho=-\nu_{2}-1$ and $\rho'=-\nu_{1}-1$
into $\tilde{v}(\xi)$
and then reversing the order of integration,
we obtain
\begin{equation} \label{int:tilde_v}
\begin{aligned}
\tilde{v}(\xi)
&=
\frac{e^{\pi\sqrt{-1}(\nu_{1}-\lambda_{i,k})}
      \Gamma(\nu_{2}-\nu_{1}+1)}
     {\Gamma(\nu_{2}-\lambda_{i,k}+1)\Gamma(\lambda_{i,k}-\nu_{1})}
\\
&\phantom{{}={}}
{}\times
\int_{t'_i}^{\infty}\biggl\{\int_{\eta_0}^{\xi}
\left(\frac{\xi-\zeta}{\xi-\eta_0}\right)^{\nu_{2}}
(\zeta-\eta_0)^{-\nu_{1}-1}
(\zeta-\tau)^{\nu_{1}-\nu_{2}-1}
\,d\zeta\biggr\}
w_{t'_i,k,h}(-\nu_{1}-1;\tau)
\,d\tau,
\end{aligned}
\end{equation}
where $\arg(\zeta-\tau)\in[\chi_i,\theta'_i-\pi]^{*}
\subset(\theta'_i-2\pi,\theta'_i)$.
Setting $\zeta-\tau=e^{-\pi\sqrt{-1}}(\tau-\zeta)$
and $\zeta=\eta_0+(\xi-\eta_0)s$,
we find
\begin{displaymath}
\begin{aligned}
&
\int_{\eta_0}^{\xi}
\left(\frac{\xi-\zeta}{\xi-\eta_0}\right)^{\nu_{2}}
(\zeta-\eta_0)^{-\nu_{1}-1}
(\zeta-\tau)^{\nu_{1}-\nu_{2}-1}
\,d\zeta
\\
&=
(\xi-\eta_0)^{-\nu_{1}}
\left(e^{-\pi\sqrt{-1}}
(\tau-\eta_0)\right)^{\nu_{1}-\nu_{2}-1}
\int_{0}^{1}
(1-s)^{\nu_{2}}
s^{-\nu_{1}-1}
\left(1-\frac{\xi-\eta_0}{\tau-\eta_0}s
\right)^{\nu_{1}-\nu_{2}-1}
\,ds,
\end{aligned}
\end{displaymath}
where $\arg(\tau-\eta_0)=\theta'_i$
and $\displaystyle
\arg\left(1-\frac{\xi-\eta_0}{\tau-\eta_0}s\right)
\in(-\pi,\pi)$.
So the integral on the right-hand side is
expressed by the Gauss hypergeometric function as
\begin{displaymath}
\frac{\Gamma(\nu_{2}+1)\Gamma(-\nu_{1})}
{\Gamma(\nu_{2}-\nu_{1}+1)}
{}_2F_1\biggl(
\genfrac{}{}{0pt}{0}
{\nu_{2}-\nu_{1}+1, \ \  -\nu_{1}}
{\nu_{2}-\nu_{1}+1}
\,;\,
\frac{\xi-\eta_0}{\tau-\eta_0}\biggr),
\end{displaymath}
which is reduced to
\begin{displaymath}
\frac{\Gamma(\nu_{2}+1)\Gamma(-\nu_{1})}
{\Gamma(\nu_{2}-\nu_{1}+1)}
\left(\frac{e^{\pi\sqrt{-1}}(\xi-\tau)}
{\tau-\eta_0}\right)^{\nu_{1}},
\end{displaymath}
where $\arg(\xi-\tau)\in[\chi_i,\theta'_i-\pi]^{*}
\subset(\theta'_i-2\pi,\theta'_i)$.
Thus we have
\begin{displaymath}
\int_{\eta_0}^{\xi}
\left(\frac{\xi-\zeta}{\xi-\eta_0}\right)^{\nu_{2}}
(\eta_0-\zeta)^{-\nu_{1}-1}
(\zeta-\tau)^{\nu_{1}-\nu_{2}-1}
\,d\zeta
=
e^{\pi\sqrt{-1}(\nu_{2}+1)}
\frac{\Gamma(\nu_{2}+1)\Gamma(-\nu_{1})}
{\Gamma(\nu_{2}-\nu_{1}+1)}
\left(\frac{\xi-\tau}{\xi-\eta_0}\right)^{\nu_{1}}
(\tau-\eta_0)^{-\nu_{2}-1}.
\end{displaymath}
Substituting this formula into (\ref{int:tilde_v}),
we have
\begin{displaymath}
\begin{aligned}
\tilde{v}(\xi)
&=
-\frac{e^{\pi\sqrt{-1}(\nu_{1}+\nu_{2}-\lambda_{i,k})}
\Gamma(\nu_{2}+1)\Gamma(-\nu_{1})}
{\Gamma(\nu_{2}-\lambda_{i,k}+1)
\Gamma(\lambda_{i,k}-\nu_{1})}
v(\xi).
\end{aligned}
\end{displaymath}
This implies (\ref{int:W_t'iinf_t'i}).
\end{proof}

\subsection{Integrals associated with the solutions near infinity}
Similarly to Theorem \ref{thm:asymp_W_eta0_xi},
we can prove the following theorem.
\begin{theorem} \label{thm:asymp_Winf_eta0_xi}
Assume that $\nu_{1}\notin \mathbb{Z}_{<0}$
and $\nu_{2}\notin \mathbb{Z}_{\geq0}$.
As $\xi\longrightarrow\eta_0$,
$\xi\in\mathcal{S}^{\prime\pm}_{i}$,
we have
\begin{displaymath}
\begin{aligned}
&
W^{\langle\mathcal{S}^{\prime\pm}_i;\eta_0\xi\rangle}
_{\infty,k,h}(\nu_{1},\nu_{2};\xi)
\\
&=
(\xi-\eta_0)^{-\nu_{2}}
\left\{
\frac{\Gamma(\nu_{1}+1)\Gamma(-\nu_{2})}
{\Gamma(\nu_{1}-\nu_{2}+1)}
\begin{pmatrix}
P \\
\nu_{2}I_n-A'
\end{pmatrix}
P^{-1}(\eta_0 I_n-T')
w_{\infty,k,h}|_{\mathcal{S}^{\prime\pm}_{i}}(-\nu_{1}-1;\eta_0)
+O\left(\xi-\eta_0\right)
\right\},
\end{aligned}
\end{displaymath}
where
$w_{\infty,k,h}|_{\mathcal{S}^{\prime\pm}_{i}}(-\nu_{1}-1;\zeta)$
denotes the analytic continuation of
$w_{\infty,k,h}(-\nu_{1}-1;\zeta)$
into a neighborhood of $\zeta=\eta_0$
through the sector $\mathcal{S}^{\prime\pm}_{i}$
(the double-signs correspond).
\end{theorem}

\begin{proof}
Omitted.
\end{proof}


\begin{theorem} \label{thm:asymp_Winf_xi_inf}
Assume that $\nu_{1}\notin \mathbb{Z}_{<0}$
and $\nu_{2}-\mu_{k}\notin \mathbb{Z}_{<0}$.
The analytic continuation of the integral
$W^{\langle\mathcal{S}^{\prime-}_i;\infty\xi\rangle}
_{\infty,k,h}(\nu_{1},\nu_{2};\xi)$
across the open ray
$\Sigma^{(1,\infty)}_i$
into $\mathcal{S}^{\prime+}_i$
is equal to the integral
$W^{\langle\mathcal{S}^{\prime+}_i;\infty\xi\rangle}
_{\infty,k,h}(\nu_{1},\nu_{2};\xi)$
multiplied by $e^{2\pi\sqrt{-1}\nu_{1}}$.
Moreover, as $\xi\longrightarrow\infty$,
$\xi\in
\mathcal{S}^{\prime-}_i\cup\Sigma^{(1,\infty)}_i
                       \cup\mathcal{S}^{\prime+}_i$,
we have
\begin{equation} \label{asymp:W_xi_inf}
\begin{aligned}
W^{\langle\mathcal{S}^{\prime\pm}_i;\infty\xi\rangle}
_{\infty,k,h}(\nu_{1},\nu_{2};\xi)
&=
-\frac{e^{\mp\pi\sqrt{-1}\nu_{1}}
\Gamma(\nu_{1}+1)
\Gamma(\nu_{2}-\mu_{k}+1)}
{\Gamma(\nu_{1}+\nu_{2}-\mu_{k}+1)}
\\
&\phantom{{}={}}
{}\times
(\xi-\eta_0)^{\mu_{k}-\nu_{1}-\nu_{2}}
\left\{
\varepsilon_{2n}(n+m_1+\cdots+m_{k-1}+h)
+O \left(\frac{1}{\xi-\eta_0}\right)
\right\},
\end{aligned}
\end{equation}
where the double-signs correspond.
\end{theorem}

\begin{proof}
If we extend the assignment of $\arg(\xi-\zeta)$ for
$\xi\in\mathcal{S}^{\prime-}_i$
into $\mathcal{S}^{\prime-}_i\cup\Sigma^{(1,\infty)}_i
\cup\mathcal{S}^{\prime+}_i$ continuously,
then for $\xi\in\mathcal{S}^{\prime+}_i$ we have
\begin{displaymath}
\arg(\xi-\zeta)=\psi+\pi.
\end{displaymath}
On the other hand,
the assignment of $\arg(\xi-\zeta)$
for $W^{\langle\mathcal{S}^{\prime+}_i;\infty\xi\rangle}_{\infty,k,h}$
prescribed in Table \ref{tab:arg_t'i_plus} is
\begin{displaymath}
\arg(\xi-\zeta)=\psi-\pi.
\end{displaymath}
This means that
the analytic continuation of
$W^{\langle\mathcal{S}^{\prime-}_i;\infty\xi\rangle}_{\infty,k,h}$
across the ray $\Sigma^{(1,\infty)}_i$
into $\mathcal{S}^{\prime+}_i$
differs from $W^{\langle\mathcal{S}^{\prime+}_i;\infty\xi\rangle}_{\infty,k,h}$
in the multiplier factor $e^{2\pi\sqrt{-1}\nu_{1}}$.

Assume that $|\xi-\eta_0|>\max_{1\leq k\leq p}|t'_k-\eta_0|$.
Substituting (\ref{def:w_inf}) with $\rho=-\nu_{1}-1$
into $W^{\langle\mathcal{S}^{\prime\pm}_i;\infty\xi\rangle}_{\infty,k,h}$,
we obtain
\begin{displaymath}
\begin{aligned}
&
W^{\langle\mathcal{S}^{\prime\pm}_i;\infty\xi\rangle}
_{\infty,k,h}(\nu_{1},\nu_{2};\xi)
\\
&=
\begin{pmatrix}
\eta_0 I_n-T' \\
P^{-1}M_{\eta_0}(T',A)
\end{pmatrix}
\sum_{m=0}^{\infty}
\frac{\Gamma(\nu_{1}-\mu_{k}+1+m)}{\Gamma(\nu_{1}-\mu_{k}+1)}
g_{\infty,k,h}(m)
\int_{\infty}^{\xi}
\left(\frac{\xi-\zeta}{\xi-\eta_0}\right)^{\nu_{1}}
(\zeta-\eta_0)^{-\nu_{2}-\nu_{1}+\mu_{k}-m-2}
\,d\zeta.
\end{aligned}
\end{displaymath}
Allowing for the assignment of the arguments of
$\xi-\zeta$,
we have
\begin{displaymath}
\int_{\infty}^{\xi}
\left(\frac{\xi-\zeta}{\xi-\eta_0}\right)^{\nu_{1}}
(\zeta-\eta_0)^{-\nu_{2}-\nu_{1}+\mu_{k}-m-2}
\,d\zeta
=
e^{\mp\pi\sqrt{-1}\nu_{1}}
(\xi-\eta_0)^{-\nu_{1}}
\int_{\infty}^{\xi}
(\zeta-\xi)^{\nu_{1}}
(\zeta-\eta_0)^{-\nu_{1}-\nu_{2}+\mu_{k}-m-2}
\,d\zeta,
\end{displaymath}
where $\arg(\zeta-\xi)=\arg(\zeta-\eta_0)=\psi$
and the double-signs correspond.
Changing the variable of integration $\zeta$ to $s$ by
$\displaystyle
 \zeta=\eta_0+\frac{\xi-\eta_0}{s}$,
we can easily see that the integral on the right-hand side is equal to
\begin{displaymath}
-(\xi-\eta_0)^{-\nu_{2}+\mu_{k}-m-1}
\frac{\Gamma(\nu_{1}+1)
\Gamma(\nu_{2}-\mu_{k}+m+1)}
{\Gamma(\nu_{1}+\nu_{2}-\mu_{k}+m+2)}.
\end{displaymath}
Hence we have
\begin{displaymath}
\begin{aligned}
&
W^{\langle\mathcal{S}^{\prime\pm}_i;\infty\xi\rangle}
_{\infty,k,h}(\nu_{1},\nu_{2};\xi)
\\
&=-
e^{\mp\pi\sqrt{-1}\nu_{1}}
\Gamma(\nu_{1}+1)
\\
&\phantom{{}={}}{}\times
\begin{pmatrix}
\eta_0 I_n-T' \\
P^{-1}M_{\eta_0}(T',A)
\end{pmatrix}
\sum_{m=0}^{\infty}
\frac{\Gamma(\nu_{1}-\mu_{k}+m+1)\Gamma(\nu_{2}-\mu_{k}+m+1)}
     {\Gamma(\nu_{1}-\mu_{k}+1)\Gamma(\nu_{1}+\nu_{2}-\mu_{k}+m+2)}
g_{\infty,k,h}(m)
(\xi-\eta_0)^{-\nu_{1}-\nu_{2}+\mu_{k}-m-1}
\\
&=-
e^{\mp\pi\sqrt{-1}\nu_{1}}
\Gamma(\nu_{1}+1)
\\
&\phantom{{}={}}{}\times
(\xi-\eta_0)^{-\nu_{1}-\nu_{2}+\mu_{k}}
\left\{
\frac{\Gamma(\nu_{2}-\mu_{k}+1)}
{\Gamma(\nu_{1}+\nu_{2}-\mu_{k}+2)}
\begin{pmatrix}
O \\
-(-\nu_{1}-\nu_{2}+\mu_{k}-1)P^{-1}
\end{pmatrix}
g_{\infty,k,h}(0)
+O\left(\frac{1}{\xi-\eta_0}\right)
\right\},
\end{aligned}
\end{displaymath}
which implies (\ref{asymp:W_xi_inf}).
\end{proof}


\begin{figure}
 \centering
 \includegraphics{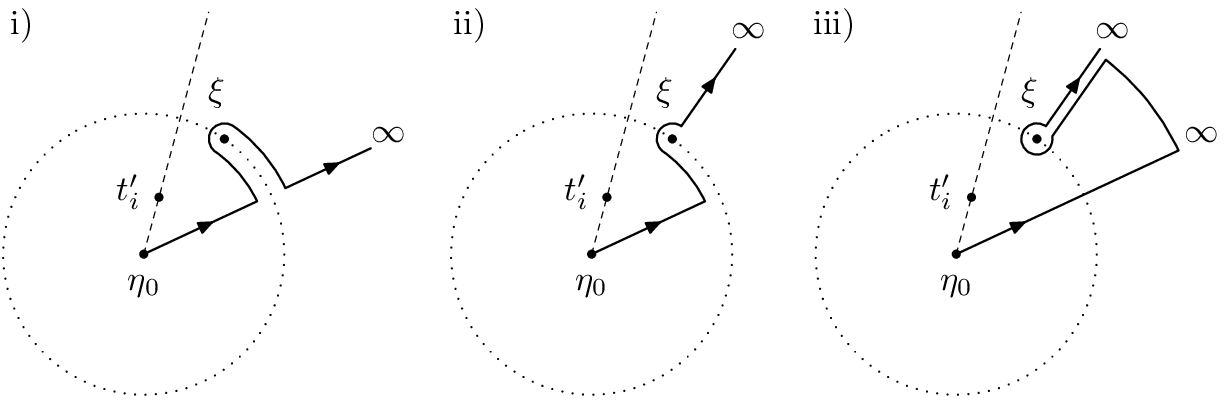}
 \caption{Deformation of the path $\overline{\eta_0 \infty}$
          for $\xi\in\mathcal{S}^{\prime-}_i$}
 \label{fig:deformation_large_S'minus}
\end{figure}

\begin{figure}
 \centering
 \includegraphics{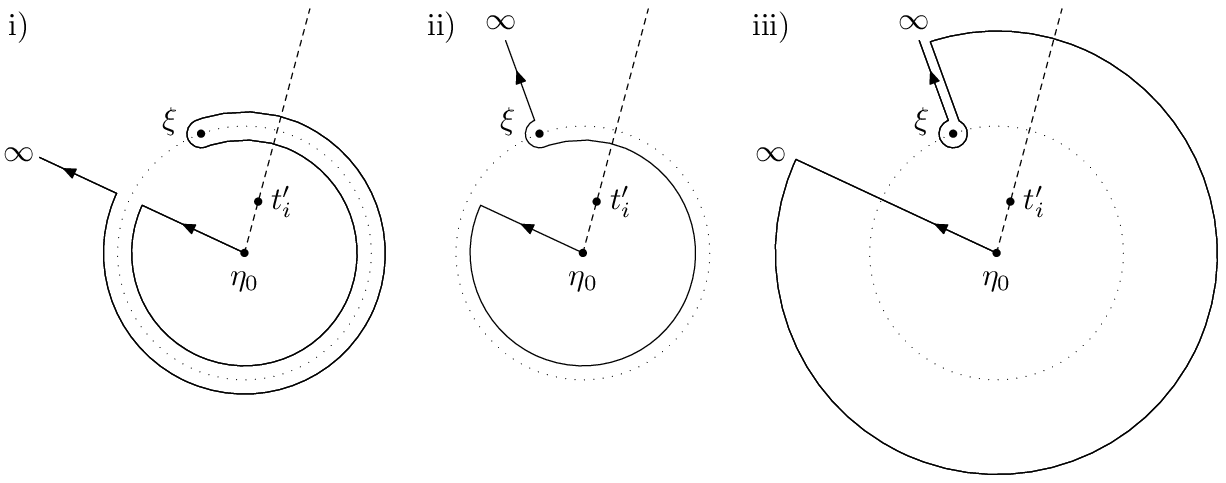}
 \caption{Deformation of the path $\overline{\eta_0 \infty}$
          for $\xi\in\mathcal{S}^{\prime+}_i$}
 \label{fig:deformation_large_S'plus}
\end{figure}

\begin{theorem} \label{thm:asymp_Winf_eta0_inf}
Assume that $\nu_{2}\notin \mathbb{Z}_{\geq0}$,
$\nu_{2}-\mu_{k}\notin \mathbb{Z}_{<0}$
and
$\nu_{1}+\nu_{2}-\mu_{k}\notin \mathbb{Z}$.
The integral $W^{\langle\mathcal{S}^{\prime\pm}_i;\eta_0\infty\rangle}
_{\infty,k,h}(\nu_{1},\nu_{2};\xi)$
is holomorphic for $\xi\in\mathbb{C}\setminus
\{\eta_0+se^{\sqrt{-1}\phi^{\pm}_i}\mid s\geq 0\}$.
Moreover, near $\xi=\infty$,
$\xi\in\mathcal{S}^{\prime\pm}_{i}$,
we have
\begin{equation} \label{asymp:W_eta0_inf}
W^{\langle\mathcal{S}^{\prime\pm}_i;\eta_0\infty\rangle}
_{\infty,k,h}(\nu_{1},\nu_{2};\xi)
=
\frac{e^{\pm\pi\sqrt{-1}(\mu_{k}-\nu_{2})}
       \sin\pi\nu_{1}}
      {\sin\pi(\mu_{k}-\nu_{1}-\nu_{2})}
W^{\langle\mathcal{S}^{\prime\pm}_i;\infty\xi\rangle}
_{\infty,k,h}(\nu_{1},\nu_{2};\xi)
+\mathrm{hol}\left(\frac{1}{\xi-\eta_0}\right).
\end{equation}
Here the double-signs correspond.
\end{theorem}

\begin{proof}
It is trivial that
$W^{\langle\mathcal{S}^{\prime\pm}_i;\eta_0\infty\rangle}
_{\infty,k,h}(\nu_{1},\nu_{2};\xi)$
is holomorphic in $\mathbb{C}\setminus
\{\eta_0+se^{\sqrt{-1}\phi^{\pm}_i}\mid s\geq 0\}$.

To find the behavior of
$W^{\langle\mathcal{S}^{\prime\pm}_i;\eta_0\infty\rangle}
_{\infty,k,h}(\nu_{1},\nu_{2};\xi)$
near $\zeta=\infty$
we analytically continue
$W^{\langle\mathcal{S}^{\prime\pm}_i;\eta_0\infty\rangle}
_{\infty,k,h}(\nu_{1},\nu_{2};\xi)$
along a circle of center $\eta_0$ with sufficiently large radius
in the positive direction.
Then from
Figure \ref{fig:deformation_large_S'minus} for
$W^{\langle\mathcal{S}^{\prime-}_i;\eta_0\infty\rangle}_{\infty,k,h}$
or from
Figure \ref{fig:deformation_large_S'plus} for
$W^{\langle\mathcal{S}^{\prime+}_i;\eta_0\infty\rangle}_{\infty,k,h}$
we see that the analytic continuation of
$W^{\langle\mathcal{S}^{\prime\pm}_i;\eta_0\infty\rangle}
_{\infty,k,h}(\nu_{1},\nu_{2};\xi)$
is
\begin{displaymath}
W^{\langle\mathcal{S}^{\prime\pm}_i;\eta_0\infty\rangle}
_{\infty,k,h}(\nu_{1},\nu_{2};\xi)
+\epsilon^{\pm}
\left(1-e^{-2\pi\sqrt{-1}\nu_{1}}\right)
W^{\langle\mathcal{S}^{\prime\pm}_i;\infty\xi\rangle}
_{\infty,k,h}(\nu_{1},\nu_{2};\xi),
\end{displaymath}
where $\epsilon^{-}\!=1$ for
$W^{\langle\mathcal{S}^{\prime-}_i;\eta_0\infty\rangle}_{\infty,k,h}$
while
$\epsilon^{+}\!=e^{2\pi\sqrt{-1}(\mu_{k}-\nu_{2})}$
for $W^{\langle\mathcal{S}^{\prime+}_i;\eta_0\infty\rangle}_{\infty,k,h}$.
This means that
\begin{displaymath}
W^{\langle\mathcal{S}^{\prime\pm}_i;\eta_0\infty\rangle}
_{\infty,k,h}(\nu_{1},\nu_{2};\xi)
=
\frac{\epsilon^{\pm}\bigl(
1-e^{-2\pi\sqrt{-1}\nu_{1}}\bigr)}
{e^{2\pi\sqrt{-1}(\mu_{k}-\nu_{1}-\nu_{2})}-1}
W^{\langle\mathcal{S}^{\prime\pm}_i;\infty\xi\rangle}
_{\infty,k,h}(\nu_{1},\nu_{2};\xi)
+\mathrm{hol}\left(\frac{1}{\xi-\eta_0}\right)
\end{displaymath}
near $\xi=\infty$,
which leads to (\ref{asymp:W_eta0_inf}).
\end{proof}

\subsection{Reducible cases}
The following proposition sophisticates Haraoka's result
\cite[Proposition 3.3 and Corollary 3.2]{Ha}.
Recall the notation $[\quad]_{m}$
that denotes the $m$-th component of a column vector.

\begin{proposition} \label{prop:Wvanish}
The integrals
$W^{\langle\mathcal{S}^{\prime\pm}_i;ab\rangle}
_{\mathit{sing},k,h}(\nu_{1},\nu_{2};\xi)$
treated in Theorems
$\ref{thm:asymp_W_eta0_xi}$,
$\ref{thm:asymp_W_t'i_xi}$,
$\ref{thm:asymp_W_t'i_eta0}$,
$\ref{thm:asymp_Winf_eta0_xi}$,
$\ref{thm:asymp_Winf_xi_inf}$ and
$\ref{thm:asymp_Winf_eta0_inf}$
satisfy
\begin{equation} \label{vanish:Wnu1mul}
\left[
W^{\langle\mathcal{S}^{\prime\pm}_i;ab\rangle}
_{\mathit{sing},k,h}(\nu_{1},\mu_{l};\xi)
\right]_{m}=0
\quad \text{for } n+m_1+\cdots+m_{l-1}+1\leq m\leq n+m_1+\cdots+m_{l}
\end{equation}
except for
$W^{\langle\mathcal{S}^{\prime\pm}_i;\infty\xi\rangle}
_{\infty,k,h}(\nu_{1},\mu_{k};\xi)$
and
$W^{\langle\mathcal{S}^{\prime\pm}_i;\eta_0\infty\rangle}
_{\infty,k,h}(\nu_{1},\mu_{k};\xi)$,
and
\begin{equation} \label{vanish:Wmulnu2}
\left[
W^{\langle\mathcal{S}^{\prime\pm}_i;ab\rangle}
_{\mathit{sing},k,h}(\mu_{l},\nu_{2};\xi)
\right]_{m}=0
\quad \text{for } n+m_1+\cdots+m_{l-1}+1\leq m\leq n+m_1+\cdots+m_{l}
\end{equation}
except for
$W^{\langle\mathcal{S}^{\prime\pm}_i;ab\rangle}
_{\infty,k,h}(\mu_{k},\nu_{2};\xi)$,
$a,b\in\{\xi,\eta_0,\infty\}$.
\end{proposition}

\begin{proof}
We can apply Lemma \ref{lem:Mint} to
the integrals
$W^{\langle\mathcal{S}^{\prime\pm}_i;ab\rangle}
_{\mathit{sing},k,h}(\nu_{1},\nu_{2};\xi)$
except for
$W^{\langle\mathcal{S}^{\prime\pm}_i;\infty\xi\rangle}
_{\infty,k,h}(\nu_{1},\mu_{k};\xi)$
and
$W^{\langle\mathcal{S}^{\prime\pm}_i;\eta_0\infty\rangle}
_{\infty,k,h}(\nu_{1},\mu_{k};\xi)$,
and obtain
\begin{displaymath}
W^{\langle\mathcal{S}^{\prime\pm}_i;ab\rangle}
_{\mathit{sing},k,h}(\nu_{1},\nu_{2};\xi)
=
\int_{a}^{b}
\left(\frac{\xi-\zeta}{\xi-\eta_0}\right)^{\nu_{1}}
(\zeta-\eta_0)^{-\nu_{2}-1}
\begin{pmatrix}
(\eta_0 I_n-T') w_{\mathit{sing},k,h}(-\nu_{1}-1;\zeta)
\\
(\nu_{2} I_n-A')
P^{-1}(\zeta I_n-T') w_{\mathit{sing},k,h}(-\nu_{1}-1;\zeta)
\end{pmatrix}
d\zeta.
\end{displaymath}
Combining this expression with the definition of $A'$
and
with the facts stated in the proof of Theorem \ref{thm:zero_conn.coef},
we obtain (\ref{vanish:Wnu1mul}) and (\ref{vanish:Wmulnu2}), respectively.
\end{proof}


\section{Relations of the integrals}
\label{sec:relation}
In this section we investigate relations of the integrals
defined in the preceding section, 
and then translate the relations to those for solutions of
the system $(\ref{sys:E2})$ in the next section.

We assume that
\begin{align}
\label{assumption:NUj_1}
\nu_{j}
&\not\in \mathbb{Z}
&&\text{\hskip-10em for\ }
1\leq j\leq 2,
\\
\label{assumption:NUj_2}
\nu_{j}-\lambda_{i,k}
&\not\in \mathbb{Z}_{\geq0}
&&\text{\hskip-10em for\ }
1\leq j\leq 2,\
1\leq i\leq p,\
1\leq k\leq r_i,
\\
\label{assumption:NUj_3}
\nu_{j}-\mu_{k}
&\not\in \mathbb{Z}_{<0}
&&\text{\hskip-10em for\ }
1\leq j\leq 2,\
1\leq k\leq q.
\end{align}

\begin{theorem} \label{thm:rel_cauchy}
For $\xi\in\mathcal{S}^{\prime-}_i$
we have
\begin{displaymath}
\begin{aligned}
W^{\langle\mathcal{S}^{\prime-}_i;\eta_0\xi\rangle}
_{t'_i,k,h}(\nu_{1},\nu_{2};\xi)
-W^{\langle\mathcal{S}^{\prime-}_i;t'_i \xi\rangle}
_{t'_i,k,h}(\nu_{1},\nu_{2};\xi)
-W^{\langle\mathcal{S}^{\prime-}_i;\eta_0 t'_i\rangle}
_{t'_i,k,h}(\nu_{1},\nu_{2};\xi)
&=0,
\\
W^{\langle\mathcal{S}^{\prime-}_i;\eta_0\xi\rangle}
_{\infty,k,h}(\nu_{1},\nu_{2};\xi)
-W^{\langle\mathcal{S}^{\prime-}_i;\infty\xi\rangle}
_{\infty,k,h}(\nu_{1},\nu_{2};\xi)
-W^{\langle\mathcal{S}^{\prime-}_i;\eta_0 \infty\rangle}
_{\infty,k,h}(\nu_{1},\nu_{2};\xi)
&=0,
\end{aligned}
\end{displaymath}
and for $\xi\in\mathcal{S}^{\prime+}_i$
\begin{displaymath}
\begin{aligned}
e^{-2\pi\sqrt{-1}\nu_{1}}
W^{\langle\mathcal{S}^{\prime+}_i;\eta_0\xi\rangle}
_{t'_i,k,h}(\nu_{1},\nu_{2};\xi)
-W^{\langle\mathcal{S}^{\prime+}_i;t'_i \xi\rangle}
_{t'_i,k,h}(\nu_{1},\nu_{2};\xi)
-W^{\langle\mathcal{S}^{\prime+}_i;\eta_0 t'_i\rangle}
_{t'_i,k,h}(\nu_{1},\nu_{2};\xi)
&=0,
\\
W^{\langle\mathcal{S}^{\prime+}_i;\eta_0\xi\rangle}
_{\infty,k,h}(\nu_{1},\nu_{2};\xi)
-W^{\langle\mathcal{S}^{\prime+}_i;\infty\xi\rangle}
_{\infty,k,h}(\nu_{1},\nu_{2};\xi)
-W^{\langle\mathcal{S}^{\prime+}_i;\eta_0 \infty\rangle}
_{\infty,k,h}(\nu_{1},\nu_{2};\xi)
&=0.
\end{aligned}
\end{displaymath}
\end{theorem}

\begin{proof}
Applying the Cauchy theorem to the triangles
$\triangle(\eta_0, \xi, t'_i)$ and $\triangle(\eta_0, \xi, \infty)$,
we have
\begin{displaymath}
\int_{\partial\triangle(\eta_0, \xi, t'_i)}
\left(\frac{\xi-\zeta}{\xi-\eta_0}\right)^{\nu_{1}}
(\zeta-\eta_0)^{-\nu_{2}-1}
w_{t'_i,k,h}(-\nu_{1}-1;\zeta)
\,d\zeta
=0
\end{displaymath}
and
\begin{displaymath}
\int_{\partial\triangle(\eta_0, \xi, \infty)}
\left(\frac{\xi-\zeta}{\xi-\eta_0}\right)^{\nu_{1}}
(\zeta-\eta_0)^{-\nu_{2}-1}
w_{\infty,k,h}(-\nu_{1}-1;\zeta)
\,d\zeta
=0,
\end{displaymath}
respectively.
Comparing the branches of these integrals and those of
the integrals
$W^{\langle\mathcal{S}^{\prime\pm}_i;ab\rangle}
  _{t'_i,k,h}(\nu_{1},\nu_{2};\xi)$
and
$W^{\langle\mathcal{S}^{\prime\pm}_i;ab\rangle}
  _{\infty,k,h}(\nu_{1},\nu_{2};\xi)$,
we can easily see that the relations above hold.
\end{proof}

Set
\begin{displaymath}
\begin{aligned}
V^{\langle\mathcal{S}^{\prime\pm}_i;t'_i \xi\rangle}
_{t'_i,k,h}(\nu_{1},\nu_{2};\xi)
&=
-\frac{\Gamma(\lambda_{i,k}+1)}
     {\Gamma(\nu_{1}+1)\Gamma(\lambda_{i,k}-\nu_{1})}
W^{\langle\mathcal{S}^{\prime\pm}_i;t'_i \xi\rangle}
_{t'_i,k,h}(\nu_{1},\nu_{2};\xi),
\\
V^{\langle\mathcal{S}^{\prime\pm}_i;\eta_0 t'_i\rangle}
_{t'_i,k,h}(\nu_{1},\nu_{2};\xi)
&=
\frac{e^{\pm\pi\sqrt{-1}\nu_{1}}\Gamma(-\nu_{1})}
     {\Gamma(\lambda_{i,k}-\nu_{1})\Gamma(-\lambda_{i,k})}
W^{\langle\mathcal{S}^{\prime\pm}_i;\eta_0 t'_i\rangle}
_{t'_i,k,h}(\nu_{1},\nu_{2};\xi),
\\
V^{\langle\mathcal{S}^{\prime\pm}_i;\infty\xi\rangle}
_{\infty,k,h}(\nu_{1},\nu_{2};\xi)
&=
-\frac{e^{\pm\pi\sqrt{-1}\nu_{1}}\Gamma(\nu_{1}+\nu_{2}-\mu_{k}+1)}
     {\Gamma(\nu_{1}+1)\Gamma(\nu_{2}-\mu_{k}+1)}
W^{\langle\mathcal{S}^{\prime\pm}_i;\infty\xi\rangle}
_{\infty,k,h}(\nu_{1},\nu_{2};\xi),
\\
V^{\langle\mathcal{S}^{\prime\pm}_i;\eta_0 \infty\rangle}
_{\infty,k,h}(\nu_{1},\nu_{2};\xi)
&=
\frac{\Gamma(-\nu_{1})}
     {\Gamma(\mu_{k}-\nu_{1}-\nu_{2})\Gamma(\nu_{2}-\mu_{k}+1)}
W^{\langle\mathcal{S}^{\prime\pm}_i;\eta_0 \infty\rangle}
_{\infty,k,h}(\nu_{1},\nu_{2};\xi),
\end{aligned}
\end{displaymath}
where the double-signs correspond.
As a corollary to Theorems
\ref{thm:asymp_W_eta0_xi},
\ref{thm:asymp_W_t'i_xi},
\ref{thm:asymp_Winf_xi_inf}
and
\ref{thm:asymp_Winf_eta0_inf},
we have the following theorem.

\begin{theorem}[{cf.~\cite[\S7.4]{Hi}}] \label{thm:for_def_V}
Each of the
$V^{\langle\mathcal{S}^{\prime\pm}_i;ab\rangle}
_{\mathit{sing},k,h}(\nu_{1},\nu_{2};\xi)$'s
is not altered if $\nu_{1}$ and $\nu_{2}$ are interchanged.
\end{theorem}

\begin{proof}
The integrals
$V^{\langle\mathcal{S}^{\prime+}_i;t'_i \xi\rangle}
_{t'_i,k,h}(\nu_{1},\nu_{2};\xi)$
and
$V^{\langle\mathcal{S}^{\prime+}_i;t'_i \xi\rangle}
_{t'_i,k,h}(\nu_{2},\nu_{1};\xi)$
not only
satisfy the system $(\ref{sys:forW})$
but also
have the same asymptotic behavior as $\xi\longrightarrow t'_i$.
This implies that the integrals coincide with each other.
As for other integrals, the coincidence
when $\nu_{1}$ and $\nu_{2}$ are interchanged
follows from their asymptotic behavior also.
\end{proof}

\begin{remark}
From Theorems
\ref{thm:asymp_W_t'i_xi},
\ref{thm:asymp_W_t'i_eta0} and
\ref{thm:asymp_Winf_xi_inf}
we see that
$V^{\langle\mathcal{S}^{\prime\pm}_i;ab\rangle}
_{\mathit{sing},k,h}(\nu_{1},\nu_{2};\xi)$
and its analytic continuation satisfy
\begin{displaymath}
\begin{aligned}
V^{\langle\mathcal{S}^{\prime+}_i;t'_i \xi\rangle}
_{t'_i,k,h}(\nu_{1},\nu_{2};\xi)
&=
V^{\langle\mathcal{S}^{\prime-}_i;t'_i \xi\rangle}
_{t'_i,k,h}(\nu_{1},\nu_{2};\xi)
&&
\text{for\ } \xi\in
\mathcal{S}^{\prime-}_i \cup \Sigma^{(0,1)}_i
                        \cup \mathcal{S}^{\prime+}_i,
\\
V^{\langle\mathcal{S}^{\prime+}_i;\eta_0 t'_i\rangle}
_{t'_i,k,h}(\nu_{1},\nu_{2};\xi)
&=
V^{\langle\mathcal{S}^{\prime-}_i;\eta_0 t'_i\rangle}
_{t'_i,k,h}(\nu_{1},\nu_{2};\xi)
&&
\text{for\ } \xi\in
\mathbb{C} \setminus \Sigma^{[0,1]}_i,
\\
V^{\langle\mathcal{S}^{\prime+}_i;\infty\xi\rangle}
_{\infty,k,h}(\nu_{1},\nu_{2};\xi)
&=
V^{\langle\mathcal{S}^{\prime-}_i;\infty\xi\rangle}
_{\infty,k,h}(\nu_{1},\nu_{2};\xi)
&&
\text{for\ } \xi\in
\mathcal{S}^{\prime-}_i \cup \Sigma^{(1,\infty)}_i
                        \cup \mathcal{S}^{\prime+}_i.
\end{aligned}
\end{displaymath}
\end{remark}

Rewriting the relations stated in Theorem \ref{thm:rel_cauchy}
in terms of
$V^{\langle\mathcal{S}^{\prime\pm}_i;ab\rangle}
_{\mathit{sing},k,h}(\nu_{1},\nu_{2};\xi)$,
we obtain
\begin{equation} \label{rel:inf-finite}
\begin{aligned}
e^{-\pi\sqrt{-1}\nu_{1}}
W^{\langle\mathcal{S}^{\prime\pm}_i;\eta_0\xi\rangle}
_{t'_i,k,h}(\nu_{1},\nu_{2};\xi)
&
+\frac{e^{\pm\pi\sqrt{-1}\nu_{1}}
       \Gamma(\nu_{1}+1)\Gamma(\lambda_{i,k}-\nu_{1})}
      {\Gamma(\lambda_{i,k}+1)}
V^{\langle\mathcal{S}^{\prime\pm}_i;t'_i \xi\rangle}
_{t'_i,k,h}(\nu_{1},\nu_{2};\xi)
\\
&-\frac{\Gamma(\lambda_{i,k}-\nu_{1})\Gamma(-\lambda_{i,k})}
     {\Gamma(-\nu_{1})}
V^{\langle\mathcal{S}^{\prime\pm}_i;\eta_0 t'_i\rangle}
_{t'_i,k,h}(\nu_{1},\nu_{2};\xi)
=0
\end{aligned}
\end{equation}
and
\begin{equation} \label{rel:inf-inf}
\begin{aligned}
W^{\langle\mathcal{S}^{\prime\pm}_i;\eta_0\xi\rangle}
_{\infty,k,h}(\nu_{1},\nu_{2};\xi)
&
+\frac{e^{\mp\pi\sqrt{-1}\nu_{1}}
       \Gamma(\nu_{1}+1)\Gamma(\nu_{2}-\mu_{k}+1)}
      {\Gamma(\nu_{1}+\nu_{2}-\mu_{k}+1)}
V^{\langle\mathcal{S}^{\prime\pm}_i;\infty\xi\rangle}
_{\infty,k,h}(\nu_{1},\nu_{2};\xi)
\\
&-\frac{\Gamma(\nu_{2}-\mu_{k}+1)\Gamma(\mu_{k}-\nu_{1}-\nu_{2})}
     {\Gamma(-\nu_{1})}
V^{\langle\mathcal{S}^{\prime\pm}_i;\eta_0 \infty\rangle}
_{\infty,k,h}(\nu_{1},\nu_{2};\xi)
=0,
\end{aligned}
\end{equation}
where the double-signs correspond.

Combining
(\ref{rel:inf-finite})
with
(\ref{asymp:W_eta0_t'i})
and (\ref{rel:inf-inf})
with
(\ref{asymp:W_eta0_inf}),
we obtain the following relations.

\begin{theorem} \label{thm:inv_lincomb_Veta0xi}
We have
\begin{equation} \label{exp:Weta0xi_by_Vt'ixi}
W^{\langle\mathcal{S}^{\prime\pm}_i;\eta_0\xi\rangle}
_{t'_i,k,h}(\nu_{1},\nu_{2};\xi)
=\frac{e^{\pi\sqrt{-1}\nu_{1}}
       \Gamma(\nu_{1}+1)\Gamma(-\lambda_{i,k})}
      {\Gamma(\nu_{1}-\lambda_{i,k}+1)}
V^{\langle\mathcal{S}^{\prime\pm}_i;t'_i \xi\rangle}
_{t'_i,k,h}(\nu_{1},\nu_{2};\xi)
+\mathrm{hol}(\xi-t'_i)
\end{equation}
for $\xi\in\mathcal{S}^{\prime\pm}_i$ near $t'_i$,
and
\begin{equation} \label{exp:Weta0xi_by_Vinfxi}
W^{\langle\mathcal{S}^{\prime\pm}_i;\eta_0\xi\rangle}
_{\infty,k,h}(\nu_{1},\nu_{2};\xi)
=-\frac{\Gamma(\nu_{1}+1)\Gamma(\mu_{k}-\nu_{1}-\nu_{2})}
       {\Gamma(\mu_{k}-\nu_{2})}
V^{\langle\mathcal{S}^{\prime\pm}_i;\infty\xi\rangle}
_{\infty,k,h}(\nu_{1},\nu_{2};\xi)
+\mathrm{hol}\left(\frac{1}{\xi-\eta_0}\right)
\end{equation}
for $\xi\in\mathcal{S}^{\prime\pm}_i$ near $\infty$.
\end{theorem}

\begin{proof}
Representing (\ref{asymp:W_eta0_t'i}) by
$V^{\langle\mathcal{S}^{\prime\pm}_i;\eta_0 t'_i\rangle}
_{t'_i,k,h}$  
and
$V^{\langle\mathcal{S}^{\prime\pm}_i;t'_i \xi\rangle}
_{t'_i,k,h}$,  
we have
\begin{displaymath}
V^{\langle\mathcal{S}^{\prime\pm}_i;\eta_0 t'_i\rangle}
_{t'_i,k,h}(\nu_{1},\nu_{2};\xi)
=
e^{\pm\pi\sqrt{-1}\lambda_{i,k}}
V^{\langle\mathcal{S}^{\prime\pm}_i;t'_i \xi\rangle}
_{t'_i,k,h}(\nu_{1},\nu_{2};\xi)
+\mathrm{hol}(\xi-t'_i).
\end{displaymath}
Substituting this into (\ref{rel:inf-finite}),
we obtain (\ref{exp:Weta0xi_by_Vt'ixi}).
Similarly,
representing (\ref{asymp:W_eta0_inf}) by
$V^{\langle\mathcal{S}^{\prime\pm}_i;\eta_0 \infty\rangle}
_{\infty,k,h}$  
and
$V^{\langle\mathcal{S}^{\prime\pm}_i;\infty\xi\rangle}
_{\infty,k,h}$,  
we have
\begin{displaymath}
V^{\langle\mathcal{S}^{\prime\pm}_i;\eta_0 \infty\rangle}
_{\infty,k,h}(\nu_{1},\nu_{2};\xi)
=
e^{\pm\pi\sqrt{-1}(\mu_k-\nu_{1}-\nu_{2})}
V^{\langle\mathcal{S}^{\prime\pm}_i;\infty\xi\rangle}
_{\infty,k,h}(\nu_{1},\nu_{2};\xi)
+\mathrm{hol}\left(\frac{1}{\xi-\eta_0}\right).
\end{displaymath}
Substituting this into (\ref{rel:inf-inf}),
we obtain (\ref{exp:Weta0xi_by_Vinfxi}).
\end{proof}

In the relations (\ref{rel:inf-finite}) and (\ref{rel:inf-inf})
interchanging $\nu_{1}$ and $\nu_{2}$ and then
using Theorem  \ref{thm:for_def_V},
we obtain
\begin{equation} \label{rel:inf-finite-change}
\begin{aligned}
e^{-\pi\sqrt{-1}\nu_{2}}
W^{\langle\mathcal{S}^{\prime\pm}_i;\eta_0\xi\rangle}
_{t'_i,k,h}(\nu_{2},\nu_{1};\xi)
&
+\frac{e^{\pm\pi\sqrt{-1}\nu_{2}}
       \Gamma(\nu_{2}+1)\Gamma(\lambda_{i,k}-\nu_{2})}
      {\Gamma(\lambda_{i,k}+1)}
V^{\langle\mathcal{S}^{\prime\pm}_i;t'_i \xi\rangle}
_{t'_i,k,h}(\nu_{1},\nu_{2};\xi)
\\
&-\frac{\Gamma(\lambda_{i,k}-\nu_{2})\Gamma(-\lambda_{i,k})}
     {\Gamma(-\nu_{2})}
V^{\langle\mathcal{S}^{\prime\pm}_i;\eta_0 t'_i\rangle}
_{t'_i,k,h}(\nu_{1},\nu_{2};\xi)
=0
\end{aligned}
\end{equation}
and
\begin{equation} \label{rel:inf-inf-change}
\begin{aligned}
W^{\langle\mathcal{S}^{\prime\pm}_i;\eta_0\xi\rangle}
_{\infty,k,h}(\nu_{2},\nu_{1};\xi)
&
+\frac{e^{\mp\pi\sqrt{-1}\nu_{2}}
       \Gamma(\nu_{2}+1)\Gamma(\nu_{1}-\mu_{k}+1)}
      {\Gamma(\nu_{1}+\nu_{2}-\mu_{k}+1)}
V^{\langle\mathcal{S}^{\prime\pm}_i;\infty\xi\rangle}
_{\infty,k,h}(\nu_{1},\nu_{2};\xi)
\\
&-\frac{\Gamma(\nu_{1}-\mu_{k}+1)\Gamma(\mu_{k}-\nu_{1}-\nu_{2})}
     {\Gamma(-\nu_{2})}
V^{\langle\mathcal{S}^{\prime\pm}_i;\eta_0 \infty\rangle}
_{\infty,k,h}(\nu_{1},\nu_{2};\xi)
=0,
\end{aligned}
\end{equation}
where the double-signs correspond.
Solving the simultaneous equations
(\ref{rel:inf-finite})
and
(\ref{rel:inf-finite-change})
with respect to
$V^{\langle\mathcal{S}^{\prime\pm}_i;t'_i \xi\rangle}
_{t'_i,k,h}$  
and
$V^{\langle\mathcal{S}^{\prime\pm}_i;\eta_0 t'_i\rangle}
_{t'_i,k,h}$,  
and
(\ref{rel:inf-inf})
and
(\ref{rel:inf-inf-change})
with respect to
$V^{\langle\mathcal{S}^{\prime\pm}_i;\infty\xi\rangle}
_{\infty,k,h}$  
and
$V^{\langle\mathcal{S}^{\prime\pm}_i;\eta_0\infty\rangle}
_{\infty,k,h}$,  
we have the following relations.

\begin{theorem} \label{thm:lincomb_Veta0xi}
For $\xi\in\mathcal{S}^{\prime\pm}_i$
we have
\begin{align} \label{exp:Vt'i_by_Veta0xi}
V^{\langle\mathcal{S}^{\prime\pm}_i;t'_i \xi\rangle}
_{t'_i,k,h}(\nu_{1},\nu_{2};\xi)
&=
-
\sum_{j=1}^{2}
\frac{e^{-\pi\sqrt{-1}\nu_{j}}\sin\pi\nu_{j'}\Gamma(\lambda_{i,k}+1)}
 {\sin\pi(\nu_{j'}-\nu_{j})\Gamma(\nu_{j}+1)\Gamma(\lambda_{i,k}-\nu_{j})}
W^{\langle\mathcal{S}^{\prime\pm}_i;\eta_0\xi\rangle}
_{t'_i,k,h}(\nu_{j},\nu_{j'};\xi),
\end{align}
and
\begin{align} \label{exp:Vinf_by_Veta0xi}
V^{\langle\mathcal{S}^{\prime\pm}_i;\infty\xi\rangle}
_{\infty,k,h}(\nu_{1},\nu_{2};\xi)
&=
-
\sum_{j=1}^{2}
\frac{\sin\pi\nu_{j'}\Gamma(\nu_{1}+\nu_{2}-\mu_{k}+1)}
     {\sin\pi(\nu_{j'}-\nu_{j})\Gamma(\nu_{j}+1)\Gamma(\nu_{j'}-\mu_{k}+1)}
W^{\langle\mathcal{S}^{\prime\pm}_i;\eta_0\xi\rangle}
_{\infty,k,h}(\nu_{j},\nu_{j'};\xi),
\end{align}
where the double-signs correspond
and $j'$ denotes the complement of $j$
in $\{1,2\}$, namely, $(j,j')=(1,2),(2,1)$.
\end{theorem}

Now, let us recall the connection formulas
(\ref{conn_formula:t'i_t'nu}) in $\mathcal{P}'$
and
(\ref{conn_formula:t'i_infty}) in $\check{\mathcal{P}}'$.
Note that we have determined the assignment of argument of
$\zeta-t'_i$ as
\begin{displaymath}
\arg(\zeta-t'_i)\in
\left\{\begin{aligned}
&(\theta'_i-2\pi,\,\theta'_i)
&& \text{for }\zeta\in\mathcal{P}',
\\
&(\theta'_i-\pi,\,\theta'_i+\pi)
&& \text{for }\zeta\in\check{\mathcal{P}}'
\end{aligned}\right.
\end{displaymath}
while
\begin{displaymath}
\arg(\zeta-t'_i)\in
\left\{\begin{aligned}
&(\theta'_i-\pi,\,\theta'_i)
&& \text{for }\zeta\in\mathcal{S}^{\prime-}_i,
\\
&(\theta'_i-2\pi,\,\theta'_i-\pi)
&& \text{for }\zeta\in\mathcal{S}^{\prime+}_i.
\end{aligned}\right.
\end{displaymath}

\begin{theorem} \label{thm:rel_Vti_ti+1}
For $\xi\in\mathcal{S}^{\prime-}_i\cap\mathcal{S}^{\prime+}_{i+1}$
we have
\begin{equation} \label{exp:Vt'ixi_by_Vt'i+1xi}
V^{\langle\mathcal{S}^{\prime-}_i;t'_i \xi\rangle}
_{t'_i,k,h}(\nu_{1},\nu_{2};\xi)
=\sum_{\tilde{k}=1}^{r_{i+1}} \sum_{\tilde{h}=1}^{\ell_{i+1,\tilde{k}}}
c_{t'_{i+1},\tilde{k},\tilde{h};t'_i,k,h}
V^{\langle\mathcal{S}^{\prime+}_{i+1};t'_{i+1} \xi\rangle}
_{t'_{i+1},\tilde{k},\tilde{h}}(\nu_{1},\nu_{2};\xi)
+\mathrm{hol}(\xi-t'_{i+1}),
\end{equation}
and for $\xi\in\mathcal{S}^{\prime+}_i\cap\mathcal{S}^{\prime-}_{i-1}$
\begin{equation} \label{exp:Vt'ixi_by_Vt'i-1xi}
V^{\langle\mathcal{S}^{\prime+}_i;t'_i \xi\rangle}
_{t'_i,k,h}(\nu_{1},\nu_{2};\xi)
=\sum_{\tilde{k}=1}^{r_{i-1}} \sum_{\tilde{h}=1}^{\ell_{i-1,\tilde{k}}}
c_{t'_{i-1},\tilde{k},\tilde{h};t'_i,k,h}
V^{\langle\mathcal{S}^{\prime-}_{i-1};t'_{i-1} \xi\rangle}
_{t'_{i-1},\tilde{k},\tilde{h}}(\nu_{1},\nu_{2};\xi)
+\mathrm{hol}(\xi-t'_{i-1}),
\end{equation}
where
$c_{t'_{i\pm1},\tilde{k},\tilde{h};t'_i,k,h}
=c_{t'_{i\pm1},\tilde{k},\tilde{h};t'_i,k,h}(0)$.
\end{theorem}

\begin{proof}
We give the proof only for (\ref{exp:Vt'ixi_by_Vt'i+1xi}).
Substituting (\ref{conn_formula:t'i_t'nu})
with $\nu=i+1$ and $\rho=-\nu_{j}-1$
into the integral
$W^{\langle\mathcal{S}^{\prime-}_i;\eta_0\xi\rangle}
_{t'_i,k,h}(\nu_{j},\nu_{j'};\xi)$,
we have
\begin{displaymath}
W^{\langle\mathcal{S}^{\prime-}_i;\eta_0\xi\rangle}
_{t'_i,k,h}(\nu_{j},\nu_{j'};\xi)
=
\sum_{\tilde{k}=1}^{r_{i+1}} \sum_{\tilde{h}=1}^{\ell_{i+1,\tilde{k}}}
c_{t'_{i+1},\tilde{k},\tilde{h};t'_i,k,h}(-\nu_{j}-1)
W^{\langle\mathcal{S}^{\prime+}_{i+1};\eta_0\xi\rangle}
_{t'_{i+1},\tilde{k},\tilde{h}}(\nu_{j},\nu_{j'};\xi)
+\mathrm{hol}(\xi-t'_{i+1}).
\end{displaymath}
Here
the holomorphy of the integral of the term $\mathrm{hol}(\zeta-t'_{i+1})$
can be proven
by representing the term as a linear combination
of $\tilde{w}_{t'_{\nu},k,h}(-\nu_{j}-1;\zeta)$
($\nu\ne i+1$) and then applying Theorem \ref{thm:asymp_W_t'i_inf}.
Substituting this relation 
into 
(\ref{exp:Vt'i_by_Veta0xi})
and then
substituting (\ref{conn_coef:t'i_t'nu})
with $\nu=i+1$ and $\rho=-\nu_{j}-1$ into the result,
we have
\begin{displaymath}
\begin{aligned}
&
V^{\langle\mathcal{S}^{\prime-}_i;t'_i \xi\rangle}
_{t'_i,k,h}(\nu_{1},\nu_{2};\xi)
\\
&=
\sum_{\tilde{k}=1}^{r_{i+1}} \sum_{\tilde{h}=1}^{\ell_{i+1,\tilde{k}}}
\sum_{j=1}^{2}
\frac{e^{-2\pi\sqrt{-1}\nu_{j}}\sin\pi\nu_{j'}
      \Gamma(\nu_{j}-\lambda_{i+1,\tilde{k}}+1)}
     {\sin\pi(\nu_{j'}-\nu_{j})
      \Gamma(\nu_{j}+1)\Gamma(-\lambda_{i+1,\tilde{k}})}
c_{t'_{i+1},\tilde{k},\tilde{h};t'_i,k,h} 
W^{\langle\mathcal{S}^{\prime+}_{i+1};\eta_0\xi\rangle}
_{t'_{i+1},\tilde{k},\tilde{h}}(\nu_{j},\nu_{j'};\xi)
+\mathrm{hol}(\xi-t'_{i+1}).
\end{aligned}
\end{displaymath}
Lastly,
substituting (\ref{exp:Weta0xi_by_Vt'ixi})
with $i$ replaced by $i+1$,
we have
\begin{displaymath}
V^{\langle\mathcal{S}^{\prime-}_i;t'_i \xi\rangle}
_{t'_i,k,h}(\nu_{1},\nu_{2};\xi)
=
\sum_{\tilde{k}=1}^{r_{i+1}} \sum_{\tilde{h}=1}^{\ell_{i+1,\tilde{k}}}
\sum_{j=1}^{2}
 \frac{e^{-\pi\sqrt{-1}\nu_{j}}\sin\pi\nu_{j'}}
      {\sin\pi(\nu_{j'}-\nu_{j})}
c_{t'_{i+1},\tilde{k},\tilde{h};t'_i,k,h}
V^{\langle\mathcal{S}^{\prime+}_{i+1};t'_{i+1} \xi\rangle}
_{t'_{i+1},\tilde{k},\tilde{h}}(\nu_{1},\nu_{2};\xi)
+\mathrm{hol}(\xi-t'_{i+1}),
\end{displaymath}
which implies (\ref{exp:Vt'ixi_by_Vt'i+1xi})
since the identity
\begin{equation} \label{identity:sin}
\sum_{j=1}^{2}
 \frac{e^{-\pi\sqrt{-1}\nu_{j}}\sin\pi\nu_{j'}}
      {\sin\pi(\nu_{j'}-\nu_{j})}
=1
\end{equation}
holds.
\end{proof}

\begin{theorem} \label{thm:rel_Vti_inf}
For $\xi\in\mathcal{S}^{\prime-}_i$
we have
\begin{equation} \label{exp:Vt'ixi_by_Vinfxi_S'-}
V^{\langle\mathcal{S}^{\prime-}_i;t'_i \xi\rangle}
_{t'_i,k,h}(\nu_{1},\nu_{2};\xi)
=
\sum_{\tilde{k}=1}^{q} \sum_{\tilde{h}=1}^{m_{\tilde{k}}}
\frac{\Gamma(\mu_{\tilde{k}}+1)
      \Gamma(\mu_{\tilde{k}}-\nu_{1}-\nu_{2})}
     {\Gamma(\mu_{\tilde{k}}-\nu_{1})
      \Gamma(\mu_{\tilde{k}}-\nu_{2})}
c_{\infty,\tilde{k},\tilde{h};t'_i,k,h}
V^{\langle\mathcal{S}^{\prime-}_i;\infty\xi\rangle}
_{\infty,\tilde{k},\tilde{h}}(\nu_{1},\nu_{2};\xi)
+\mathrm{hol}\left(\frac{1}{\xi-\eta_0}\right),
\end{equation}
and for $\xi\in\mathcal{S}^{\prime+}_i$
\begin{equation} \label{exp:Vt'ixi_by_Vinfxi_S'+}
\begin{aligned}
V^{\langle\mathcal{S}^{\prime+}_i;t'_i \xi\rangle}
_{t'_i,k,h}(\nu_{1},\nu_{2};\xi)
&=
e^{-2\pi\sqrt{-1}\lambda_{i,k}}
\sum_{\tilde{k}=1}^{q} \sum_{\tilde{h}=1}^{m_{\tilde{k}}}
\frac{\Gamma(\mu_{\tilde{k}}+1)
      \Gamma(\mu_{\tilde{k}}-\nu_{1}-\nu_{2})}
     {\Gamma(\mu_{\tilde{k}}-\nu_{1})
      \Gamma(\mu_{\tilde{k}}-\nu_{2})}
c_{\infty,\tilde{k},\tilde{h};t'_i,k,h}
V^{\langle\mathcal{S}^{\prime+}_i;\infty\xi\rangle}
_{\infty,\tilde{k},\tilde{h}}(\nu_{1},\nu_{2};\xi)
\\
&\phantom{{}={}}
+\mathrm{hol}\left(\frac{1}{\xi-\eta_0}\right),
\end{aligned}
\end{equation}
where
$c_{\infty,\tilde{k},\tilde{h};t'_i,k,h}
=c_{\infty,\tilde{k},\tilde{h};t'_i,k,h}(0)$.
\end{theorem}

\begin{proof}
First, we prove (\ref{exp:Vt'ixi_by_Vinfxi_S'-}).
Substituting (\ref{conn_formula:t'i_infty}) with $\rho=-\nu_{j}-1$
into the integral
$W^{\langle\mathcal{S}^{\prime-}_i;\eta_0\xi\rangle}
_{t'_i,k,h}(\nu_{j},\nu_{j'};\xi)$,
we have
\begin{displaymath}
W^{\langle\mathcal{S}^{\prime-}_i;\eta_0\xi\rangle}
_{t'_i,k,h}(\nu_{j},\nu_{j'};\xi)
=
\sum_{\tilde{k}=1}^{q} \sum_{\tilde{h}=1}^{m_{\tilde{k}}}
c_{\infty,\tilde{k},\tilde{h};t'_i,k,h}(-\nu_{j}-1)
W^{\langle\mathcal{S}^{\prime-}_i;\eta_0\xi\rangle}
_{\infty,\tilde{k},\tilde{h}}(\nu_{j},\nu_{j'};\xi).
\end{displaymath}
Substituting this relation into (\ref{exp:Vt'i_by_Veta0xi})
and then
substituting (\ref{conn_coef:t'i_infty}) with $\rho=-\nu_{j}-1$
into the result,
we obtain
\begin{displaymath}
V^{\langle\mathcal{S}^{\prime-}_i;t'_i \xi\rangle}
_{t'_i,k,h}(\nu_{1},\nu_{2};\xi)
=
-
\sum_{\tilde{k}=1}^{q} \sum_{\tilde{h}=1}^{m_{\tilde{k}}}
\sum_{j=1}^{2}
\frac{e^{-\pi\sqrt{-1}\nu_{j}}\sin\pi\nu_{j'}\Gamma(\mu_{\tilde{k}}+1)}
 {\sin\pi(\nu_{j'}-\nu_{j})\Gamma(\nu_{j}+1)\Gamma(\mu_{\tilde{k}}-\nu_{j})}
c_{\infty,\tilde{k},\tilde{h};t'_i,k,h}
W^{\langle\mathcal{S}^{\prime-}_i;\eta_0\xi\rangle}
_{\infty,\tilde{k},\tilde{h}}(\nu_{j},\nu_{j'};\xi).
\end{displaymath}
Lastly,
substituting (\ref{exp:Weta0xi_by_Vinfxi}),
we have
\begin{displaymath}
\begin{aligned}
V^{\langle\mathcal{S}^{\prime-}_i;t'_i \xi\rangle}
_{t'_i,k,h}(\nu_{1},\nu_{2};\xi)
&=
\sum_{\tilde{k}=1}^{q} \sum_{\tilde{h}=1}^{m_{\tilde{k}}}
\sum_{j=1}^{2}
\frac{e^{-\pi\sqrt{-1}\nu_{j}}\sin\pi\nu_{j'}}
 {\sin\pi(\nu_{j'}-\nu_{j})}
\frac{\Gamma(\mu_{\tilde{k}}+1)\Gamma(\mu_{k}-\nu_{1}-\nu_{2})}
       {\Gamma(\mu_{\tilde{k}}-\nu_{1})\Gamma(\mu_{\tilde{k}}-\nu_{2})}
c_{\infty,\tilde{k},\tilde{h};t'_i,k,h}
V^{\langle\mathcal{S}^{\prime-}_i;\infty\xi\rangle}
_{\infty,\tilde{k},\tilde{h}}(\nu_{1},\nu_{2};\xi)
\\
&\phantom{{}={}}
+\mathrm{hol}\left(\frac{1}{\xi-\eta_0}\right),
\end{aligned}
\end{displaymath}
which leads to (\ref{exp:Vt'ixi_by_Vinfxi_S'-}) by virtue of
the identity (\ref{identity:sin}).

\begin{figure}
 \centering
 \includegraphics{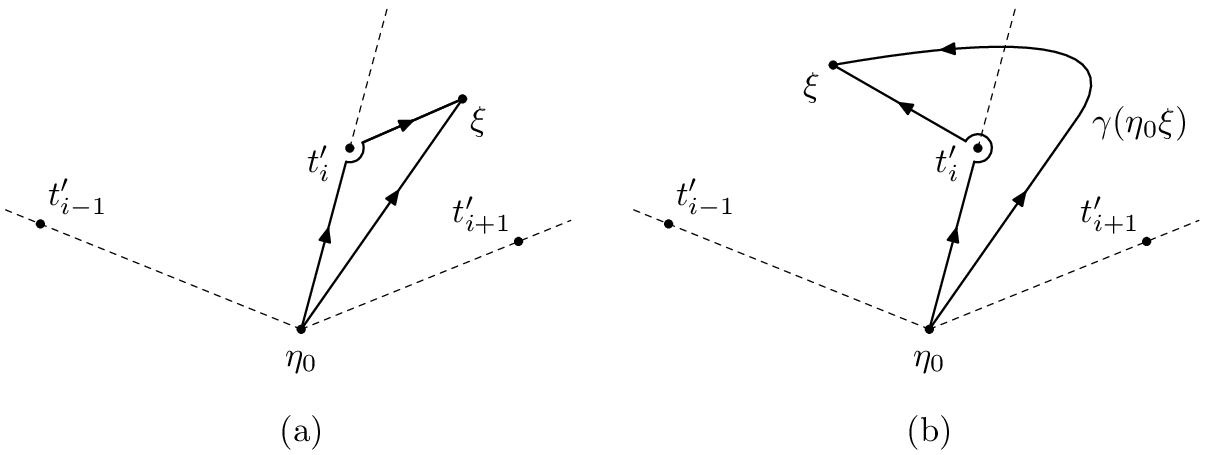}
 \caption{Deformation of the path $\overline{\eta_0 \xi}$}
 \label{fig:analyticconti_basesol}
\end{figure}

Next, we prove (\ref{exp:Vt'ixi_by_Vinfxi_S'+}).
Note that the assignment of argument of $\zeta-t'_i$
in $\mathcal{S}^{\prime+}_i$
differs from that in $\check{\mathcal{P}}'$.
We analytically continue
$W^{\langle\mathcal{S}^{\prime-}_i;\eta_0\xi\rangle}
_{t'_i,k,h}$
and
$W^{\langle\mathcal{S}^{\prime-}_i;\eta_0\xi\rangle}
_{\infty,k,h}$
across the ray $\Sigma^{(1,\infty)}_i$
into $\mathcal{S}^{\prime+}_i$.
The analytic continuation is found by deforming the path
of integration $\overline{\eta_0 \xi}$ to the path $\gamma(\eta_0\xi)$
indicated in Figure \ref{fig:analyticconti_basesol} (b),
and the relation
\begin{displaymath}
W^{\langle\mathcal{S}^{\prime-}_i;\gamma(\eta_0\xi)\rangle}
_{t'_i,k,h}(\nu_{j},\nu_{j'};\xi)
=\sum_{\tilde{k}=1}^{q} \sum_{\tilde{h}=1}^{m_{\tilde{k}}}
c_{\infty,\tilde{k},\tilde{h};t'_i,k,h}(-\nu_{j}-1)
W^{\langle\mathcal{S}^{\prime-}_i;\gamma(\eta_0\xi)\rangle}
_{\infty,\tilde{k},\tilde{h}}(\nu_{j},\nu_{j'};\xi)
\end{displaymath}
holds.
Here, from Figure \ref{fig:analyticconti_basesol} (b)
we see that the relation
\begin{displaymath}
W^{\langle\mathcal{S}^{\prime-}_i;\gamma(\eta_0\xi)\rangle}
_{t'_i,k,h}(\nu_{1},\nu_{2};\xi)
-
e^{2\pi\sqrt{-1}\lambda_{i,k}}
W^{\langle\mathcal{S}^{\prime+}_i;t'_i\xi\rangle}
_{t'_i,k,h}(\nu_{1},\nu_{2};\xi)
-
W^{\langle\mathcal{S}^{\prime-}_i;\eta_0t'_i\rangle}
_{t'_i,k,h}(\nu_{1},\nu_{2};\xi)
=0
\end{displaymath}
holds.  
Besides, from Theorem \ref{thm:asymp_Winf_xi_inf}
we see that the relation
\begin{displaymath}
W^{\langle\mathcal{S}^{\prime-}_i;\gamma(\eta_0\xi)\rangle}
_{\infty,k,h}(\nu_{1},\nu_{2};\xi)
-
e^{2\pi\sqrt{-1}\nu_{1}}
W^{\langle\mathcal{S}^{\prime+}_i;\infty\xi\rangle}
_{\infty,k,h}(\nu_{1},\nu_{2};\xi)
-
W^{\langle\mathcal{S}^{\prime-}_i;\eta_0 \infty\rangle}
_{\infty,k,h}(\nu_{1},\nu_{2};\xi)
=0
\end{displaymath}
holds.  
Using these relations instead of the first and the second relations
of Theorem \ref{thm:rel_cauchy},
we can establish the relation (\ref{exp:Vt'ixi_by_Vinfxi_S'+})
in a similar way to (\ref{exp:Vt'ixi_by_Vinfxi_S'-}).
\end{proof}


\section{Connection formulas}
\label{sec:connformulas}
Taking account of the transformation
(\ref{change:xitox}),
we see that the integrals
\begin{displaymath}
W^{\langle\mathcal{S}^{\prime\pm}_i;ab\rangle}
_{\mathit{sing},k,h}(\rho_{j},\rho_{j'};\eta_0+\frac{1}{x-t_{p+1}}),
\quad
V^{\langle\mathcal{S}^{\prime\pm}_i;ab\rangle}
_{\mathit{sing},k,h}(\rho_{1},\rho_{2};\eta_0+\frac{1}{x-t_{p+1}})
\end{displaymath}
become solutions of the system (\ref{sys:E2}).
We shall establish relations between these integrals
and the local solutions $U_{\mathit{sing},k,h}(x)$
defined in Section \ref{sec:localsol},
and then rewrite the connection formulas among
the integrals
$W^{\langle\mathcal{S}^{\prime\pm}_i;ab\rangle}
_{\mathit{sing},k,h}(\rho_{j},\rho_{j'};\xi)$
and
$V^{\langle\mathcal{S}^{\prime\pm}_i;ab\rangle}
_{\mathit{sing},k,h}(\rho_{1},\rho_{2};\xi)$
by means of $U_{\mathit{sing},k,h}(x)$.

From now on in the case that the endpoint $b$ of
the path $\overline{ab}$ is equal to $\xi$
we omit $\xi$ in the superscript, namely, we write
\begin{displaymath}
\begin{aligned}
W^{\langle\mathcal{S}^{\prime\pm}_i;\eta_0\rangle}
_{t'_i,k,h}
&=W^{\langle\mathcal{S}^{\prime\pm}_i;\eta_0\xi\rangle}
_{t'_i,k,h},
&
V^{\langle\mathcal{S}^{\prime\pm}_i;t'_i\rangle}
_{t'_i,k,h}
&=V^{\langle\mathcal{S}^{\prime\pm}_i;t'_i\xi\rangle}
_{t'_i,k,h},
\\
W^{\langle\mathcal{S}^{\prime\pm}_i;\eta_0\rangle}
_{\infty,k,h}
&=W^{\langle\mathcal{S}^{\prime\pm}_i;\eta_0\xi\rangle}
_{\infty,k,h},
&
V^{\langle\mathcal{S}^{\prime\pm}_i;\infty\rangle}
_{\infty,k,h}
&=V^{\langle\mathcal{S}^{\prime\pm}_i;\infty\xi\rangle}
_{\infty,k,h}.
\end{aligned}
\end{displaymath}
For $1\leq i\leq p$ we set
\begin{displaymath}
\begin{aligned}
\mathcal{S}^{+}_i
&=\{x\mid\theta_i<\arg(x-t_{p+1})<\min(\theta_{i+1}-\delta,\,\theta_i+\pi)\},
\\
\mathcal{S}^{-}_i
&=\{x\mid\max(\theta_{i-1}+\delta,\,\theta_i-\pi)<\arg(x-t_{p+1})<\theta_i\},
\end{aligned}
\end{displaymath}
where $\theta_0=\theta_{p+1}-2\pi$.
Note that
$\xi\in\mathcal{S}^{\prime-}_i$
(resp.\ $\mathcal{S}^{\prime+}_i$)
corresponds to
$x\in\mathcal{S}^{+}_i$
(resp.\ $\mathcal{S}^{-}_i$)
through the transformation (\ref{change:xitox}).

\subsection{Generic case}
First, we consider the connection formulas for the system (\ref{sys:E2})
in the case that
none of the $\rho_{j}$'s is an eigenvalue of the matrix $A$.
We assume the conditions
(\ref{assumption:E2_1})--(\ref{assumption:E2_4}),
(\ref{assumption:E2_0})
and
(\ref{assumption:underlying}).

\begin{proposition} \label{thm:V_Uinf}
For $1\leq i\leq p$ we have
\begin{equation} \label{relation:Wt'i_Uinf}
\begin{aligned}
W^{\langle\mathcal{S}^{\prime\pm}_i;\eta_0\rangle}
_{t'_i,k,h}(\rho_{j},\rho_{j'};\eta_0+\frac{1}{x-t_{p+1}})
=
 \frac{\Gamma(\rho_{j}+1)\Gamma(-\rho_{j'})}{\Gamma(\rho_{j}-\rho_{j'}+1)}
 \sum_{\tilde{h}=1}^{n}
 \gamma_{\tilde{h};t'_i,k,h}(-\rho_{j}-1)
 U_{\infty,j',\tilde{h}}(x)
\\
(1\leq j\leq 2, \ 1\leq k\leq r_i, \ 1\leq h\leq \ell_{i,k})
\end{aligned}
\end{equation}
for $x\in\mathcal{S}^{\mp}_i$,
where
\begin{displaymath}
\gamma_{\tilde{h};t'_i,k,h}(\rho)
=
[P^{-1}]_{\tilde{h}}(\eta_0 I_n-T')w_{t'_i,k,h}(\rho;\eta_0)
\quad
(1\leq \tilde{h}\leq n),
\end{displaymath}
and
\begin{equation} \label{relation:Winf_Uinf}
\begin{aligned}
W^{\langle\mathcal{S}^{\prime\pm}_i;\eta_0\rangle}
_{\infty,k,h}(\rho_{j},\rho_{j'};\eta_0+\frac{1}{x-t_{p+1}})
=
 \frac{\Gamma(\rho_{j}+1)\Gamma(-\rho_{j'})}{\Gamma(\rho_{j}-\rho_{j'}+1)}
 \sum_{\tilde{h}=1}^{n}
 \gamma^{\langle\mathcal{S}^{\prime\pm}_{i}\rangle}
       _{\tilde{h};\infty,k,h}(-\rho_{j}-1)
 U_{\infty,j',\tilde{h}}(x)
\\
(1\leq j\leq 2, \ 1\leq k\leq q, \ 1\leq h\leq m_{k})
\end{aligned}
\end{equation}
for $x\in\mathcal{S}^{\mp}_i$,
where
\begin{displaymath}
\gamma^{\langle\mathcal{S}^{\prime\pm}
 _{i}\rangle}_{\tilde{h};\infty,k,h}(\rho)
=
[P^{-1}]_{\tilde{h}}
(\eta_0 I_n-T')w_{\infty,k,h}|_{\mathcal{S}^{\prime\pm}_{i}}(\rho;\eta_0)
\quad
(1\leq \tilde{h}\leq n).
\end{displaymath}
\end{proposition}

\begin{proof}
Under the transformation (\ref{change:xitox}),
$\xi\longrightarrow\eta_0$ is equivalent to
$x\longrightarrow\infty$.
From Theorem \ref{thm:asymp_W_eta0_xi} we find
\begin{displaymath}
\begin{aligned}
&
\frac{\Gamma(\rho_{j}-\rho_{j'}+1)}{\Gamma(\rho_{j}+1)\Gamma(-\rho_{j'})}
W^{\langle\mathcal{S}^{\prime\pm}_i;\eta_0\rangle}
_{t'_i,k,h}(\rho_{j},\rho_{j'};\eta_0+\frac{1}{x-t_{p+1}})
\\
&=\left(x-t_{p+1}\right)^{\rho_{j'}}
\left\{
\begin{pmatrix}
P \\
\rho_{j'} I_n-A'
\end{pmatrix}
P^{-1}(\eta_0 I_n-T')
w_{t'_i,k,h}(-\rho_{j}-1;\eta_0)
 + O\left(\frac{1}{x-t_{p+1}}\right)
\right\}
\end{aligned}
\end{displaymath}
as $x\longrightarrow\infty$, $x\in\mathcal{S}^{\mp}_i$.
Here the initial term satisfies
\begin{displaymath}
\begin{pmatrix}
P \\
\rho_{j'} I_n-A'
\end{pmatrix}
P^{-1}(\eta_0 I_n-T')
w_{t'_i,k,h}(-\rho_{j}-1;\eta_0)
=
 \sum_{\tilde{h}=1}^{n}
 \gamma_{\tilde{h};t'_i,k,h}(-\rho_{j}-1)
 G_{\infty,j',\tilde{h}}(0),
\end{displaymath}
which implies (\ref{relation:Wt'i_Uinf}).
Similarly,
from Theorem \ref{thm:asymp_Winf_eta0_xi}
we obtain (\ref{relation:Winf_Uinf}).
\end{proof}

For $x\in\mathcal{S}^{\pm}_i$ we determine the assignment of argument of
$x-t_i$ as
\begin{equation} \label{arg:x-ti}
\begin{aligned}
 \arg(x-t_i) &\in (\theta_i,\,\theta_i+\pi)
 && \text{for\ }x\in\mathcal{S}^{+}_i, \\
 \arg(x-t_i) &\in (\theta_i-\pi,\,\theta_i)
 && \text{for\ }x\in\mathcal{S}^{-}_i.
  \end{aligned}
\end{equation}
Under the assignment (\ref{arg:xi-t'i}) and (\ref{arg:x-ti})
we have
\begin{displaymath}
\xi-t'_i
=\frac{e^{-\pi\sqrt{-1}}}{t_i-t_{p+1}}\cdot\frac{x-t_i}{x-t_{p+1}}.
\end{displaymath}

\begin{proposition} \label{thm:V_Uti}
For $1\leq i\leq p$ we have
\begin{equation} \label{relation:Vt'i_Uti}
\begin{aligned}
V^{\langle\mathcal{S}^{\prime\pm}_i;t'_i \rangle}
_{t'_i,k,h}(\rho_{1},\rho_{2};\eta_0+\frac{1}{x-t_{p+1}})
=
e^{-\pi\sqrt{-1}\lambda_{i,k}}
(t_i-t_{p+1})^{\rho_{1}+\rho_{2}-2\lambda_{i,k}}
U_{t_i,k,h}(x)
\quad 
(1\leq k\leq r_i, \ 1\leq h\leq \ell_{i,k})
\end{aligned}
\end{equation}
for $x\in\mathcal{S}^{\mp}_i$.
\end{proposition}

\begin{proof}
Note that
\begin{equation} \label{relation:xi-t'i_x-ti}
\xi-t'_i
=
\frac{e^{-\pi\sqrt{-1}}}{(t_i-t_{p+1})^2}
\cdot(x-t_i)
\cdot\{1+O(x-t_i)\}
\end{equation}
as $x\longrightarrow t_i$,
where
\begin{displaymath}
\arg\{1+O(x-t_i)\}
=\arg\left(\frac{t_i-t_{p+1}}{x-t_{p+1}}\right)
\in(-\pi,\,\pi).
\end{displaymath}
Substituting (\ref{relation:xi-t'i_x-ti})
into (\ref{asymp:W_xi_t'i}),
we find
\begin{displaymath}
\begin{aligned}
V^{\langle\mathcal{S}^{\prime\pm}_i;t'_i \rangle}
_{t'_i,k,h}(\rho_{1},\rho_{2};\eta_0+\frac{1}{x-t_{p+1}})
&=
\left(\frac{1}{t_i-t_{p+1}}\right)^{-\rho_{1}-\rho_{2}}
\left(
\frac{e^{-\pi\sqrt{-1}}}{(t_i-t_{p+1})^2}
\right)^{\lambda_{i,k}}
\\
&\phantom{{}={}}
{}\times
(x-t_i)^{\lambda_{i,k}}
\left\{\varepsilon_{2n}
(n_1+\cdots+n_{i-1}+\ell_{i,1}+\cdots+\ell_{i,k-1}+h)
+O(x-t_i)\right\}
\end{aligned}
\end{displaymath}
as $x\longrightarrow t_i$, $x\in\mathcal{S}^{\mp}_i$.
This implies (\ref{relation:Vt'i_Uti}).
\end{proof}

\begin{proposition} \label{thm:V_Utp+1}
We have
\begin{equation} \label{relation:Vinf_Utp+1}
V^{\langle\mathcal{S}^{\prime\pm}_i;\infty\rangle}
_{\infty,k,h}(\rho_{1},\rho_{2};\eta_0+\frac{1}{x-t_{p+1}})
= U_{t_{p+1},k,h}(x)
\quad (1\leq k\leq q, \ 1\leq h\leq m_{k})
\end{equation}
for $x\in\mathcal{S}^{\mp}_i$.
\end{proposition}

\begin{proof}
Note that $\xi\longrightarrow\infty$ is equivalent to
$x\longrightarrow t_{p+1}$.
From Theorem \ref{thm:asymp_Winf_xi_inf}
we obtain
\begin{displaymath}
V^{\langle\mathcal{S}^{\prime\pm}_i;\infty\rangle}
_{\infty,k,h}(\rho_{1},\rho_{2};\eta_0+\frac{1}{x-t_{p+1}})
=
\left(x-t_{p+1}\right)^{\rho_{1}+\rho_{2}-\mu_{k}}
\{
\varepsilon_{2n}(n+m_1+\cdots+m_{k-1}+h)
+O \left(x-t_{p+1}\right)
\}
\end{displaymath}
as $x\longrightarrow t_{p+1}$, $x\in\mathcal{S}^{\mp}_i$.
This implies (\ref{relation:Vinf_Utp+1}).
\end{proof}

Combining Theorems
\ref{thm:lincomb_Veta0xi},
\ref{thm:rel_Vti_ti+1} and
\ref{thm:rel_Vti_inf}
with Propositions
\ref{thm:V_Uinf},
\ref{thm:V_Utp+1} and
\ref{thm:V_Uti},
we obtain the following conclusions.

\begin{theorem}
For $1\leq i\leq p$ the coefficients
$C_{\infty,j,\tilde{h};t_i,k,h}$
in the connection formula
\begin{equation} \label{conn_form:Ui_Uinf}
U_{t_i,k,h}(x)
=
\sum_{j=1}^{2}
\sum_{\tilde{h}=1}^{n}
C_{\infty,j,\tilde{h};t_i,k,h}
U_{\infty,j,\tilde{h}}(x)
\quad
(1\leq k\leq r_i, \ 1\leq h\leq \ell_{i,k})
\end{equation}
for $x\in\mathcal{S}^{\pm}_{i}$
are given by
\begin{equation} \label{conn_coef:Ui_Uinf}
\begin{aligned}
C_{\infty,j,\tilde{h};t_i,k,h}
=
e^{\pi\sqrt{-1}(\lambda_{i,k}-\rho_{j'})}
(t_i-t_{p+1})^{2\lambda_{i,k}-\rho_{1}-\rho_{2}}
 \frac{\Gamma(\rho_{j}-\rho_{j'})\Gamma(\lambda_{i,k}+1)}
      {\Gamma(\rho_{j}+1)\Gamma(\lambda_{i,k}-\rho_{j'})}
 \gamma_{\tilde{h};t'_i,k,h}(-\rho_{j'}-1)
\\
(1\leq j\leq 2, \ 1\leq \tilde{h}\leq n, \
 1\leq k\leq r_i, \ 1\leq h\leq \ell_{i,k}).
\end{aligned}
\end{equation}
Besides,
the coefficients
$C^{\langle\mathcal{S}^{\pm}_{i}\rangle}
  _{\infty,j,\tilde{h};t_{p+1},k,h}$
in the connection formula
\begin{equation} \label{conn_form:Up+1_Uinf}
U_{t_{p+1},k,h}(x)
=
\sum_{j=1}^{2}
\sum_{\tilde{h}=1}^{n}
C^{\langle\mathcal{S}^{\pm}_{i}\rangle}
 _{\infty,j,\tilde{h};t_{p+1},k,h}
U_{\infty,j,\tilde{h}}(x)
\quad
(1\leq k\leq q, \ 1\leq h\leq m_{k})
\end{equation}
for $x\in\mathcal{S}^{\pm}_{i}$
are given by
\begin{equation} \label{conn_coef:Up+1_Uinf}
\begin{aligned}
C^{\langle\mathcal{S}^{\pm}_{i}\rangle}
 _{\infty,j,\tilde{h};t_{p+1},k,h}
=
\frac{\Gamma(\rho_{j}-\rho_{j'})\Gamma(\rho_{1}+\rho_{2}-\mu_{k}+1)}
      {\Gamma(\rho_{j}+1)\Gamma(\rho_{j}-\mu_{k}+1)}
\gamma^{\langle\mathcal{S}^{\prime\mp}_{i}\rangle}
      _{\tilde{h};\infty,k,h}(-\rho_{j'}-1)
\\
(1\leq j\leq 2, \ 1\leq \tilde{h}\leq n, \
 1\leq k\leq q, \ 1\leq h\leq m_{k}).
\end{aligned}
\end{equation}
\end{theorem}

\begin{proof}
Substituting (\ref{relation:Wt'i_Uinf})
and (\ref{relation:Vt'i_Uti}) into
(\ref{exp:Vt'i_by_Veta0xi})${}_{\nu_{1}=\rho_{1},\nu_{2}=\rho_{2}}$
and then interchanging $j$ and $j'$,
we obtain (\ref{conn_form:Ui_Uinf}) with (\ref{conn_coef:Ui_Uinf}).
Similarly,
substituting (\ref{relation:Winf_Uinf})
and (\ref{relation:Vinf_Utp+1}) into
(\ref{exp:Vinf_by_Veta0xi})${}_{\nu_{1}=\rho_{1},\nu_{2}=\rho_{2}}$
and then interchanging $j$ and $j'$,
we obtain (\ref{conn_form:Up+1_Uinf}) with (\ref{conn_coef:Up+1_Uinf}).
\end{proof}

\begin{theorem}
For $1\leq i\leq p-1$ the coefficients
$C_{t_{i+1},\tilde{k},\tilde{h};t_i,k,h}$
in the connection formula
\begin{equation} \label{conn_form:Ui_Ui+1}
U_{t_i,k,h}(x)
=
\sum_{\tilde{k}=1}^{r_{i+1}} \sum_{\tilde{h}=1}^{\ell_{i+1,\tilde{k}}}
C_{t_{i+1},\tilde{k},\tilde{h};t_i,k,h}
U_{t_{i+1},\tilde{k},\tilde{h}}(x)
+\mathrm{hol}(x-t_{i+1})
\quad
(1\leq k\leq r_i, \ 1\leq h\leq \ell_{i,k})
\end{equation}
for $x\in\mathcal{S}^{+}_{i}\cup
         \mathcal{S}^{-}_{i+1}$
are given by
\begin{equation} \label{conn_coef:Ui_Ui+1}
\begin{aligned}
C_{t_{i+1},\tilde{k},\tilde{h};t_i,k,h}
=
\frac{e^{\pi\sqrt{-1}\lambda_{i,k}}
      (t_i-t_{p+1})^{2\lambda_{i,k}-\rho_{1}-\rho_{2}}}
     {e^{\pi\sqrt{-1}\lambda_{i+1,\tilde{k}}}
      (t_{i+1}-t_{p+1})^{2\lambda_{i+1,\tilde{k}}-\rho_{1}-\rho_{2}}}
c_{t'_{i+1},\tilde{k},\tilde{h};t'_i,k,h}
\\
(1\leq \tilde{k}\leq r_{i+1}, \ 1\leq \tilde{h}\leq \ell_{i+1,\tilde{k}}, \
1\leq k\leq r_i, \ 1\leq h\leq \ell_{i,k}).
\end{aligned}
\end{equation}
Besides,
for $2\leq i\leq p$ the coefficients
$C_{t_{i-1},\tilde{k},\tilde{h};t_i,k,h}$
in the connection formula
\begin{equation} \label{conn_form:Ui_Ui-1}
U_{t_i,k,h}(x)
=
\sum_{\tilde{k}=1}^{r_{i-1}} \sum_{\tilde{h}=1}^{\ell_{i-1,\tilde{k}}}
C_{t_{i-1},\tilde{k},\tilde{h};t_i,k,h}
U_{t_{i-1},\tilde{k},\tilde{h}}(x)
+\mathrm{hol}(x-t_{i-1})
\quad
(1\leq k\leq r_i, \ 1\leq h\leq \ell_{i,k})
\end{equation}
for $x\in\mathcal{S}^{-}_{i}\cup
         \mathcal{S}^{+}_{i-1}$
are given by
\begin{equation} \label{conn_coef:Ui_Ui-1}
\begin{aligned}
C_{t_{i-1},\tilde{k},\tilde{h};t_i,k,h}
=
\frac{e^{\pi\sqrt{-1}\lambda_{i,k}}
      (t_i-t_{p+1})^{2\lambda_{i,k}-\rho_{1}-\rho_{2}}}
     {e^{\pi\sqrt{-1}\lambda_{i-1,\tilde{k}}}
      (t_{i-1}-t_{p+1})^{2\lambda_{i-1,\tilde{k}}-\rho_{1}-\rho_{2}}}
c_{t'_{i-1},\tilde{k},\tilde{h};t'_i,k,h}
\\
(1\leq \tilde{k}\leq r_{i-1}, \ 1\leq \tilde{h}\leq \ell_{i-1,\tilde{k}}, \
1\leq k\leq r_i, \ 1\leq h\leq \ell_{i,k}).
\end{aligned}
\end{equation}
\end{theorem}

\begin{proof}
Substituting (\ref{relation:Vt'i_Uti})
and the same expression with $i$ replaced by $i+1$ into
(\ref{exp:Vt'ixi_by_Vt'i+1xi})${}_{\nu_{1}=\rho_{1},\nu_{2}=\rho_{2}}$,
we obtain (\ref{conn_form:Ui_Ui+1}) with (\ref{conn_coef:Ui_Ui+1}).
Similarly,
substituting (\ref{relation:Vt'i_Uti})
and the same expression with $i$ replaced by $i-1$ into
(\ref{exp:Vt'ixi_by_Vt'i-1xi})${}_{\nu_{1}=\rho_{1},\nu_{2}=\rho_{2}}$,
we obtain (\ref{conn_form:Ui_Ui-1}) with (\ref{conn_coef:Ui_Ui-1}).
\end{proof}

\begin{theorem}
For $1\leq i\leq p$ the coefficients
$C^{\pm}_{t_{p+1},\tilde{k},\tilde{h};t_i,k,h}$
in the connection formula
\begin{equation} \label{conn_form:Ui_Up+1}
U_{t_i,k,h}(x)
=
\sum_{\tilde{k}=1}^{q} \sum_{\tilde{h}=1}^{m_{\tilde{k}}}
C^{\pm}_{t_{p+1},\tilde{k},\tilde{h};t_i,k,h}
U_{t_{p+1},\tilde{k},\tilde{h}}(x)
+\mathrm{hol}\left(x-t_{p+1}\right)
\quad
(1\leq k\leq r_i, \ 1\leq h\leq \ell_{i,k})
\end{equation}
for $x\in\mathcal{S}^{\pm}_{i}$
are given by
\begin{equation} \label{conn_coef:Ui_Up+1}
\begin{aligned}
C^{\pm}_{t_{p+1},\tilde{k},\tilde{h};t_i,k,h}
=
e^{\pm\pi\sqrt{-1}\lambda_{i,k}}
(t_i-t_{p+1})^{2\lambda_{i,k}-\rho_{1}-\rho_{2}}
\frac{\Gamma(\mu_{\tilde{k}}+1)
      \Gamma(\mu_{\tilde{k}}-\rho_{1}-\rho_{2})}
     {\Gamma(\mu_{\tilde{k}}-\rho_{1})
      \Gamma(\mu_{\tilde{k}}-\rho_{2})}
c_{\infty,\tilde{k},\tilde{h};t'_i,k,h}
\\
(1\leq \tilde{k}\leq q, \ 1\leq \tilde{h}\leq m_{\tilde{k}}, \
1\leq k\leq r_i, \ 1\leq h\leq \ell_{i,k}),
\end{aligned}
\end{equation}
where the double-signs correspond.
\end{theorem}

\begin{proof}
Substituting (\ref{relation:Vt'i_Uti})
and (\ref{relation:Vinf_Utp+1}) into
(\ref{exp:Vt'ixi_by_Vinfxi_S'-})${}_{\nu_{1}=\rho_{1},\nu_{2}=\rho_{2}}$
for $x\in\mathcal{S}^{+}_i$ and
(\ref{exp:Vt'ixi_by_Vinfxi_S'+})${}_{\nu_{1}=\rho_{1},\nu_{2}=\rho_{2}}$
for $x\in\mathcal{S}^{-}_i$,
we obtain (\ref{conn_form:Ui_Up+1}) with (\ref{conn_coef:Ui_Up+1}).
\end{proof}

\begin{remark}
The quantities
$\gamma_{\tilde{h};t'_i,k,h}(-\rho_{j'}-1)$
and
$\gamma^{\langle\mathcal{S}^{\prime\pm}_{i}\rangle}
       _{\tilde{h};\infty,k,h}(-\rho_{j'}-1)$
do not depend on $\eta_0$,
since
$u(\sigma)=\left((\sigma+\eta_0)I_n-T'\right)w(\rho;\sigma+\eta_0)$
satisfies the system
\begin{displaymath}
\frac{du}{d\sigma}
=((\rho+1) I_n+A)\left(\sigma I_n-(T-t_{p+1}I_n)^{-1}\right)^{-1}u,
\end{displaymath}
which does not depend on $\rho$,
and $\left(\eta_0I_n-T'\right)w(\rho;\eta_0)=u(0)$.
The coefficients
$c_{t'_{i\pm1},\tilde{k},\tilde{h};t'_i,k,h}$
and
$c_{\infty,\tilde{k},\tilde{h};t'_i,k,h}$
do not depend on $\eta_0$ also.
\end{remark}

\subsection{Reducible case (i)}
Next, we consider the connection formulas for the system (\ref{sys:E1}).
We assume the conditions
(\ref{assumption:E2_1})--(\ref{assumption:E2_4}),
(\ref{assumption:E1_0})
and
(\ref{assumption:underlying}).
Thanks to Proposition \ref{prop:Wvanish},
we obtain the following results.

\begin{proposition}
For $1\leq i\leq p$ we have
\begin{alignat}{2}
W^{\langle\mathcal{S}^{\prime\pm}_i;\eta_0\rangle}
_{t'_i,k,h}(\rho_{1},\mu_{q};\eta_0+\frac{1}{x-t_{p+1}})
&=
 \frac{\Gamma(\rho_{1}+1)\Gamma(-\mu_{q})}{\Gamma(\rho_{1}-\mu_{q}+1)}
 \sum_{\tilde{h}=1}^{n}
 \gamma_{\tilde{h};t'_i,k,h}(-\rho_{1}-1)
 \begin{pmatrix}
 U'_{\infty,\mu_{q},\tilde{h}}(x)
 \\ 0_{m_{q}} \end{pmatrix} \notag
\\
 &&\llap{$(1\leq k\leq r_i, \ 1\leq h\leq \ell_{i,k}),$} \notag
\\
W^{\langle\mathcal{S}^{\prime\pm}_i;\eta_0\rangle}
_{t'_i,k,h}(\mu_{q},\rho_{1};\eta_0+\frac{1}{x-t_{p+1}})
&=
 \frac{\Gamma(\mu_{q}+1)\Gamma(-\rho_{1})}{\Gamma(\mu_{q}-\rho_{1}+1)}
 \sum_{\tilde{h}=1}^{n-m_{q}}
 \gamma_{\tilde{h};t'_i,k,h}(-\mu_{q}-1)
 \begin{pmatrix}
 U'_{\infty,1,\tilde{h}}(x)
 \\ 0_{m_{q}} \end{pmatrix} \notag
\\
 &&\llap{$(1\leq k\leq r_i, \ 1\leq h\leq \ell_{i,k})$} \notag
\end{alignat}
for $x\in\mathcal{S}^{\mp}_i$,
and
\begin{alignat}{2}
W^{\langle\mathcal{S}^{\prime\pm}_i;\eta_0\rangle}
_{\infty,k,h}(\rho_{1},\mu_{q};\eta_0+\frac{1}{x-t_{p+1}})
&=
 \frac{\Gamma(\rho_{1}+1)\Gamma(-\mu_{q})}{\Gamma(\rho_{1}-\mu_{q}+1)}
 \sum_{\tilde{h}=1}^{n}
 \gamma^{\langle\mathcal{S}^{\prime\pm}_{i}\rangle}
       _{\tilde{h};\infty,k,h}(-\rho_{1}-1)
 \begin{pmatrix}
 U'_{\infty,\mu_{q},\tilde{h}}(x)
 \\ 0_{m_{q}} \end{pmatrix} \notag
\\
 &&\llap{$(1\leq k\leq q, \ 1\leq h\leq m_{k}),$} \notag
\\
W^{\langle\mathcal{S}^{\prime\pm}_i;\eta_0\rangle}
_{\infty,k,h}(\mu_{q},\rho_{1};\eta_0+\frac{1}{x-t_{p+1}})
&=
 \frac{\Gamma(\mu_{q}+1)\Gamma(-\rho_{1})}{\Gamma(\mu_{q}-\rho_{1}+1)}
 \sum_{\tilde{h}=1}^{n-m_{q}}
 \gamma^{\langle\mathcal{S}^{\prime\pm}_{i}\rangle}
       _{\tilde{h};\infty,k,h}(-\mu_{q}-1)
 \begin{pmatrix}
 U'_{\infty,1,\tilde{h}}(x)
 \\ 0_{m_{q}} \end{pmatrix} \notag
\\
 &&\llap{$(1\leq k\leq q-1, \ 1\leq h\leq m_{k})$} \notag
\end{alignat}
for $x\in\mathcal{S}^{\mp}_i$.
\end{proposition}

\begin{proposition}
For $1\leq i\leq p$ we have
\begin{displaymath}
V^{\langle\mathcal{S}^{\prime\pm}_i;t'_i \rangle}
_{t'_i,k,h}(\rho_{1},\mu_{q};\eta_0+\frac{1}{x-t_{p+1}})
=
e^{-\pi\sqrt{-1}\lambda_{i,k}}
(t_i-t_{p+1})^{\rho_{1}+\mu_{q}-2\lambda_{i,k}}
 \begin{pmatrix}
 U'_{t_i,k,h}(x)
 \\ 0_{m_{q}} \end{pmatrix}
 \quad
(1\leq k\leq r_i, \ 1\leq h\leq \ell_{i,k})
\end{displaymath}
for $x\in\mathcal{S}^{\mp}_i$.
\end{proposition}

\begin{proposition}
We have
\begin{displaymath}
V^{\langle\mathcal{S}^{\prime\pm}_i;\infty\rangle}
_{\infty,k,h}(\rho_{1},\mu_{q};\eta_0+\frac{1}{x-t_{p+1}})
=
 \begin{pmatrix}
   U'_{t_{p+1},k,h}(x)
 \\ 0_{m_{q}} \end{pmatrix}
\quad
(1\leq k\leq q-1, \ 1\leq h\leq m_{k})
\end{displaymath}
for $x\in\mathcal{S}^{\mp}_i$.
\end{proposition}

\begin{theorem}
For $1\leq i\leq p$ the coefficients
$C'_{\infty,1,\tilde{h};t_i,k,h}$
and
$C'_{\infty,\mu_{q},\tilde{h};t_i,k,h}$
in the connection formula
\begin{displaymath}
U'_{t_i,k,h}(x)
=
\sum_{\tilde{h}=1}^{n-m_{q}}
C'_{\infty,1,\tilde{h};t_i,k,h}
U'_{\infty,1,\tilde{h}}(x)
+
\sum_{\tilde{h}=1}^{n}
C'_{\infty,\mu_{q},\tilde{h};t_i,k,h}
U'_{\infty,\mu_{q},\tilde{h}}(x)
\quad
(1\leq k\leq r_i, \ 1\leq h\leq \ell_{i,k})
\end{displaymath}
for $x\in\mathcal{S}^{\pm}_{i}$
are given by
\begin{alignat}{2}
C'_{\infty,1,\tilde{h};t_i,k,h}
&=
e^{\pi\sqrt{-1}(\lambda_{i,k}-\mu_{q})}
(t_i-t_{p+1})^{2\lambda_{i,k}-\rho_{1}-\mu_{q}}
 \frac{\Gamma(\rho_{1}-\mu_{q})\Gamma(\lambda_{i,k}+1)}
      {\Gamma(\rho_{1}+1)\Gamma(\lambda_{i,k}-\mu_{q})}
 \gamma_{\tilde{h};t'_i,k,h}(-\mu_{q}-1) \notag
\\
 &&\llap{$(1\leq \tilde{h}\leq n-m_{q}, \
           1\leq k\leq r_i, \ 1\leq h\leq \ell_{i,k}),$} \notag
\\
C'_{\infty,\mu_{q},\tilde{h};t_i,k,h}
&=
e^{\pi\sqrt{-1}(\lambda_{i,k}-\rho_{1})}
(t_i-t_{p+1})^{2\lambda_{i,k}-\rho_{1}-\mu_{q}}
 \frac{\Gamma(\mu_{q}-\rho_{1})\Gamma(\lambda_{i,k}+1)}
      {\Gamma(\mu_{q}+1)\Gamma(\lambda_{i,k}-\rho_{1})}
 \gamma_{\tilde{h};t'_i,k,h}(-\rho_{1}-1) \notag
\\
 &&\llap{$(1\leq \tilde{h}\leq n, \
           1\leq k\leq r_i, \ 1\leq h\leq \ell_{i,k}).$} \notag
\end{alignat}
Besides,
the coefficients
$C^{\prime\langle\mathcal{S}^{\pm}_{i}\rangle}
  _{\infty,1,\tilde{h};t_{p+1},k,h}$
and
$C^{\prime\langle\mathcal{S}^{\pm}_{i}\rangle}
  _{\infty,\mu_{q},\tilde{h};t_{p+1},k,h}$
in the connection formula
\begin{displaymath}
U'_{t_{p+1},k,h}(x)
=
\sum_{\tilde{h}=1}^{n-m_{q}}
C^{\prime\langle\mathcal{S}^{\pm}_{i}\rangle}
 _{\infty,1,\tilde{h};t_{p+1},k,h}
U'_{\infty,1,\tilde{h}}(x)
+
\sum_{\tilde{h}=1}^{n}
C^{\prime\langle\mathcal{S}^{\pm}_{i}\rangle}
 _{\infty,\mu_{q},\tilde{h};t_{p+1},k,h}
U'_{\infty,\mu_{q},\tilde{h}}(x)
\quad
(1\leq k\leq q-1, \ 1\leq h\leq m_{k})
\end{displaymath}
for $x\in\mathcal{S}^{\pm}_{i}$
are given by
\begin{alignat}{2}
C^{\prime\langle\mathcal{S}^{\pm}_{i}\rangle}
 _{\infty,1,\tilde{h};t_{p+1},k,h}
&=
\frac{\Gamma(\rho_{1}-\mu_{q})\Gamma(\rho_{1}+\mu_{q}-\mu_{k}+1)}
      {\Gamma(\rho_{1}+1)\Gamma(\rho_{1}-\mu_{k}+1)}
\gamma^{\langle\mathcal{S}^{\prime\mp}_{i}\rangle}
      _{\tilde{h};\infty,k,h}(-\mu_{q}-1) \notag
\\
 &&\llap{$(1\leq \tilde{h}\leq n-m_{q}, \
           1\leq k\leq q-1, \ 1\leq h\leq m_{k}),$} \notag
\\
C^{\prime\langle\mathcal{S}^{\pm}_{i}\rangle}
 _{\infty,\mu_{q},\tilde{h};t_{p+1},k,h}
&=
\frac{\Gamma(\mu_{q}-\rho_{1})\Gamma(\rho_{1}+\mu_{q}-\mu_{k}+1)}
      {\Gamma(\mu_{q}+1)\Gamma(\mu_{q}-\mu_{k}+1)}
\gamma^{\langle\mathcal{S}^{\prime\mp}_{i}\rangle}
      _{\tilde{h};\infty,k,h}(-\rho_{1}-1) \notag
\\
 &&\llap{$(1\leq \tilde{h}\leq n, \
           1\leq k\leq q-1, \ 1\leq h\leq m_{k}).$} \notag
\end{alignat}
\end{theorem}

\begin{theorem}
For $1\leq i\leq p-1$ the coefficients
$C'_{t_{i+1},\tilde{k},\tilde{h};t_i,k,h}$
in the connection formula
\begin{displaymath}
U'_{t_i,k,h}(x)
=
\sum_{\tilde{k}=1}^{r_{i+1}} \sum_{\tilde{h}=1}^{\ell_{i+1,\tilde{k}}}
C'_{t_{i+1},\tilde{k},\tilde{h};t_i,k,h}
U'_{t_{i+1},\tilde{k},\tilde{h}}(x)
+\mathrm{hol}(x-t_{i+1})
\quad
(1\leq k\leq r_i, \ 1\leq h\leq \ell_{i,k})
\end{displaymath}
for $x\in\mathcal{S}^{+}_{i}\cup
         \mathcal{S}^{-}_{i+1}$
are given by
\begin{displaymath}
\begin{aligned}
C'_{t_{i+1},\tilde{k},\tilde{h};t_i,k,h}
=
\frac{e^{\pi\sqrt{-1}\lambda_{i,k}}
      (t_i-t_{p+1})^{2\lambda_{i,k}-\rho_{1}-\mu_{q}}}
     {e^{\pi\sqrt{-1}\lambda_{i+1,\tilde{k}}}
      (t_{i+1}-t_{p+1})^{2\lambda_{i+1,\tilde{k}}-\rho_{1}-\mu_{q}}}
c_{t'_{i+1},\tilde{k},\tilde{h};t'_i,k,h}
\\
(1\leq \tilde{k}\leq r_{i+1}, \ 1\leq \tilde{h}\leq \ell_{i+1,\tilde{k}}, \
1\leq k\leq r_i, \ 1\leq h\leq \ell_{i,k}).
\end{aligned}
\end{displaymath}
Besides,
for $2\leq i\leq p$ the coefficients
$C'_{t_{i-1},\tilde{k},\tilde{h};t_i,k,h}$
in the connection formula
\begin{displaymath}
U'_{t_i,k,h}(x)
=
\sum_{\tilde{k}=1}^{r_{i-1}} \sum_{\tilde{h}=1}^{\ell_{i-1,\tilde{k}}}
C'_{t_{i-1},\tilde{k},\tilde{h};t_i,k,h}
U'_{t_{i-1},\tilde{k},\tilde{h}}(x)
+\mathrm{hol}(x-t_{i-1})
\quad
(1\leq k\leq r_i, \ 1\leq h\leq \ell_{i,k})
\end{displaymath}
for $x\in\mathcal{S}^{-}_{i}\cup
         \mathcal{S}^{+}_{i-1}$
are given by
\begin{displaymath}
\begin{aligned}
C'_{t_{i-1},\tilde{k},\tilde{h};t_i,k,h}
=
\frac{e^{\pi\sqrt{-1}\lambda_{i,k}}
      (t_i-t_{p+1})^{2\lambda_{i,k}-\rho_{1}-\mu_{q}}}
     {e^{\pi\sqrt{-1}\lambda_{i-1,\tilde{k}}}
      (t_{i-1}-t_{p+1})^{2\lambda_{i-1,\tilde{k}}-\rho_{1}-\mu_{q}}}
c_{t'_{i-1},\tilde{k},\tilde{h};t'_i,k,h}
\\
(1\leq \tilde{k}\leq r_{i-1}, \ 1\leq \tilde{h}\leq \ell_{i-1,\tilde{k}}, \
1\leq k\leq r_i, \ 1\leq h\leq \ell_{i,k}).
\end{aligned}
\end{displaymath}
\end{theorem}

\begin{theorem}
For $1\leq i\leq p$ the coefficients
$C^{\prime\pm}_{t_{p+1},\tilde{k},\tilde{h};t_i,k,h}$
in the connection formula
\begin{displaymath}
U'_{t_i,k,h}(x)
=
\sum_{\tilde{k}=1}^{q-1} \sum_{\tilde{h}=1}^{m_{\tilde{k}}}
C^{\prime\pm}_{t_{p+1},\tilde{k},\tilde{h};t_i,k,h}
U'_{t_{p+1},\tilde{k},\tilde{h}}(x)
+\mathrm{hol}\left(x-t_{p+1}\right)
\quad
(1\leq k\leq r_i, \ 1\leq h\leq \ell_{i,k})
\end{displaymath}
for $x\in\mathcal{S}^{\pm}_{i}$
are given by
\begin{displaymath}
\begin{aligned}
C^{\prime\pm}_{t_{p+1},\tilde{k},\tilde{h};t_i,k,h}
=
e^{\pm\pi\sqrt{-1}\lambda_{i,k}}
(t_i-t_{p+1})^{2\lambda_{i,k}-\rho_{1}-\mu_{q}}
\frac{\Gamma(\mu_{\tilde{k}}+1)
      \Gamma(\mu_{\tilde{k}}-\rho_{1}-\mu_{q})}
     {\Gamma(\mu_{\tilde{k}}-\rho_{1})
      \Gamma(\mu_{\tilde{k}}-\mu_{q})}
c_{\infty,\tilde{k},\tilde{h};t'_i,k,h}
\\
(1\leq \tilde{k}\leq q-1, \ 1\leq \tilde{h}\leq m_{\tilde{k}}, \
1\leq k\leq r_i, \ 1\leq h\leq \ell_{i,k}),
\end{aligned}
\end{displaymath}
where the double-signs correspond.
\end{theorem}

\subsection{Reducible case (ii)}
Finally, we consider the connection formulas for the system (\ref{sys:E0}).
We assume the conditions
(\ref{assumption:E2_1})--(\ref{assumption:E2_4})
and
(\ref{assumption:underlying}).
We use the notation of an index set
\begin{displaymath}
\Lambda_{l}
=
\{1,\ldots,n\}
\setminus\{m_1+\cdots+m_{l-1}+1,\ldots,m_{1}+\cdots+m_{l}\}
\end{displaymath}
for $l=q-1,q$.

\begin{proposition}
For $1\leq i\leq p$ we have
\begin{alignat}{2}
W^{\langle\mathcal{S}^{\prime\pm}_i;\eta_0\rangle}
_{t'_i,k,h}(\mu_{q-1},\mu_{q};\eta_0+\frac{1}{x-t_{p+1}})
&=
 \frac{\Gamma(\mu_{q-1}+1)\Gamma(-\mu_{q})}{\Gamma(\mu_{q-1}-\mu_{q}+1)}
 \sum_{\tilde{h}\in\Lambda_{q-1}}
 \gamma_{\tilde{h};t'_i,k,h}(-\mu_{q-1}-1)
 \begin{pmatrix}
 U''_{\infty,\mu_{q},\tilde{h}}(x)
 \\ 0_{m_{q-1}+m_{q}} \end{pmatrix} \notag
\\
 &&\llap{$(1\leq k\leq r_i, \ 1\leq h\leq \ell_{i,k}),$} \notag
\\
W^{\langle\mathcal{S}^{\prime\pm}_i;\eta_0\rangle}
_{t'_i,k,h}(\mu_{q},\mu_{q-1};\eta_0+\frac{1}{x-t_{p+1}})
&=
 \frac{\Gamma(\mu_{q}+1)\Gamma(-\mu_{q-1})}{\Gamma(\mu_{q}-\mu_{q-1}+1)}
 \sum_{\tilde{h}\in\Lambda_{q}}
 \gamma_{\tilde{h};t'_i,k,h}(-\mu_{q}-1)
 \begin{pmatrix}
 U''_{\infty,\mu_{q-1},\tilde{h}}(x)
 \\ 0_{m_{q-1}+m_{q}} \end{pmatrix} \notag
\\
 &&\llap{$(1\leq k\leq r_i, \ 1\leq h\leq \ell_{i,k})$} \notag
\end{alignat}
for $x\in\mathcal{S}^{\mp}_i$,
and
\begin{alignat}{2}
W^{\langle\mathcal{S}^{\prime\pm}_i;\eta_0\rangle}
_{\infty,k,h}(\mu_{q-1},\mu_{q};\eta_0+\frac{1}{x-t_{p+1}})
&=
 \frac{\Gamma(\mu_{q-1}+1)\Gamma(-\mu_{q})}{\Gamma(\mu_{q-1}-\mu_{q}+1)}
 \sum_{\tilde{h}\in\Lambda_{q-1}}
 \gamma^{\langle\mathcal{S}^{\prime\pm}_{i}\rangle}
       _{\tilde{h};\infty,k,h}(-\mu_{q-1}-1)
 \begin{pmatrix}
 U''_{\infty,\mu_{q},\tilde{h}}(x)
 \\ 0_{m_{q-1}+m_{q}} \end{pmatrix} \notag
 \\
 &&\llap{$(1\leq k\ne q-1\leq q, \ 1\leq h\leq m_{k}),$} \notag
\\
W^{\langle\mathcal{S}^{\prime\pm}_i;\eta_0\rangle}
_{\infty,k,h}(\mu_{q},\mu_{q-1};\eta_0+\frac{1}{x-t_{p+1}})
&=
 \frac{\Gamma(\mu_{q}+1)\Gamma(-\mu_{q-1})}{\Gamma(\mu_{q}-\mu_{q-1}+1)}
 \sum_{\tilde{h}\in\Lambda_{q}}
 \gamma^{\langle\mathcal{S}^{\prime\pm}_{i}\rangle}
       _{\tilde{h};\infty,k,h}(-\mu_{q}-1)
 \begin{pmatrix}
 U''_{\infty,\mu_{q-1},\tilde{h}}(x)
 \\ 0_{m_{q-1}+m_{q}} \end{pmatrix} \notag
 \\
 &&\llap{$(1\leq k\leq q-1, \ 1\leq h\leq m_{k})$} \notag
\end{alignat}
for $x\in\mathcal{S}^{\mp}_i$.
\end{proposition}

\begin{proposition}
We have
\begin{displaymath}
\begin{aligned}
V^{\langle\mathcal{S}^{\prime\pm}_i;t'_i \rangle}
_{t'_i,k,h}(\mu_{q-1},\mu_{q};\eta_0+\frac{1}{x-t_{p+1}})
=
e^{-\pi\sqrt{-1}\lambda_{i,k}}
(t_i-t_{p+1})^{\mu_{q-1}+\mu_{q}-2\lambda_{i,k}}
 \begin{pmatrix}
 U''_{t_i,k,h}(x)
 \\ 0_{m_{q-1}+m_{q}} \end{pmatrix}
\\
(1\leq k\leq r_i, \ 1\leq h\leq \ell_{i,k})
\end{aligned}
\end{displaymath}
for $x\in\mathcal{S}^{\mp}_i$.
\end{proposition}

\begin{proposition}
We have
\begin{displaymath}
V^{\langle\mathcal{S}^{\prime\pm}_i;\infty\rangle}
_{\infty,k,h}(\mu_{q-1},\mu_{q};\eta_0+\frac{1}{x-t_{p+1}})
=
 \begin{pmatrix}
   U''_{t_{p+1},k,h}(x)
 \\ 0_{m_{q-1}+m_{q}} \end{pmatrix}
\quad
(1\leq k\leq q-2, \ 1\leq h\leq m_{k})
\end{displaymath}
for $x\in\mathcal{S}^{\mp}_i$.
\end{proposition}

\begin{theorem}
For $1\leq i\leq p$ the coefficients
$C''_{\infty,\mu_{q-1},\tilde{h};t_i,k,h}$
and
$C''_{\infty,\mu_{q},\tilde{h};t_i,k,h}$
in the connection formula
\begin{displaymath}
U''_{t_i,k,h}(x)
=
 \sum_{\tilde{h}\in\Lambda_{q}}
C''_{\infty,\mu_{q-1},\tilde{h};t_i,k,h}
U''_{\infty,\mu_{q-1},\tilde{h}}(x)
+
 \sum_{\tilde{h}\in\Lambda_{q-1}}
C''_{\infty,\mu_{q},\tilde{h};t_i,k,h}
U''_{\infty,\mu_{q},\tilde{h}}(x)
\quad
(1\leq k\leq r_i, \ 1\leq h\leq \ell_{i,k})
\end{displaymath}
for $x\in\mathcal{S}^{\pm}_{i}$
are given by
\begin{alignat}{2}
C''_{\infty,\mu_{q-1},\tilde{h};t_i,k,h}
&=
e^{\pi\sqrt{-1}(\lambda_{i,k}-\mu_{q})}
(t_i-t_{p+1})^{2\lambda_{i,k}-\mu_{q-1}-\mu_{q}}
 \frac{\Gamma(\mu_{q-1}-\mu_{q})\Gamma(\lambda_{i,k}+1)}
      {\Gamma(\mu_{q-1}+1)\Gamma(\lambda_{i,k}-\mu_{q})}
 \gamma_{\tilde{h};t'_i,k,h}(-\mu_{q}-1) \notag
\\
 &&\llap{$(\tilde{h}\in\Lambda_{q}, \
           1\leq k\leq r_i, \ 1\leq h\leq \ell_{i,k}),$} \notag
\\
C''_{\infty,\mu_{q},\tilde{h};t_i,k,h}
&=
e^{\pi\sqrt{-1}(\lambda_{i,k}-\mu_{q-1})}
(t_i-t_{p+1})^{2\lambda_{i,k}-\mu_{q-1}-\mu_{q}}
 \frac{\Gamma(\mu_{q}-\mu_{q-1})\Gamma(\lambda_{i,k}+1)}
      {\Gamma(\mu_{q}+1)\Gamma(\lambda_{i,k}-\mu_{q-1})}
 \gamma_{\tilde{h};t'_i,k,h}(-\mu_{q-1}-1) \notag
\\
 &&\llap{$(\tilde{h}\in\Lambda_{q-1}, \
           1\leq k\leq r_i, \ 1\leq h\leq \ell_{i,k}).$} \notag
\end{alignat}
Besides,
the coefficients
$C^{\prime\prime\langle\mathcal{S}^{\pm}_{i}\rangle}
  _{\infty,\mu_{q-1},\tilde{h};t_{p+1},k,h}$
and
$C^{\prime\prime\langle\mathcal{S}^{\pm}_{i}\rangle}
  _{\infty,\mu_{q},\tilde{h};t_{p+1},k,h}$
in the connection formula
\begin{displaymath}
\begin{aligned}
U''_{t_{p+1},k,h}(x)
=
 \sum_{\tilde{h}\in\Lambda_{q}}
C^{\prime\prime\langle\mathcal{S}^{\pm}_{i}\rangle}
 _{\infty,\mu_{q-1},\tilde{h};t_{p+1},k,h}
U''_{\infty,\mu_{q-1},\tilde{h}}(x)
+
 \sum_{\tilde{h}\in\Lambda_{q-1}}
C^{\prime\prime\langle\mathcal{S}^{\pm}_{i}\rangle}
 _{\infty,\mu_{q},\tilde{h};t_{p+1},k,h}
U''_{\infty,\mu_{q},\tilde{h}}(x)
\\
(1\leq k\leq q-2, \ 1\leq h\leq m_{k})
\end{aligned}
\end{displaymath}
for $x\in\mathcal{S}^{\pm}_{i}$
are given by
\begin{alignat}{2}
C^{\prime\prime\langle\mathcal{S}^{\pm}_{i}\rangle}
 _{\infty,\mu_{q-1},\tilde{h};t_{p+1},k,h}
&=
\frac{\Gamma(\mu_{q-1}-\mu_{q})\Gamma(\mu_{q-1}+\mu_{q}-\mu_{k}+1)}
      {\Gamma(\mu_{q-1}+1)\Gamma(\mu_{q-1}-\mu_{k}+1)}
\gamma^{\langle\mathcal{S}^{\prime\mp}_{i}\rangle}
      _{\tilde{h};\infty,k,h}(-\mu_{q}-1) \notag
\\
 &&\llap{$(\tilde{h}\in\Lambda_{q}, \
           1\leq k\leq q-2, \ 1\leq h\leq m_{k}),$} \notag
\\
C^{\prime\prime\langle\mathcal{S}^{\pm}_{i}\rangle}
 _{\infty,\mu_{q},\tilde{h};t_{p+1},k,h}
&=
\frac{\Gamma(\mu_{q}-\mu_{q-1})\Gamma(\mu_{q-1}+\mu_{q}-\mu_{k}+1)}
      {\Gamma(\mu_{q}+1)\Gamma(\mu_{q}-\mu_{k}+1)}
\gamma^{\langle\mathcal{S}^{\prime\mp}_{i}\rangle}
      _{\tilde{h};\infty,k,h}(-\mu_{q-1}-1) \notag
\\
 &&\llap{$(\tilde{h}\in\Lambda_{q-1}, \
           1\leq k\leq q-2, \ 1\leq h\leq m_{k}).$} \notag
\end{alignat}
\end{theorem}

\begin{theorem}
For $1\leq i\leq p-1$ the coefficients
$C''_{t_{i+1},\tilde{k},\tilde{h};t_i,k,h}$
in the connection formula
\begin{displaymath}
U''_{t_i,k,h}(x)
=
\sum_{\tilde{k}=1}^{r_{i+1}} \sum_{\tilde{h}=1}^{\ell_{i+1,\tilde{k}}}
C''_{t_{i+1},\tilde{k},\tilde{h};t_i,k,h}
U''_{t_{i+1},\tilde{k},\tilde{h}}(x)
+\mathrm{hol}(x-t_{i+1})
\quad
(1\leq k\leq r_i, \ 1\leq h\leq \ell_{i,k})
\end{displaymath}
for $x\in\mathcal{S}^{+}_{i}\cup
         \mathcal{S}^{-}_{i+1}$
are given by
\begin{displaymath}
\begin{aligned}
C''_{t_{i+1},\tilde{k},\tilde{h};t_i,k,h}
=
\frac{e^{\pi\sqrt{-1}\lambda_{i,k}}
      (t_i-t_{p+1})^{2\lambda_{i,k}-\mu_{q-1}-\mu_{q}}}
     {e^{\pi\sqrt{-1}\lambda_{i+1,\tilde{k}}}
      (t_{i+1}-t_{p+1})^{2\lambda_{i+1,\tilde{k}}-\mu_{q-1}-\mu_{q}}}
c_{t'_{i+1},\tilde{k},\tilde{h};t'_i,k,h}
\\
(1\leq \tilde{k}\leq r_{i+1}, \ 1\leq \tilde{h}\leq \ell_{i+1,\tilde{k}}, \
 1\leq k\leq r_i, \ 1\leq h\leq \ell_{i,k}).
\end{aligned}
\end{displaymath}
Besides,
for $2\leq i\leq p$ the coefficients
$C''_{t_{i-1},\tilde{k},\tilde{h};t_i,k,h}$
in the connection formula
\begin{displaymath}
U''_{t_i,k,h}(x)
=
\sum_{\tilde{k}=1}^{r_{i-1}} \sum_{\tilde{h}=1}^{\ell_{i-1,\tilde{k}}}
C''_{t_{i-1},\tilde{k},\tilde{h};t_i,k,h}
U''_{t_{i-1},\tilde{k},\tilde{h}}(x)
+\mathrm{hol}(x-t_{i-1})
\quad
(1\leq k\leq r_i, \ 1\leq h\leq \ell_{i,k})
\end{displaymath}
for $x\in\mathcal{S}^{-}_{i}\cup
         \mathcal{S}^{+}_{i-1}$
are given by
\begin{displaymath}
\begin{aligned}
C''_{t_{i-1},\tilde{k},\tilde{h};t_i,k,h}
=
\frac{e^{\pi\sqrt{-1}\lambda_{i,k}}
      (t_i-t_{p+1})^{2\lambda_{i,k}-\mu_{q-1}-\mu_{q}}}
     {e^{\pi\sqrt{-1}\lambda_{i-1,\tilde{k}}}
      (t_{i-1}-t_{p+1})^{2\lambda_{i-1,\tilde{k}}-\mu_{q-1}-\mu_{q}}}
c_{t'_{i-1},\tilde{k},\tilde{h};t'_i,k,h}
\\
(1\leq \tilde{k}\leq r_{i-1}, \ 1\leq \tilde{h}\leq \ell_{i-1,\tilde{k}}, \
 1\leq k\leq r_i, \ 1\leq h\leq \ell_{i,k}).
\end{aligned}
\end{displaymath}
\end{theorem}

\begin{theorem}
For $1\leq i\leq p$ the coefficients
$C^{\prime\prime\pm}_{t_{p+1},\tilde{k},\tilde{h};t_i,k,h}$
in the connection formula
\begin{displaymath}
U''_{t_i,k,h}(x)
=
\sum_{\tilde{k}=1}^{q-2} \sum_{\tilde{h}=1}^{m_{\tilde{k}}}
C^{\prime\prime\pm}_{t_{p+1},\tilde{k},\tilde{h};t_i,k,h}
U''_{t_{p+1},\tilde{k},\tilde{h}}(x)
+\mathrm{hol}\left(x-t_{p+1}\right)
\quad
(1\leq k\leq r_i, \ 1\leq h\leq \ell_{i,k})
\end{displaymath}
for $x\in\mathcal{S}^{\pm}_{i}$
are given by
\begin{displaymath}
\begin{aligned}
C^{\prime\prime\pm}_{t_{p+1},\tilde{k},\tilde{h};t_i,k,h}
=
e^{\pm\pi\sqrt{-1}\lambda_{i,k}}
(t_i-t_{p+1})^{2\lambda_{i,k}-\mu_{q-1}-\mu_{q}}
\frac{\Gamma(\mu_{\tilde{k}}+1)
      \Gamma(\mu_{\tilde{k}}-\mu_{q-1}-\mu_{q})}
     {\Gamma(\mu_{\tilde{k}}-\mu_{q-1})
      \Gamma(\mu_{\tilde{k}}-\mu_{q})}
c_{\infty,\tilde{k},\tilde{h};t'_i,k,h}
\\
(1\leq \tilde{k}\leq q-2, \ 1\leq \tilde{h}\leq m_{\tilde{k}}, \
 1\leq k\leq r_i, \ 1\leq h\leq \ell_{i,k}),
\end{aligned}
\end{displaymath}
where the double-signs correspond.
\end{theorem}


\end{document}